\documentclass[a4paper,11pt]{article}
\usepackage{amsmath,amsthm,amssymb,enumitem,xcolor}
\usepackage[hang]{footmisc}

\usepackage[nosort,nocompress,noadjust]{cite}

\usepackage[bookmarks=false,hyperfootnotes=false,colorlinks,linktoc=all,
    linkcolor={red!60!black},
    citecolor={blue!50!black},
    urlcolor={blue!80!black}]{hyperref}

\renewcommand{\eqref}[1]{\hyperref[#1]{(\ref{#1})}}

\newlist{enumlist}{enumerate}{2}
\setlist[enumlist,1]{labelindent=0cm,label=(\roman*),ref=(\roman*),labelwidth=4.5ex,labelsep=1ex,leftmargin=5.5ex,align=right,topsep=0.5ex,itemsep=1ex,parsep=1ex}
\setlist[enumlist,2]{labelindent=0cm,label=\theenumlisti.\arabic*.,ref=\arabic*,labelwidth=5ex,labelsep=0.5ex,leftmargin=5.5ex,align=left,topsep=0.5ex,itemsep=1ex,parsep=1ex}

\newlist{itemlist}{itemize}{1}
\setlist[itemlist]{labelindent=0cm,label=$\bullet$,labelwidth=2.5ex,labelsep=0.5ex,leftmargin=3ex,align=left,topsep=0.5ex,itemsep=1ex,parsep=1ex}

\numberwithin{equation}{section}

{\theoremstyle{definition}\newtheorem{definitiona}{Definition}[section]
\newtheorem{remarka}[definitiona]{Remark}
\newtheorem{examplea}[definitiona]{Example}}

\newtheorem{propositiona}[definitiona]{Proposition}
\newtheorem{lemmaa}[definitiona]{Lemma}
\newtheorem{theorema}[definitiona]{Theorem}
\newtheorem{corollarya}[definitiona]{Corollary}

\newtheorem{letterthma}{Theorem}
\renewcommand{\theletterthma}{\Alph{letterthma}}

\newenvironment{definition}[1][]{\begin{definitiona}[#1]\setlist[enumlist,1]{label=(\roman*),ref=\thedefinitiona(\roman*)}}{\end{definitiona}}
\newenvironment{remark}[1][]{\begin{remarka}[#1]\setlist[enumlist,1]{label=(\roman*),ref=\theremarka(\roman*)}}{\end{remarka}}
\newenvironment{example}[1][]{\begin{examplea}[#1]\setlist[enumlist,1]{label=(\roman*),ref=\theexamplea(\roman*)}}{\end{examplea}}
\newenvironment{proposition}[1][]{\begin{propositiona}[#1]\setlist[enumlist,1]{label=(\roman*),ref=\thepropositiona(\roman*)}}{\end{propositiona}}
\newenvironment{lemma}[1][]{\begin{lemmaa}[#1]\setlist[enumlist,1]{label=(\roman*),ref=\thelemmaa(\roman*)}}{\end{lemmaa}}
\newenvironment{theorem}[1][]{\begin{theorema}[#1]\setlist[enumlist,1]{label=(\roman*),ref=\thetheorema(\roman*)}}{\end{theorema}}
\newenvironment{corollary}[1][]{\begin{corollarya}[#1]\setlist[enumlist,1]{label=(\roman*),ref=\thecorollarya(\roman*)}}{\end{corollarya}}
\newenvironment{letterthm}[1][]{\begin{letterthma}[#1]\setlist[enumlist,1]{label=(\roman*),ref=\theletterthma.(\roman*)}}{\end{letterthma}}

\newcommand{\C}{\mathbb{C}}
\newcommand{\cC}{\mathcal{C}}
\newcommand{\eps}{\varepsilon}
\newcommand{\al}{\alpha}
\newcommand{\be}{\beta}
\newcommand{\ot}{\otimes}

\newcommand{\Z}{\mathbb{Z}}
\newcommand{\vphi}{\varphi}

\newcommand{\id}{\mathord{\text{\rm id}}}
\newcommand{\om}{\omega}
\newcommand{\N}{\mathbb{N}}
\newcommand{\ovt}{\mathbin{\overline{\otimes}}}
\newcommand{\Tr}{\operatorname{Tr}}
\newcommand{\Om}{\Omega}
\newcommand{\cD}{\mathcal{D}}
\newcommand{\si}{\sigma}
\newcommand{\R}{\mathbb{R}}
\newcommand{\F}{\mathbb{F}}
\newcommand{\cH}{\mathcal{H}}
\newcommand{\cZ}{\mathcal{Z}}
\newcommand{\Ad}{\operatorname{Ad}}
\newcommand{\cG}{\mathcal{G}}

\newcommand{\cF}{\mathcal{F}}
\newcommand{\T}{\mathbb{T}}
\newcommand{\actson}{\curvearrowright}
\newcommand{\cA}{\mathcal{A}}
\newcommand{\cB}{\mathcal{B}}

\newcommand{\cU}{\mathcal{U}}
\newcommand{\Ker}{\operatorname{Ker}}
\newcommand{\cM}{\mathcal{M}}
\newcommand{\lspan}{\operatorname{span}}
\newcommand{\cN}{\mathcal{N}}
\newcommand{\cR}{\mathcal{R}}

\newcommand{\cE}{\mathcal{E}}
\newcommand{\Hom}{\operatorname{Hom}}
\newcommand{\Aut}{\operatorname{Aut}}

\newcommand{\dpr}{^{\prime\prime}}

\newcommand{\Stab}{\operatorname{Stab}}
\newcommand{\Q}{\mathbb{Q}}
\newcommand{\GL}{\operatorname{GL}}

\newcommand{\End}{\operatorname{End}}

\newcommand{\op}{^\text{\rm op}}
\newcommand{\fc}{_\text{\rm fc}}
\newcommand{\Gtil}{\widetilde{G}}
\newcommand{\cGtil}{\widetilde{\mathcal{G}}}

\newcommand{\Omtil}{\widetilde{\Omega}}
\newcommand{\Oms}{\Omega_{\text{\rm s}}}
\newcommand{\PU}{\operatorname{PU}}
\newcommand{\Bchar}{\operatorname{Bchar}}
\newcommand{\Achar}{\operatorname{Achar}}

\newcommand{\bim}[3]{\mathord{\raisebox{-0.4ex}[0ex][0ex]{\scriptsize $#1$}{#2}\hspace{-0.2ex}\raisebox{-0.4ex}[0ex][0ex]{\scriptsize $#3$}}}

\begin{document}

\begin{center}
{\boldmath\LARGE\bf W$^*$-superrigidity for cocycle twisted\vspace{0.5ex}\\
group von Neumann algebras}

\vspace{1ex}

{\sc by Milan Donvil\footnote{KU Leuven, Department of Mathematics, Leuven (Belgium), milan.donvil@kuleuven.be\\ Supported by PhD grant 1162024N funded by the Research Foundation Flanders (FWO)} and Stefaan Vaes\footnote{KU~Leuven, Department of Mathematics, Leuven (Belgium), stefaan.vaes@kuleuven.be\\ Supported by FWO research project G090420N of the Research Foundation Flanders and by Methusalem grant METH/21/03 –- long term structural funding of the Flemish Government.}}
\end{center}

\begin{abstract}\noindent
We construct countable groups $G$ with the following new degree of W$^*$-superrigidity: if $L(G)$ is virtually isomorphic, in the sense of admitting a bifinite bimodule, with any other group von Neumann algebra $L(\Lambda)$, then the groups $G$ and $\Lambda$ must be virtually isomorphic. Moreover, we allow both group von Neumann algebras to be twisted by an arbitrary $2$-cocycle. We also give examples of II$_1$ factors $N$ that are indecomposable in every sense: they are not virtually isomorphic to any cocycle twisted groupoid von Neumann algebra.
\end{abstract}

\section{Introduction and main results}

Every countable group $G$ gives rise to the \emph{group von Neumann algebra} $L(G)$ generated by the left regular representation. When $G$ has infinite conjugacy classes (icc), $L(G)$ is a II$_1$ factor. In the passage from $G$ to $L(G)$, typically most of the structure of $G$ is wiped out. Most notably, by Connes' theorem \cite{Con76}, all group von Neumann algebras of amenable icc groups are isomorphic.

In the nonamenable world, a more rigid behavior appears and in the most extreme cases, it is possible to entirely recover $G$ from its ambient II$_1$ factor $L(G)$. This phenomenon is called \emph{W$^*$-superrigidity} and the first families of W$^*$-superrigid groups were discovered in \cite{IPV10}, by using Popa's deformation/rigidity theory. Since then, several new classes of W$^*$-superrigid groups were found. In particular, \cite{CIOS21} provided the first W$^*$-superrigid groups with Kazhdan's property~(T).

We obtain in this paper a new degree of W$^*$-superrigidity for group von Neumann algebras, replacing isomorphism by \emph{virtual isomorphism}. Recall that II$_1$ factors $M$ and $N$ are said to be virtually isomorphic if there exists a bifinite $M$-$N$-bimodule or equivalently, if a corner $pMp$ can be realized as a subfactor of $N$ with finite Jones index. Two countable groups $G$ and $\Lambda$ are called virtually isomorphic if they are isomorphic up to taking finite index subgroups and quotients by their finite normal subgroups. In Theorem \ref{thm.A}, we provide the first family of groups $G$ with the following W$^*$-superrigidity property: if $L(G)$ is virtually isomorphic with any other group von Neumann algebra $L(\Lambda)$, then the groups $G$ and $\Lambda$ must be virtually isomorphic.

Moreover, we allow throughout that our group von Neumann algebras are twisted by a \emph{$2$-cocycle}. Even when we restrict ourselves to isomorphisms, this is the first W$^*$-superrigidity theorem for cocycle twisted group von Neumann algebras.

As in \cite[Definition 3.1]{PV21}, we consider the class $\cC$ of nonamenable groups $\Gamma$ that are weakly amenable and biexact and such that for every nontrivial element $g \in \Gamma \setminus \{e\}$, the centralizer $C_\Gamma(g)$ is amenable. This class contains all free groups $\F_n$, $2 \leq n \leq \infty$, all free products of amenable groups $\Gamma_1 \ast \Gamma_2$ with $|\Gamma_1| \geq 2$ and $|\Gamma_2| \geq 3$, and all torsion-free hyperbolic groups.

We know from \cite[Theorem 8.1]{BV12} that for every countable group $\Gamma$ in the class $\cC$, the left-right wreath product $G = (\Z/2\Z)^{(\Gamma)} \rtimes (\Gamma \times \Gamma)$ is W$^*$-superrigid: if $L(G) \cong L(\Lambda)$ with $\Lambda$ an arbitrary countable group, then $G \cong \Lambda$. We prove that these same groups $G$ are W$^*$-superrigid up to virtual isomorphisms, and remain so when we allow arbitrary cocycle twists. Thus the following is our main result.

Recall that a $2$-cocycle $\mu \in Z^2(G,\T)$ on a countable group is said to be \emph{of finite type} if $\mu$ can be realized as the $2$-cocycle of a finite dimensional projective representation of $G$. We write $\mu \sim \om$ when $\mu$ and $\om$ are cohomologous $2$-cocycles.

\begin{letterthm}\label{thm.A}
Let $\Gamma$ be a group in $\cC$ and $G = (\Z/2\Z)^{(\Gamma)} \rtimes (\Gamma \times \Gamma)$ the left-right wreath product.
Let $\Lambda$ be any countable group and let $\mu \in Z^2(G,\T)$ and $\om \in Z^2(\Lambda,\T)$ be any $2$-cocycles.
\begin{enumlist}
\item We have $L_\mu(G) \cong L_\om(\Lambda)$ if and only if there exists an isomorphism $\delta : \Lambda \to G$ such that $\mu \circ \delta \sim \om$.
\item There exists a nonzero bifinite $L_\mu(G)$-$L_\om(\Lambda)$-bimodule if and only if $(G,\mu)$ and $(\Lambda,\om)$ are virtually isomorphic: there exist a finite index subgroup $\Lambda_0 < \Lambda$ and a group homomorphism $\delta : \Lambda_0 \to G$ such that the kernel $\Ker \delta$ is finite, the image $\delta(\Lambda_0) < G$ has finite index and $(\mu \circ \delta) \, \overline{\om}$ is of finite type on $\Lambda_0$.
\end{enumlist}
\end{letterthm}

Since wreath products have plenty of $2$-cocycles (see Proposition \ref{prop.cohom-left-right-wreath}), Theorem \ref{thm.A} provides in particular the first examples of cocycle twisted group von Neumann algebras that are not isomorphic to \emph{any} non-twisted group von Neumann algebra. In Section \ref{sec.examples}, we actually use Theorem \ref{thm.A} to provide examples of $2$-cocycles $\mu \in Z^2(G,\T)$ that are untwistable in even more remarkable ways. This includes examples where $L_\mu(G)$ is canonically isomorphic to a corner $p L(\Gtil) p$ of a group von Neumann algebra of a (non icc) group $\Gtil$, even though there is no \emph{faithful} bifinite $L_\mu(G)$-$L(\Lambda)$-bimodule with \emph{any} group von Neumann algebra. We also give examples where $L_\mu(G) \cong L_\mu(G)\op$ and where at the same time there is no nonzero bifinite $L_\mu(G)$-$L(\Lambda)$-bimodule with any group von Neumann algebra, as well as examples where $L_\mu(G)$ is not virtually isomorphic to its opposite $L_\mu(G)\op$.

Starting with \cite{Voi95}, several classes of II$_1$ factors $N$ were shown to have no Cartan subalgebra. These $N$ cannot be decomposed as the von Neumann algebra $L_\om(\cR)$ of a countable probability measure preserving (pmp) equivalence relation $\cR$, possibly twisted by a $2$-cocycle $\om$.

In \cite{Con75,Jon79}, it was proven that certain II$_1$ factors $N$ cannot be decomposed as a group von Neumann algebra $L(\Lambda)$. Only much later in \cite[Corollary G]{Ioa10}, it was shown that there are II$_1$ factors $N$ that are not isomorphic to any cocycle twisted group von Neumann algebra $L_\om(\Lambda)$. Although it was proven in \cite[Corollary G]{Ioa10} that also the nonzero corners $pNp$ cannot be written as a twisted group von Neumann algebra, it remained an open problem to construct II$_1$ factors $N$ such that none of its amplifications $N^t$ are isomorphic to a twisted group von Neumann algebra.

In our second main result, we solve this problem by constructing II$_1$ factors that satisfy all of the above indecomposability results at once, and even up to virtual isomorphism. The twisted von Neumann algebras $L_\om(\cR)$ and $L_\om(\Lambda)$ of a pmp equivalence relation $\cR$, resp.\ a countable group $\Lambda$, are both examples of twisted von Neumann algebras $L_\om(\cG)$ of a discrete pmp groupoid (see Section \ref{sec.twisted-groupoid-vNalg}). We provide the following class of II$_1$ factors $N$ that do not admit a nonzero bifinite $N$-$L_\om(\cG)$-bimodule for any discrete pmp groupoid $\cG$ and any $2$-cocycle $\om$. In particular, $N$ has no Cartan subalgebra and none of the amplifications $N^t$ is isomorphic to a cocycle twisted group von Neumann algebra.

We use left-right Bernoulli crossed products with a discrete abelian base $A_0 = \ell^\infty(I)$. We say that a faithful normal trace $\tau_0$ on $A_0$ is uniform if all minimal projections have the same trace, which forces $I$ to be finite.

\begin{letterthm}\label{thm.B}
Let $\Gamma$, $\Lambda$ be groups in $\cC$ and let $(A_0,\tau_0)$ be a discrete abelian tracial von Neumann algebra. Denote by $M = (A_0,\tau_0)^\Gamma \rtimes (\Gamma \times \Gamma)$ the left-right Bernoulli crossed product. Write $N = M \ast L(\Lambda)$.
\begin{enumlist}
\item If $\tau_0$ is not uniform, there is no nonzero bifinite $N$-$L_\om(\cG)$-bimodule for any discrete pmp groupoid $\cG$ and any $2$-cocycle $\om \in Z^2(\cG,\T)$. In particular, $N^t \not\cong L_\om(\cG)$ for all $t > 0$.
\item If $\tau_0$ is uniform, then $N \cong L(\cG)$, where $\cG$ is the group $((\Z/n\Z)^{(\Gamma)}\rtimes (\Gamma \times \Gamma)) \ast \Lambda$ and $n = \dim A_0$.
\end{enumlist}
\end{letterthm}

Given any left-right wreath product group $G = \Lambda_0^{(\Gamma)} \rtimes (\Gamma \times \Gamma)$ and any $2$-cocycle $\om_0 \in Z^2(\Lambda_0,\T)$, we consider the canonical diagonal $2$-cocycle $\om_0^\Gamma \in Z^2(G,\T)$ given by
\begin{equation}\label{eq.diagonal-2-cocycle}
\om_0^\Gamma((a,(g,h)),(a',(g',h'))) = \prod_{k \in \Gamma} \om_0(a_k,a'_k) \; .
\end{equation}

We finally prove the following result, giving rise to another cocycle $W^*$-superrigidity theorem for left-right wreath products with more general base groups than $\Z/2\Z$.

\begin{letterthm}\label{thm.C}
Let $\Gamma$ be a group in $\cC$ and $(A_0,\tau_0)$ any nontrivial amenable tracial von Neumann algebra. Denote by $M = (A_0,\tau_0)^\Gamma \rtimes (\Gamma \times \Gamma)$ the left-right Bernoulli crossed product. Let $\Lambda$ be any countable group and $\om \in Z^2(\Lambda,\T)$ any $2$-cocycle.

We have $M \cong L_\om(\Lambda)$ if and only if there exists a countable group $\Lambda_0$, an isomorphism $\delta : \Lambda \to \Lambda_0^{(\Gamma)} \rtimes (\Gamma \times \Gamma)$ and a $2$-cocycle $\om_0 \in Z^2(\Lambda_0,\T)$ such that $\om_0^\Gamma \circ \delta \sim \om$ and $(A_0,\tau_0) \cong (L_{\om_0}(\Lambda_0),\tau_0)$.
\end{letterthm}

In Definition \ref{def.cocycle-Wstar-superrigid}, we say that a pair $(G,\mu)$ of a countable group $G$ and a $2$-cocycle $\mu \in Z^2(G,\T)$ is \emph{cocycle W$^*$-superrigid} if the first statement of Theorem \ref{thm.A} holds: whenever $L_\mu(G)$ is isomorphic to some $L_\om(\Lambda)$, with $\Lambda$ an arbitrary countable group and $\om \in Z^2(\Lambda,\T)$ an arbitrary $2$-cocycle, there must exist a group isomorphism $\delta : \Lambda \to G$ such that $\mu \circ \delta \sim \om$. As we explain in Corollary \ref{cor.further-superrigidity}, Theorem \ref{thm.C} immediately implies that $(G,\mu)$ is cocycle W$^*$-superrigid whenever $G = C^{(\Gamma)} \times (\Gamma \times \Gamma)$ and $\mu = \mu_0^\Gamma$, where $\Gamma$ belongs to $\cC$ and the base $(C,\mu_0)$ satisfies one of the following conditions: $|C| = p$ with $p$ prime, $|C| = pq$ with $p,q$ distinct primes or finally, $|C| = p^2$ with $p$ prime and $\mu_0\not\sim 1$.

All W$^*$-superrigidity theorems in the literature use the following approach of \cite{Ioa10,IPV10}: whenever a II$_1$ factor $M$ is decomposed as a group von Neumann algebra $M \cong L(\Lambda)$, the \emph{comultiplication} $u_g \mapsto u_g \ot u_g$ on $L(\Lambda)$ induces an embedding $\Delta : M \to M \ovt M$. If $M$ is sufficiently rigid, as for instance Bernoulli crossed products tend to be as discovered in \cite{Pop03}, one may hope to completely describe $\Delta$ and find back $\Lambda$. When $M \cong L_\om(\Lambda)$ is decomposed as a twisted group von Neumann algebra, we still have the \emph{triple comultiplication}
$$\Delta_3 : M \to M \ovt M\op \ovt M : u_g \mapsto u_g \ot \overline{u_g} \ot u_g \quad\text{for all $g \in \Lambda$,}$$
introduced in \cite{Ioa10} and follow the same approach.

To allow for greater flexibility, as needed when we consider virtual isomorphisms, we introduce the abstract concept of a \emph{coarse embedding}
$$\psi : M \to (M_1 \ovt \cdots \ovt M_k)^t$$
of a II$_1$ factor $M$ into an amplified tensor product of II$_1$ factors, see Definition \ref{def.coarse-embedding}. This is the key new notion of this paper.

On the one hand, this notion of coarse embedding is strict enough to allow us to classify all possible coarse embeddings when $M$ and $M_i$ are left-right Bernoulli crossed products with groups in $\cC$, possibly twisted with $2$-cocycles.

On the other hand, this notion of coarse embedding is flexible enough to get the following: whenever $\bim{M}{K}{L_\om(\Lambda)}$ is a bifinite $M$-$L_\om(\Lambda)$-bimodule between a II$_1$ factor $M$ and a twisted group von Neumann algebra $L_\om(\Lambda)$, the triple comultiplication on $L_\om(\Lambda)$ can be transferred to a coarse embedding $M \to (M \ovt M\op \ovt M)^t$.

When $k=1$, by definition any embedding $M \to M_1^t$ is called coarse. One of the key results of \cite{PV21} is \cite[Theorem 3.2]{PV21} describing all embeddings $M \to M_1^t$ when $M$ and $M_1$ are left-right Bernoulli crossed products with groups in $\cC$. One of our key results is Theorem \ref{thm.coarse-embedding-Bernoulli-like} giving a precise description of all possible coarse embeddings when $M$ and $M_i$ are ``Bernoulli like left-right crossed products''. We make this notion more precise in Section \ref{sec.bernoulli-like}, but roughly speaking, it covers at the same time crossed products $(A_0,\tau_0)^\Gamma \rtimes (\Gamma \times \Gamma)$ with arbitrary amenable base algebras $(A_0,\tau_0)$ and twisted group von Neumann algebras $L_\mu(G)$ when $G = G_0^{(\Gamma)} \rtimes (\Gamma \times \Gamma)$ is a left-right wreath product and $\mu \in Z^2(G,\T)$ is any $2$-cocycle.

In the specific case where the base group $G_0$ is as small as possible, $G_0 = \Z/2\Z$, we make this description of all possible coarse embeddings even more precise in the complete classification Theorem \ref{thm.coarse-embedding-cocycle-twisted-wreath-product}, which is the basis to prove Theorem \ref{thm.A}.

Also in the absence of $2$-cocycles, but now with arbitrary amenable base algebras $(A_0,\tau_0)$, we make the description of all possible coarse embeddings more precise in Theorem \ref{thm.coarse-embedding-generalized-Bernoulli}, which is the key element to prove Theorem \ref{thm.C}.

To prove Theorem \ref{thm.B}, assume that $\tau_0$ is not uniform and that $K$ is a bifinite $N$-$L_\om(\cG)$-bimodule. We first invoke the main result of \cite{Ioa12} to prove that because of the free product form $N = M \ast L(\Lambda)$, the unit space of $\cG$ must have atoms. Considering the isotropy group $\cG_1$ of one these atoms, we find a bifinite $N$-$L_\om(\cG_1)$-bimodule, which induces a coarse embedding $\psi : N \to (N \ovt N\op \ovt N)^t$. Because of the special form of $M$ and the fact that $\Lambda \in \cC$, we deduce that after a unitary conjugacy, $\psi$ restricts to a coarse embedding $M \to (M \ovt M\op \ovt M)^t$. Our Theorem \ref{thm.coarse-embedding-generalized-Bernoulli} and the assumption that $\tau_0$ is not uniform will imply that there simply is no coarse embedding of $M$ into $(M \ovt M\op \ovt M)^t$.

\section{Preliminaries}

\subsection{Hilbert bimodule terminology}

Given a von Neumann algebra $A$, we denote by $A\op$ the ``same'' von Neumann algebra with the opposite multiplication. So, the elements of $A\op$ are $a\op$ for all $a \in A$ and $a\op b\op = (ba)\op$. We also write $\overline{a} = (a^*)\op \in A\op$ for all $a \in A$. The map $a \mapsto \overline{a}$ is multiplicative, but antilinear.

When $A$ and $B$ are von Neumann algebras, we say that $\bim{A}{H}{B}$ is an \emph{$A$-$B$-bimodule} if $H$ is a Hilbert space equipped with unital normal $*$-homomorphisms $\lambda : A \to B(H)$ and $\rho : B\op \to B(H)$ whose ranges commute. We write $a \cdot \xi \cdot b = \lambda(a) \rho(b\op) \xi$. We say that $\bim{A}{H}{B}$ is \emph{faithful} if both $\lambda$ and $\rho$ are faithful. Of course, when $A$ and $B$ are factors, a nonzero $A$-$B$-bimodule is automatically faithful.

We say that a right $B$-module $H_B$ is \emph{finitely generated} if there exists a finite subset $\cF \subset H$ such that $\cF \cdot B$ is total in $H$. If $(B,\tau)$ is a tracial von Neumann algebra, the finitely generated right $B$-modules are, up to unitary isomorphism, all of the form $p(\C^n \ot L^2(B))_B$, where $p \in M_n(\C) \ot B$ is a projection.

We say that an $A$-$B$-bimodule $\bim{A}{H}{B}$ is \emph{bifinite} if both $_A H$ is a finitely generated left $A$-module and $H_B$ is a finitely generated right $B$-module. When $A$ and $B$ are II$_1$ factors, the bifinite $A$-$B$-bimodules are precisely the bimodules of the form $\bim{\vphi(A)}{p(\C^n \ot L^2(B))}{B}$, where $p \in M_n(\C) \ot B$ is a projection and $\vphi : A \to p(M_n(\C) \ot B)p$ is a finite index embedding.

\subsection{Intertwining-by-bimodules and relative amenability}\label{sec.intertwining}

Recall from \cite[Section 2]{Pop03} Popa's concept of \emph{intertwining-by-bimodules}: given a tracial von Neumann algebra $(M,\tau)$, a projection $p \in M_n(\C) \ot M$ and von Neumann subalgebras $A \subset p(M_n(\C) \ot M)p$ and $B \subset M$, one writes $A \prec_M B$ if $p(\C^n \ot L^2(M))$ admits a nonzero $A$-$B$-subbimodule that is finitely generated as a right Hilbert $B$-module. Equivalently, there exist $k \in \N$, a projection $q \in M_k(\C) \ot B$, a nonzero partial isometry $V \in p (M_{n,k}(\C) \ot M)q$ and a unital normal $*$-homomorphism $\theta : A \to q (M_k(\C) \ot B)q$ satisfying $a V = V \theta(a)$ for all $a \in A$.

Note that when $A \prec_M B$, one may choose $(\theta,V,q,k)$ in such a way that $(\id \ot E_B)(V^* V)$ is an invertible element in $q (M_k(\C) \ot B)q$.

We also use the notation $A \prec^f_M B$ if for every nonzero projection $p_1 \in A' \cap p(M_n(\C) \ot M)p$, we have that $Ap_1 \prec_M B$.

Let $\cG \subset \cU(A)$ be a subgroup such that $A = \cG\dpr$. Recall that $A\not\prec_M B$ if and only if there exists a net $u_i \in \cG$ such that $\|E_B(x^* u_i y)\|_2 \to 0$ for all $x,y \in \C^n \ot M$.

Recall from \cite[Section 2.2]{OP07} the concept of \emph{relative amenability}: one says that $A$ is amenable relative to $B$ inside $M$ if the corner $p (M_n(\C) \ot \langle M,e_B \rangle) p$ of Jones's basic construction admits a (typically non normal) state $\Om$ that is $A$-central and whose restriction to $p (M_n(\C) \ot M)p$ is faithful and normal. One says that $A$ is \emph{strongly nonamenable relative to $B$} inside $M$ if for every nonzero projection $p_1 \in A' \cap p (M_n(\C) \ot M)p$, $A p_1$ is not amenable relative to $B$ inside $M$.

\begin{remark}\label{rem.everything-in-terms-of-basic-construction}
As above, let $(M,\tau)$ be a tracial von Neumann algebra, $p \in M_n(\C) \ot M$ a projection and $A \subset p(M_n(\C) \ot M)p$, $B \subset M$ von Neumann subalgebras. Using Jones's basic construction $\langle M,e_B \rangle$, one can put Popa's intertwining relation and relative amenability on the same footing by making the following observations.
\begin{enumlist}
\item We have $A \prec_M B$ if and only if $p (M_n(\C) \ot \langle M,e_B \rangle) p$ admits a nonzero normal $A$-central state.
\item We have $A \prec^f_M B$ if and only if $p (M_n(\C) \ot \langle M,e_B \rangle) p$ admits a normal $A$-central state whose restriction to $p(M_n(\C) \ot M)p$ is faithful.
\item We have that $A$ is amenable relative to $B$ if and only if $p (M_n(\C) \ot \langle M,e_B \rangle) p$ admits an $A$-central state whose restriction to $p(M_n(\C) \ot M)p$ is faithful and normal.
\item We have that $A$ is strongly nonamenable relative to $B$ if and only if $p (M_n(\C) \ot \langle M,e_B \rangle) p$ admits no $A$-central state whose restriction to $p(M_n(\C) \ot M)p$ is normal.
\end{enumlist}
\end{remark}

Given a tracial von Neumann algebra $(\cM,\tau)$, two von Neumann subalgebras $\cB,M \subset \cM$ with $B = \cB \cap M$,
\begin{equation}\label{eq.commuting-square}
\begin{array}{ccc}
\cB & \subset & \cM \\ \cup & & \cup \\ B & \subset & M
\end{array}
\end{equation}
are said to form a \emph{commuting square} if $E_\cB \circ E_M = E_M \circ E_\cB = E_B$. If moreover the linear span of $\cB M$ is $\|\cdot\|_2$-dense in $L^2(\cM)$, the commuting square is said to be \emph{nondegenerate}.

If \eqref{eq.commuting-square} is a nondegenerate commuting square, the map $x e_B y \mapsto x e_{\cB} y$ for $x,y \in M$ extends to a unital embedding $\langle M,e_B \rangle \subset \langle \cM,e_\cB\rangle$ that preserves the canonical semifinite trace on the basic construction, so that there exists a faithful normal trace preserving conditional expectation $\cE : \langle \cM,e_\cB \rangle \to \langle M,e_B \rangle$ with $\cE(x) = E_M(x)$ for all $x \in \cM$.

Using Remark \ref{rem.everything-in-terms-of-basic-construction}, we then immediately get the following result.

\begin{lemma}\label{lem.commuting-square-transfer-embedding}
Let \eqref{eq.commuting-square} be a nondegenerate commuting square of tracial von Neumann algebras, $p \in M_n(\C) \ot M$ a projection and $A \subset p(M_n(\C) \ot M)p$ a von Neumann subalgebra.
\begin{enumlist}
\item We have $A \prec_M B$ if and only if $A \prec_\cM \cB$. We have $A \prec^f_M B$ if and only $A \prec^f_\cM \cB$.
\item We have that $A$ is amenable relative to $B$ inside $M$ if and only if $A$ is amenable relative to $\cB$ inside $\cM$.
\item We have that $A$ is strongly nonamenable relative to $B$ inside $M$ if and only if $A$ is strongly nonamenable relative to $\cB$ inside $\cM$.
\end{enumlist}
\end{lemma}

\begin{proof}
Using that $\langle M,e_B \rangle \subset \langle \cM,e_\cB\rangle$ and that we have a faithful normal conditional expectation $\cE : \langle \cM,e_\cB \rangle \to \langle M,e_B \rangle$ satisfying $\cE(x) = E_M(x)$ for all $x \in \cM$, we can extend $A$-central states on $p (M_n(\C) \ot \langle M,e_B \rangle)p$ by composing them with $\id \ot \cE$ and we can conversely restrict $A$-central states on $p (M_n(\C) \ot \langle \cM,e_\cB\rangle)p$ to the unital subalgebra $p (M_n(\C) \ot \langle M,e_B \rangle)p$. The conclusion thus follows from Remark \ref{rem.everything-in-terms-of-basic-construction}.
\end{proof}

The following is a folklore lemma. The first part is explained in \cite[Remark 3.8]{Vae07} and, for completeness, we include a proof for the second part.

\begin{lemma}\label{lem.intertwine-transitivity}
Let $M$ be a tracial von Neumann algebra, $p \in M_n(\C) \ot M$ a projection and $A \subset p(M_n(\C) \ot M)p$ and $D \subset B \subset M$ von Neumann subalgebras.

Assume that $A \prec_M B$, witnessed by a projection $q \in M_k(\C) \ot B$, a nonzero partial isometry $V \in p (M_{n,k}(\C) \ot M)q$ and a unital normal $*$-homomorphism $\theta : A \to q (M_k(\C) \ot B)q$ such that $a V = V \theta(a)$ for all $a \in A$ and such that the support of $(\id \ot E_B)(V^* V)$ equals $q$.

\begin{enumlist}
\item\label{lem.intertwine-transitivity.i} If $\theta(A) \prec_B D$, then $A \prec_M D$.
\item\label{lem.intertwine-transitivity.ii} If $A$ is strongly nonamenable relative to $D$ inside $M$, then $\theta(A)$ is strongly nonamenable relative to $D$ inside $B$.
\end{enumlist}
\end{lemma}
\begin{proof}
Since (i) is proven in \cite[Remark 3.8]{Vae07}, we only prove (ii). Assume that $\theta(A)$ is not strongly nonamenable relative to $D$ inside $B$. By Remark \ref{rem.everything-in-terms-of-basic-construction}, take a nonzero, $\theta(A)$-central, positive functional $\Om$ on $q (M_k(\C) \ot \langle B,e_D \rangle)q$ whose restriction to $q (M_k(\C) \ot B)q$ is normal. Note that we can view $e_B$ as an orthogonal projection in $\langle M , e_D \rangle$ and identify $\langle B,e_D \rangle = e_B \langle M,e_D \rangle e_B$. Then $(1 \ot e_B) q (M_k(\C) \ot \langle M,e_D \rangle)q (1 \ot e_B) = q (M_k(\C) \ot \langle B,e_D \rangle)q$ and
$$\om : p (M_n(\C) \ot \langle M,e_D \rangle)p \to \C : \om(T) = \Om((1 \ot e_B)V^* T V (1 \ot e_B))$$
is an $A$-central, positive functional whose restriction to $p(M_n(\C) \ot M)p$ is normal. Since $\om(p) = \Om((\id \ot E_B)(V^* V))$, the support of $(\id \ot E_B)(V^* V)$ equals $q$, and $\Om$ is normal on $q(M_k(\C) \ot B)q$, we get that $\om \neq 0$. So, $A$ is not strongly nonamenable relative to $D$ inside $M$.
\end{proof}

For later use, we also recall the following.

\begin{lemma}\label{lem.intertwining-intersection}
Let $(M,\tau)$ be a tracial von Neumann algebra and $Q_1, Q_2 \subset M$ von Neumann subalgebras which form a commuting square with $Q_1 \cap Q_2$. Assume that $Q_1$ is regular in $M$. Let $P \subset p(M_n(\C) \ot M)p$ be a von Neumann subalgebra. Then the following hold.
\begin{enumlist}
	\item\label{lem.intertwining-intersection.i} \cite[Proposition~2.7]{PV11} If $P$ is amenable relative to both $Q_1$ and $Q_2$, then $P$ is amenable relative to $Q_1 \cap Q_2$.
	
	\item\label{lem.intertwining-intersection.ii} \cite[Lemma~2.8]{DHI16} If $P \prec^f Q_1$ and $P \prec^f Q_2$, then $P \prec^f Q_1 \cap Q_2$.
\end{enumlist}
\end{lemma}

\subsection{\boldmath Countable groups, $2$-cocycles and twisted group von Neumann algebras}

We denote $\T = \{z \in \C \mid |z|=1\}$. Given a countable group $G$, one denotes by $Z^2(G,\T)$ the group of normalized scalar \emph{$2$-cocycles}, i.e.\ maps $\mu : G \times G \to \T$ satisfying
$$\mu(g,h) \, \overline{\mu}(g,hk) \, \mu(gh,k) \, \overline{\mu}(h,k) = 1 \quad\text{and}\quad \mu(g,e) = 1 = \mu(e,g) \quad\text{for all $g,h,k \in G$.}$$
One says that $\mu,\om \in Z^2(G,\T)$ are \emph{cohomologous}, denoted as $\mu \sim \om$, if there exists a map $\vphi : G \to \T$ such that $\vphi(e)=1$ and
$$\om(g,h) = \mu(g,h) \, \vphi(gh) \, \overline{\vphi}(g) \, \overline{\vphi}(h) \quad\text{for all $g,h \in G$.}$$
Sometimes, the normalization condition $\mu(g,e) = 1 = \mu(e,g)$ is not part of the definition of a $2$-cocycle, but any such more general $2$-cocycle is trivially cohomologous to a normalized one.

One denotes by $H^2(G,\T)$ the quotient of the abelian group $Z^2(G,\T)$ by the subgroup of $2$-cocycles that are cohomologous to $1$.

When $H$ is a Hilbert space, one says that a map $\pi : G \to \cU(H)$ is a \emph{projective representation} if
$$\pi(g) \, \pi(h) \in \T \cdot \pi(gh) \quad\text{and}\quad \pi(e) = 1 \quad\text{for all $g,h \in G$.}$$
Defining $\om_\pi : G \times G \to \T$ such that $\pi(g) \, \pi(h) = \om_\pi(g,h) \, \pi(gh)$, we obtain the associated $2$-cocycle $\om_\pi \in Z^2(G,\T)$.

We say that a $2$-cocycle $\mu \in Z^2(G,\T)$ is \emph{of finite type} if there exists a \emph{finite dimensional} projective representation $\pi$ with $\mu = \om_\pi$. Note that trivially, $\mu \sim 1$ if and only if $\mu = \om_\pi$ for a one-dimensional projective representation $\pi$. Also note that since we can take the tensor product of two projective representations, the finite type $2$-cocycles form a subgroup of $Z^2(G,\T)$.

Given $\mu \in Z^2(G,\T)$, the regular $\mu$-representation on the Hilbert space $\ell^2(G)$ with orthonormal basis $(\delta_h)_{h \in G}$ is defined by
$$u^\mu_g \delta_h = \mu(g,h) \, \delta_{gh} \quad\text{for all $g,h \in G$.}$$
Then $g \mapsto u^\mu_g$ is a projective representation with associated $2$-cocycle $\mu$. The \emph{twisted group von Neumann algebra} $L_\mu(G)$ is defined as the von Neumann algebra generated by the unitaries $(u^\mu_g)_{g \in G}$. When the context is sufficiently clear, we will often omit the superscript $\mu$ and simply denote by $u_g \in L_\mu(G)$ the canonical generating unitaries, satisfying $u_g u_h = \mu(g,h) u_{gh}$. We denote by $\tau$ the canonical trace on $L_\mu(G)$ given by $\tau(a) = \langle a \delta_e,\delta_e\rangle$.

If $G \actson^\si (A,\tau)$ is a trace preserving action and $\mu \in Z^2(G,\T)$, we denote by $A \rtimes^\si_\mu G$ the \emph{cocycle crossed product}, which is the unique tracial von Neumann algebra generated by a copy of $A$ and unitaries $(u_g)_{g \in G}$ satisfying
$$u_g \, u_h = \mu(g,h) \, u_{gh} \;\; , \quad u_g a u_g^* = \si_g(a) \;\; , \quad \tau(a u_g) = \begin{cases} \tau(a) &\;\;\text{if $g = e$,}\\ 0 &\;\;\text{if $g \neq e$,}\end{cases}$$
for all $g,h \in G$ and $a \in A$.

We make use of the following lemma to prove that certain $2$-cocycles are not of finite type. Recall that for groups $\Gamma$ and $\Lambda$, a map $\Om : \Gamma \times \Lambda \to \T$ is called a bicharacter if $\Om$ is multiplicative in both variables, meaning that $\Om(\cdot,h) : \Gamma \to \T$ and $\Om(g,\cdot) : \Lambda \to \T$ are homomorphisms for all $g \in \Gamma$ and $h \in \Lambda$.

\begin{lemma}\label{lem.finite-type-cocycles}
Let $\Gamma$, $\Lambda$ be groups and let $\pi : \Gamma \to \cU(d)$, $\rho : \Lambda \to \cU(d)$ be $d$-dimensional projective representations satisfying
$$\pi(g) \rho(h) = \Om(g,h) \rho(h) \pi(g) \quad\text{for all $g \in \Gamma, h \in \Lambda$,}$$
where $\Om : \Gamma \times \Lambda \to \T$ is a bicharacter. Then, $\Om$ factors through finite quotients of $\Gamma$ and $\Lambda$.
\end{lemma}
\begin{proof}
Define the normal subgroup $\Lambda_0 < \Lambda$ by $\Lambda_0 = \{h \in \Lambda \mid \forall g \in \Gamma : \Om(g,h) = 1\}$. By symmetry, it suffices to prove that $\Lambda_0 < \Lambda$ has finite index.

Define $\PU(d) = \cU(d) / \T \cdot 1$ and define the subgroup $\Lambda_1 < \PU(d)$ by $\Lambda_1 = \{\rho(h) \T \mid h \in \Lambda\}$. When $\rho(h) \in \T \cdot 1$, $\rho(h)$ commutes with all $\pi(g)$ and we conclude that $\Om(g,h) = 1$ for all $g \in \Gamma$, so that $h \in \Lambda_0$. We thus obtain a well-defined group homomorphism $\theta : \Lambda_1 \to \Lambda/\Lambda_0 : \theta(\rho(h) \T) = h\Lambda_0$.

We claim that $\theta$ extends to a continuous homomorphism from the closure $\overline{\Lambda_1} \subset \PU(d)$ to $\Lambda/\Lambda_0$. Assume that $(h_n)_n$ is a sequence in $\Lambda$ such that $\rho(h_n) \T \to \T$ in $\PU(d)$. We have to prove that $h_n \in \Lambda_0$ eventually. Assume the contrary. After passage to a subsequence, we have $h_n \not\in \Lambda_0$ for every $n \in \N$ and $\rho(h_n) \to \al \cdot 1$ in $\cU(d)$ for some $\al \in \T$.

Fix $n$. Since $h_n \not\in \Lambda_0$, $\{\Om(g,h_n) \mid g \in \Gamma\}$ is a nontrivial subgroup of $\T$. We can thus choose $g_n \in \Gamma$ such that $|\Om(g_n,h_n)-1| \geq 1$. Since $\rho(h_n) \to \al \cdot 1$, we have that $\pi(g_n) \rho(h_n) \pi(g_n)^* \rho(h_n)^* \to 1$. But this sequence equals $\Om(g_n,h_n) \cdot 1$, which does not converge to $1$. So the claim is proven.

By the claim, we find a continuous and surjective homomorphism from a compact group to the discrete group $\Lambda/\Lambda_0$. So, $\Lambda/\Lambda_0$ is finite.
\end{proof}

We also record the following lemma but omit the elementary proof. Recall that a countable group $G$ is said to have infinite conjugacy classes (icc) if $\{ghg^{-1} \mid g \in G\}$ is infinite for every $h \in G \setminus \{e\}$. A subgroup $\Lambda < G$ is said to be \emph{relatively icc} if $\{ghg^{-1} \mid g \in \Lambda\}$ is infinite for every $h \in G \setminus \{e\}$.

\begin{lemma}\label{lem.remarks-icc}
Let $G$ be a countable group.
\begin{enumlist}
\item\label{lem.remarks-icc.i} If $G$ is icc, then $G \to \Aut G : g \mapsto \Ad g$ embeds $G$ as a relatively icc subgroup of $\Aut G$.
\item\label{lem.remarks-icc.ii} Let $\Lambda$ be a group and $\delta_1,\delta_2 : \Lambda \to G$ group homomorphisms such that $\delta_1(\Lambda) < G$ is relatively icc. Exactly one of the following statements holds.
\begin{itemlist}
\item $\delta_1 = (\Ad g) \circ \delta_2$ for some $g \in G$, or
\item $\{\delta_1(h) g \delta_2(h^{-1}) \mid h \in \Lambda\}$ is infinite for every $g \in G$.
\end{itemlist}
Also, if $\delta_1(g) = \delta_2(g)$ for all $g$ in a finite index subgroup $\Lambda_0 < \Lambda$, then $\delta_1(g) = \delta_2(g)$ for all $g \in \Lambda$.
\item\label{lem.remarks-icc.iii} Let $\Lambda$ be a group with normal subgroup $\Lambda_0$ and let $\delta : \Lambda_0 \to G$ be a group homomorphism such that $\delta(\Lambda_0) < G$ is relatively icc. If for every $k \in \Lambda$, there exists an $s \in G$ such that the set $\{\delta(kgk^{-1}) s \delta(g)^{-1} \mid g \in \Lambda_0\}$ is finite, then $\delta$ uniquely extends to a group homomorphism $\Lambda \to G$.
\end{enumlist}
\end{lemma}

\subsection{Cocycle twisted groupoid von Neumann algebras}\label{sec.twisted-groupoid-vNalg}

To prove Theorem \ref{thm.B}, we only need very limited preliminary background on discrete pmp groupoids $\cG$ and their cocycle twisted groupoid von Neumann algebras $L_\om(\cG)$. We refer e.g.\ to \cite[Section 2]{Ana12} for further background. Here we only recall that such a discrete pmp groupoid $\cG$ comes with the following extra structure and assumptions. We are given a standard Borel structure such that the set $\cG^{(0)}$ of units is Borel and the source and target maps $s,t : \cG \to \cG^{(0)}$, as well as the multiplication map $\cG^{(2)} \ni (g,h) \mapsto gh \in \cG$ on $\cG^{(2)} = \{(g,h) \in \cG \times \cG \mid s(g) = t(h)\}$ are Borel. Moreover, we are given a probability measure $\mu$ on the Borel sets of $\cG^{(0)}$ such that for every bisection, i.e.\ Borel set $U \subset \cG$ such that $s|_U$ and $t|_U$ are injective, the associated map $\si_U : s(U) \to t(U)$ given by $\si_U(s(g)) = t(g)$ for all $g \in U$, is measure preserving.

We denote by $Z^2(\cG,\T)$ the group of normalized scalar $2$-cocycles on $\cG$, i.e.\ Borel maps $\om : \cG^{(2)} \to \T$ satisfying
$$\om(g,h) \, \overline{\om}(g,hk) \, \om(gh,k) \, \overline{\om}(h,k) = 1 \quad\text{and}\quad \om(g,s(g)) = 1 = \om(t(g),g)$$
for all $g,h,k \in \cG$ with $s(g) = t(h)$ and $s(h) = t(k)$.

Whenever $U,V \subset \cG$ are bisections, one considers the bisection
$$U \cdot V = \{gh \mid g \in U, h \in V, s(g) = t(h)\} \; .$$
Given a $2$-cocycle $\om \in Z^2(\cG,\T)$, one considers the partial isometry $\om(U,V) \in L^\infty(\cG^{(0)},\mu)$ given by $\om(U,V)(x) = \om(g,h)$ when $g \in U$, $h \in V$, $s(g) = t(h)$, $s(h) =x$, and $\om(U,V)(x)=0$ when $x \not\in s(U\cdot V)$.

The twisted groupoid von Neumann algebra $L_\om(\cG)$ is the unique tracial von Neumann algebra generated by a von Neumann subalgebra $B=L^\infty(\cG^{(0)},\mu)$ and partial isometries $u_U$ for every bisection $U \subset \cG$ satisfying the following relations for all bisections $U,V \subset \cG$ and $F \in B$,
\begin{align*}
& u_U \, u_V = u_{U \cdot V} \, \om(U,V) \;\; ,\quad u_U^* u_U = 1_{s(U)} \in B \;\; ,\quad u_U u_U^* = 1_{t(U)} \in B \;\; ,\\
& u_U^* F u_U = F \circ \si_U \quad\text{and}\quad \tau(F u_U) = \int_{U \cap \cG^{(0)}} F(x) \, d\mu(x) \; .
\end{align*}
By the usual maximality argument, any bisection $U \subset \cG$ can be extended to a bisection $V \subset \cG$ such that $s(V)$ and $t(V)$ are conull, for which $u_V$ is a unitary in $L_\om(\cG)$ normalizing $B$. So, $B \subset L_\om(\cG)$ is always a regular abelian subalgebra.

For every $x \in \cG^{(0)}$, the isotropy group $\cG_x$ is defined as $\cG_x = \{g \in \cG \mid s(g) = t(g) = x\}$. Every $\om \in Z^2(\cG,\T)$ restricts to $\om_x \in Z^2(\cG_x,\T)$. If $x$ is an atom of $\mu$, $1_x \in B$ is a projection and we obtain a canonical identification $1_x L_\om(\cG) 1_x = L_{\om_x}(\cG_x)$.

\section[Class $\cC$ and relative $\omega$-solidity]{\boldmath Class $\cC$ and relative $\omega$-solidity}

The first part of the following result is \cite[Theorem 1.4]{PV12}. The second part is a relative version of Ozawa's solidity theorem from \cite{Oza03} and seems by now to be a folklore result. We briefly sketch a proof.

\begin{theorem}[{\cite{PV12}}]\label{thm.relative-omega-solidity}
Let $\Gamma$ be a weakly amenable, biexact group and $\Gamma \curvearrowright (P, \tau)$ a trace preserving action on a tracial von Neumann algebra $(P, \tau)$. Let $\mu \in Z^2(\Gamma,\T)$ and write $M = P \rtimes_\mu \Gamma$. Let $p \in M$ be a nonzero projection.
\begin{enumlist}
\item\label{thm.relative-omega-solidity.i} If $A \subset pMp$ is a von Neumann subalgebra that is amenable relative to $P$ and satisfies $A \not\prec_M P$, then $\cN_{pMp}(A)\dpr$ remains amenable relative to $P$.
\item\label{thm.relative-omega-solidity.ii} If $A \subset pMp$ is a von Neumann subalgebra and $v_n \in \cU(pMp)$ is a sequence of unitaries such that $\|E_P(x v_n y)\|_2 \to 0$ for all $x,y \in M$ and such that $\|a v_n - v_n a \|_2 \to 0$ for all $a \in A$, then $A$ is amenable relative to $P$.
\item\label{thm.relative-omega-solidity.iii} In particular, if $A \subset pMp$ is a von Neumann subalgebra and $A' \cap pMp \not\prec_M P$, then $A$ is amenable relative to $P$.
\end{enumlist}
\end{theorem}
\begin{proof}
Define $\al : M \to M \ovt L(\Gamma) : \al(x u_g) = x u_g \ot u_g$ for all $x \in P$ and $g \in \Gamma$. Note that the subalgebras $\al(M)$ and $M \ot 1$ of $M \ovt L(\Gamma)$ form a nondegenerate commuting square with intersection $\al(M) \cap M \ot 1 = P \ot 1$. So we can apply Lemma \ref{lem.commuting-square-transfer-embedding}.

(i)\ Since $A$ is amenable relative to $P$ inside $M$, we get that $\al(A)$ is amenable relative to $M \ot 1$ inside $M \ovt L(\Gamma)$. Since $A \not\prec_M P$, we have that $\al(A) \not\prec_{M \ovt L(\Gamma)} M \ot 1$. So by \cite[Theorem 1.4]{PV12}, we find that $\al(\cN_{pMp}(A)\dpr)$ is amenable relative to $M \ot 1$ inside $M \ovt L(\Gamma)$. This in turn implies that $\cN_{pMp}(A)\dpr$ is amenable relative to $P$ inside $M$.

(ii)\ We write $q = \al(p)$ and consider $\al(A) \subset q (M \ovt L(\Gamma)) q$. We write $w_n = \al(v_n)$. Then $w_n \in q(M \ovt L(\Gamma))q$ is a sequence of unitaries satisfying $\|E_{M \ot 1}(x w_n y)\|_2 \to 0$ for all $x,y \in M \ovt L(\Gamma)$ and $\|\al(a) w_n - w_n \al(a)\|_2 \to 0$ for all $a \in A$. We can repeat the proof of \cite[Section 3.4, case 1]{PV12} and obtain that $\al(A)$ is amenable relative to $M \ot 1$ inside $M \ovt L(\Gamma)$. It follows that $A$ is amenable relative to $P$ inside $M$.

(iii)\ If $A' \cap pMp \not\prec_M P$, we can choose a sequence of unitaries $v_n \in A' \cap pMp$ such that $\|E_P(x v_n y)\|_2 \to 0$ for all $x,y \in M$. By (ii), we get that $A$ is amenable relative to $P$.
\end{proof}

\begin{corollary}\label{cor.relative-omega-solidity}
Consider the same hypotheses and notations as in Theorem \ref{thm.relative-omega-solidity}. Let $(\cM,\tau)$ be any tracial von Neumann algebra and $M \subset \cM$ an embedding. Let $q \in \cM$ be a projection and $A,B \subset q\cM q$ two commuting von Neumann subalgebras.

If $A$ is strongly nonamenable relative to $P$ inside $\cM$ and $B \not\prec_\cM P$, then $A \vee B \not\prec_\cM M$.
\end{corollary}
\begin{proof}
Assume that $A \vee B \prec_\cM M$. Take a projection $q' \in M_n(\C) \ot M$, a unital normal $*$-homomorphism $\theta : A \vee B \to q'(M_n(\C) \ot M)q'$ and a nonzero partial isometry $V \in q' (\C^n \ot \cM)$ such that $\theta(d) V = V d$ for all $d \in A \vee B$. We may moreover assume that $q'$ equals the support of $(\id \ot E_M)(VV^*)$. Since $B \not\prec_\cM P$, by Lemma \ref{lem.intertwine-transitivity.i}, we have that $\theta(B) \not\prec_M P$. By Theorem \ref{thm.relative-omega-solidity.iii}, this implies that $\theta(A)$ is amenable relative to $P$ inside $M$. By Lemma \ref{lem.intertwine-transitivity.ii}, this contradicts our assumption that $A$ is strongly nonamenable relative to $P$ inside $\cM$.
\end{proof}

\section{Height in cocycle twisted group von Neumann algebras}

Recall from \cite[Section 4]{Ioa10} and \cite[Section 3]{IPV10} the notion of \emph{height} in group von Neumann algebras, which we adapt in the following obvious way to twisted group von Neumann algebras. Given a countable group $\Gamma$, a $2$-cocycle $\mu \in Z^2(\Gamma,\T)$ and a subgroup $\cG < \cU(p(M_n(\C) \ot L_\mu(\Gamma))p)$, we define
$$h_\Gamma(\cG) = \inf_{v \in \cG} \bigl(\max \{ \|(v)_g\|_2 \mid g \in \Gamma\}\bigr) \quad\text{where}\quad (v)_g = (\id \ot \tau)((1 \ot u_g^*) v) \in M_n(\C) \; .$$

The following result is a cocycle adaptation of \cite[Theorem 4.1]{KV15} and \cite[Lemma 2.4]{PV21}.

\begin{theorem}\label{thm.height-1}
Let $\Gamma$ be an icc group and $\mu \in Z^2(\Gamma,\T)$ a $2$-cocycle. Write $M = L_\mu(\Gamma)$ and let $t > 0$. Let $\Lambda \subset \cU(M^t)$ be a subgroup satisfying the following properties:
\begin{enumlist}
\item $h_\Gamma(\Lambda) > 0$~;
\item for every $k \in \Gamma \setminus \{e\}$, we have that $\Lambda\dpr \not\prec_M L_\mu(C_\Gamma(k))$~;
\item the action $(\Ad v)_{v \in \Lambda}$ on $L^2(M^t) \ominus \C 1$ is weakly mixing.
\end{enumlist}
Then $t = 1$ and there exists a unitary $w \in M$ such that $w \Lambda w^* \subset \T \cdot \Gamma = \{ a u_g \mid a \in \T , g \in \Gamma\}$.
\end{theorem}
\begin{proof}
Define the embedding $\Delta : L(\Gamma) \to L_\mu(\Gamma) \ovt L_\mu(\Gamma)\op : \Delta(u_g) = u_g \ot \overline{u_g}$ for all $g \in \Gamma$. We claim that
\begin{equation}\label{eq.claim-embed}
\{v \ot \overline{v} \mid v \in \Lambda\}\dpr \prec_{M \ovt M\op} \Delta(L(\Gamma)) \; .
\end{equation}
To prove \eqref{eq.claim-embed}, we concretely realize $M^t = p(M_k(\C) \ot L_\mu(\Gamma))p$. We decompose any $v \in M_k(\C) \ot L_\mu(\Gamma)$ as $v = \sum_{g \in \Gamma} (v)_g \ot u_g$. Since $h_\Gamma(\Lambda)>0$, we can take $\delta > 0$ such that
$$\max \{ \|(v)_g\|_2 \mid g \in \Gamma\} \geq \delta \quad\text{for every $v \in \Lambda$.}$$
We identify $(M_k(\C) \ot L_\mu(\Gamma)) \ovt (M_k(\C) \ot L_\mu(\Gamma))\op = M_k(\C) \ovt M_k(\C) \ovt L_\mu(\Gamma) \ovt L_\mu(\Gamma)\op$ in the obvious way. We write $P = M_k(\C) \ot M_k(\C) \ot \Delta(L(\Gamma))$. Then for every $v \in \Lambda$,
$$E_P(v \ot \overline{v}) = \sum_{g \in \Gamma} (v)_g \ot \overline{(v)_g} \ot \Delta(u_g)$$
and thus
$$\|E_P(v \ot \overline{v})\|_2^2 = \sum_{g \in \Gamma} \|(v)_g\|_2^4 \geq \delta^4 \quad\text{for all $v \in \Lambda$.}$$
So, \eqref{eq.claim-embed} is proven.

By \eqref{eq.claim-embed}, we can take a projection $q \in M_n(\C) \ot L(\Gamma)$, a group homomorphism $\theta : \Lambda \to \cU(q(M_n(\C) \ot L(\Gamma))q)$ and a nonzero
$$X \in (p \ot p\op)(M_{k,n}(\C) \ovt L_\mu(\Gamma) \ovt \C^k \ovt L_\mu(\Gamma)\op)(\id \ot \Delta)(q)$$
such that $(v \ot \overline{v}) X = X (\id \ot \Delta)(\theta(v))$ for all $v \in \Lambda$.

For $k \in \Gamma \setminus \{e\}$, we have $\Lambda\dpr \not\prec_M L_\mu(C_\Gamma(k))$. We may thus choose $\theta$ such that $\theta(\Lambda)\dpr \not\prec_{L(\Gamma)} L(C_\Gamma(k))$ for all $k \in \Gamma \setminus \{e\}$. It then follows from a twisted version of \cite[Proposition 7.2.3]{IPV10} (cf.\ Proposition \ref{prop.properties-triple-comult.i} below) that the relative commutant of $\{(\id \ot \Delta)(\theta(v)) \mid v \in \Lambda\}$ is contained in the image of $\id \ot \Delta$. We may thus assume that $X^* X = (\id \ot \Delta)(q)$.

By weak mixing of $(\Ad v)_{v \in \Lambda}$ on $L^2(M^t) \ominus \C 1$, it then follows that $X X^* = p \ot p\op$. So we can view $X$ as a unitary conjugacy between $v \ot \overline{v}$ and $(\id \ot \Delta)(\theta(v))$ for all $v \in \Lambda$. In particular, the action $(\Ad \theta(v))_{v \in \Lambda}$ on $qL^2(M_n(\C) \ot L(\Gamma))q \ominus \C q$ is weakly mixing. Since $h_\Gamma(\Lambda) > 0$, also $h_\Gamma(\theta(\Lambda)) > 0$, where $h_\Gamma(\theta(\Lambda))$ denotes the height of $\theta(\Lambda)$ in the untwisted amplified group von Neumann algebra $q(M_n(\C) \ot L(\Gamma))q$.

We can thus apply \cite[Lemma 2.4]{PV21} and conclude that $(\Tr \ot \tau)(q) = 1$ and that $\theta(\Lambda)$ can be unitarily conjugated into $\T \cdot \Gamma$. So $t = 1$ and we find a group homomorphism $\vphi : \Lambda \to \Gamma$, a character $\eta : \Lambda \to \T$ and a unitary $X \in L_\mu(\Gamma) \ovt L_\mu(\Gamma)\op$ such that
$$(v \ot \overline{v}) X = \eta(v) X (u_{\vphi(v)} \ot \overline{u_{\vphi(v)}}) \quad\text{for all $v \in \Lambda$.}$$
It follows that the projective representation
$$\gamma : \Lambda \to \cU(L^2(L_\mu(\Gamma))) : \gamma(v) (\xi) = v \xi u_{\vphi(v)}^*$$
is not weakly mixing. We can thus find an irreducible projective representation $\pi : \Lambda \to \cU(m)$ and a nonzero $Y \in \C^m \ot L_\mu(\Gamma)$ such that
$$(\pi(v) \ot v) Y = Y u_{\vphi(v)} \quad\text{for all $v \in \Lambda$.}$$
Since $\Lambda\dpr \not\prec_M L_\mu(C_\Gamma(k))$ when $k \in \Gamma \setminus \{e\}$, the subgroup $\vphi(\Lambda) < \Gamma$ is relatively icc. So, $Y^* Y$ is a multiple of $1$ and we may assume that $Y^*Y = 1$. By weak mixing of $(\Ad v)_{v \in \Lambda}$ on $L^2(M) \ominus \C 1$ and irreducibility of $\pi$, also $YY^* = 1$. So $m=1$ and we have unitarily conjugated $\Lambda$ into $\T \cdot \Gamma$.
\end{proof}

We also need a cocycle variant of \cite[Theorem 2.3]{PV21}, which we moreover need to adapt to a product of multiple groups. As in \cite[Definition 2.1]{PV21}, given a countable group $G$ and a $2$-cocycle $\mu \in Z^2(G,\T)$, we say that a group homomorphism $\theta : \Lambda \to \cU(p(M_n(\C) \ot L_\mu(G))p)$ from a countable group $\Lambda$ is \emph{standard} if there exist
\begin{itemlist}
\item finitely many finite index subgroups $\Lambda_i < \Lambda$,
\item group homomorphisms $\delta_i : \Lambda_i \to G$,
\item finite dimensional projective representations $\pi_i : \Lambda_i \to \cU(n_i)$ with $2$-cocycle $\om_{\pi_i} = \overline{\mu} \circ \delta_i$,
\end{itemlist}
with corresponding $\theta_i : \Lambda_i \to \cU(M_{n_i}(\C) \ot L_\mu(G)) : \theta_i(g) = \pi_i(g) \ot u_{\delta_i(g)}$, such that $\theta$ is unitarily conjugate to the direct sum of the inductions of $\theta_i$ to $\Lambda$. In particular, $(\Tr \ot \tau)(p) = \sum_i n_i [\Lambda:\Lambda_i]$.

The following is an immediate consequence of Theorem \ref{thm.height-1}. The proof of this corollary is identical to the second part of the proof of \cite[Theorem 2.3]{PV21} and thus omitted.

\begin{corollary}\label{cor.height-no-weak-mixing}
Let $\Gamma$ be an icc group and $\mu \in Z^2(\Gamma,\T)$ a $2$-cocycle. Write $M = L_\mu(\Gamma)$ and let $t > 0$. Let $\Lambda$ be a group and $\theta : \Lambda \to \cU(M^t)$ a group homomorphism. Define $A_0 \subset M^t$ as the set of all elements $a \in M^t$ such that $\lspan \{\theta(v) a \theta(v)^* \mid v \in \Lambda\}$ is finite dimensional.

Then $A_0 \subset M^t$ is a $*$-subalgebra and we denote by $A$ its weak closure. Assume that $A$ is atomic and assume that the following holds:
\begin{enumlist}
\item $h_\Gamma(\theta(\Lambda)p) > 0$ for every nonzero projection $p \in \theta(\Lambda)' \cap M^t$~;
\item for every $k \in \Gamma \setminus \{e\}$, we have that $\theta(\Lambda)\dpr \not\prec_M L_\mu(C_\Gamma(k))$.
\end{enumlist}
Then $\theta : \Lambda \to \cU(L_\mu(\Gamma)^t)$ is standard.
\end{corollary}

To apply Theorem \ref{thm.height-1}, we need to know that $(\Ad v)_{v \in \Lambda}$ is weakly mixing on $L^2(M^t) \ominus \C 1$. To apply Corollary \ref{cor.height-no-weak-mixing}, we still need to control the absence of weak mixing of $(\Ad \theta(v))_{v \in \Lambda}$ in a precise way, because we need to prove that $A$ is atomic. In the case where $\Gamma$ is a product of weakly amenable, biexact groups, the latter is often automatic thanks to the following lemma that we will use later. The proof of this lemma is very similar to the first part of the proof of \cite[Theorem 2.3]{PV21}.

\begin{lemma}\label{lem.automatic-atomic}
Let $\Gamma_1,\ldots,\Gamma_k$ be weakly amenable, biexact groups, put $G = \Gamma_1 \times \cdots \times \Gamma_k$, and let $\mu \in Z^2(G,\T)$. Let $p \in M_n(\C) \ot L_\mu(G)$ be a projection and $\Lambda \subset \cU(p(M_n(\C) \ot L_\mu(G))p)$ a subgroup.

Define $A_0$ as the set of elements $a \in p(M_n(\C) \ot L_\mu(G))p$ such that $\lspan \{v a v^* \mid v \in \Lambda\}$ is finite dimensional. Denote by $A$ the weak closure of the $*$-algebra $A_0$.

Write $G_i = \Gamma_1 \times \cdots \times \Gamma_{i-1} \times \{e\} \times \Gamma_{i+1} \times \cdots \times \Gamma_k$ and assume that $\Lambda\dpr$ is strongly nonamenable relative to $L_\mu(G_i)$ for all $i \in \{1,\ldots,k\}$.

Then $A$ is atomic.
\end{lemma}
\begin{proof}
We write $M = L_\mu(G)$ and denote $p(M_n(\C) \ot L_\mu(G))p$ as $M^t$. Note that by definition, $\Lambda \subset \cN_{M^t}(A)$ and that the action $(\Ad v)_{v \in \Lambda}$ on $A$ is compact.

Denote $G_i$ as in the theorem, denote by $\om_i$ the restriction of $\mu$ to $G_i$ and write $P_i = L_{\om_i}(G_i)$. To prove that $A$ is purely atomic, it suffices to prove that $A \prec^f_M \C 1$. So by Lemma \ref{lem.intertwining-intersection.ii}, it suffices to prove that $A \prec^f_M P_i$ for every $i \in \{1,\ldots,k\}$. Denoting by $\mu_i$ the restriction of $\mu$ to $\Gamma_i$, we find an action $\Gamma_i \actson P_i$ so that $M = P_i \rtimes_{\mu_i} \Gamma_i$. We will apply Theorem \ref{thm.relative-omega-solidity} in this context.

Let $q \in A' \cap M^t$ be the largest projection such that $Aq \not\prec_M P_i$. It suffices to prove $q=0$. Assume that $q \neq 0$. Since $q$ belongs to the center of the normalizer of $A$ inside $M^t$, we have that $q$ commutes with $\Lambda$.

Since $\Lambda\dpr \not\prec P_i$, we can take a sequence $v_n \in \Lambda$ such that $\|E_{P_i}(x^* v_n y)\|_2 \to 0$ for all $x,y \in \C^n \ot M$. Since the action $(\Ad v)_{v \in \Lambda}$ on $A$ is compact and trace preserving, after passage to a subsequence, we may assume that $\|v_n a v_n^* - \al(a)\|_2 \to 0$ for all $a \in A$, where $\al \in \Aut A$ is a trace preserving automorphism.

Choose an increasing sequence $\cF_s \subset \C^n \ot M$ of finite subsets such that $\bigcup_s \cF_s$ is $\|\cdot\|_2$-dense in $\C^n \ot M$. Also choose an increasing sequence $\cA_s \subset A$ of finite subsets such that $\bigcup_s \cA_s$ is $\|\cdot\|_2$-dense in $A$. Take $n_1$ such that $\|E_{P_i}(x^* v_{n_1} y)\|_2 < 1/2$ for all $x,y \in \cF_1$ and $\|v_{n_1} a v_{n_1}^* - \al(a)\|_2 < 1/2$ for all $a \in \cA_1$. Having chosen $n_1 < n_2 < \cdots < n_{s-1}$, choose $n_s > n_{s-1}$ such that $\|E_{P_i}(x^* \, v_{n_s} \, v_{n_{s-1}}^* y)\|_2 < 2^{-s}$ for all $x,y \in \cF_s$ and $\|v_{n_s} a v_{n_s}^* - \al(a)\|_2 < 2^{-s}$ for all $a \in \cA_s$. Defining $w_s = v_{n_s} v_{n_{s-1}}^*$, we have found a sequence $(w_s)_s$ in $\Lambda$ such that
$$\|E_{P_i}(x^* w_s y)\| \to 0 \quad\text{for all $x,y \in \C^n \ot M$, and}\quad \|w_s a - a w_s\|_2 \to 0 \quad\text{for all $a \in A$.}$$
By Theorem \ref{thm.relative-omega-solidity.ii}, $A$ is amenable relative to $P_i$ inside $M$. In particular, $Aq$ is amenable relative to $P_i$ inside $M$. Since $Aq \not\prec_M P_i$ and since the normalizer of $Aq$ contains $\Lambda q$, it follows from Theorem \ref{thm.relative-omega-solidity.i} that $\Lambda\dpr q$ is amenable relative to $P_i$. This contradicts the assumptions of the lemma.

So $q=0$ and the lemma is proven.
\end{proof}

\section{Coarse embeddings of Bernoulli like crossed products}

\subsection{Coarse embeddings into tensor products}\label{sec.coarse-embedding}

Recall that for tracial von Neumann algebras $P$ and $Q$, a $P$-$Q$-bimodule $\bim{P}{\cH}{Q}$ is called \emph{coarse} if it is isomorphic to a $P$-$Q$-subbimodule of a direct sum of copies of $\bim{(P \ot 1)}{L^2(P \ovt Q)}{(1 \ot Q)}$.

\begin{definition}\label{def.coarse-embedding}
Let $M$, $M_1,\ldots,M_k$ be tracial von Neumann algebras and take a projection $p \in M_n(\C) \ovt M_1 \ovt \cdots \ovt M_k$. Motivated by \cite[Definition 5.1]{BV22}, we say that an embedding
$$\psi : M \to p (M_n(\C) \ovt M_1 \ovt \cdots \ovt M_k) p$$
is \emph{coarse} if for every $i \in \{1,\ldots,k\}$, the bimodule
$$\bim{\psi(M)}{p (\C^n \ot L^2(M_1 \ovt \cdots \ovt M_k))}{(M_1 \ovt \cdots \ovt M_{i-1} \ovt 1 \ovt M_{i+1} \ovt \cdots \ovt M_k)}$$
is coarse. By convention, if $k=1$, the condition is empty and every embedding $\psi : M \to p (M_n(\C) \ovt M_1)p$ is called coarse.
\end{definition}

We can equivalently give the following definition of a \emph{tensor coarse} bimodule.

\begin{definition}\label{def.tensor-coarse-bimodule}
Let $M$, $M_1,\ldots,M_k$ be tracial von Neumann algebras. We say that a Hilbert $M$-$(M_1 \ovt \cdots \ovt M_k)$-bimodule $H$ is \emph{tensor coarse} if $H$ is finitely generated as a right Hilbert $(M_1 \ovt \cdots \ovt M_k)$-module and if for every $i \in \{1,\ldots,k\}$, the bimodule
$$\bim{M}{H}{M_1 \ovt \cdots \ovt M_{i-1} \ovt 1 \ovt M_{i+1} \ovt \cdots \ovt M_k}$$
is coarse.
\end{definition}

By definition, tensor coarse bimodules are precisely bimodules of the form
$$\bim{\psi(M)}{p (\C^n \ovt L^2(M_1 \ovt \cdots \ovt M_k))}{M_1 \ovt \cdots \ovt M_k}$$
where $\psi : M \to p (M_n(\C) \ovt M_1 \ovt \cdots \ovt M_k)p$ is a coarse embedding in the sense of Definition \ref{def.coarse-embedding}.

\begin{lemma}\label{lem.stability-of-coarse}
Let $P$, $P_1$, $P_2$, $M_1$, $M_2$ be tracial von Neumann algebras. Let $H$ be a Hilbert $P$-$(P_1 \ovt P_2)$-bimodule. For $i \in \{1,2\}$, let $L_i$ be a Hilbert $P_i$-$M_i$-bimodule.

If the bimodule $\bim{P}{H}{P_1 \ot 1}$ is coarse, also the bimodule $\bim{P}{(H \ot_{P_1 \ovt P_2} (L_1 \ot L_2))}{M_1 \ot 1}$ is coarse.
\end{lemma}
\begin{proof}
Define the Hilbert $P$-$(P_1 \ovt M_2)$-bimodule $K = H \ot_{P_1 \ovt P_2} (L^2(P_1) \ot L_2)$. Since $\bim{P}{H}{P_1 \ot 1}$ is coarse, we can view $H$ as a $(P \ovt P_1\op)$-$P_2$-bimodule. We can then identify $K$ with the $(P \ovt P_1\op)$-$M_2$-bimodule $H \ot_{P_2} L_2$. It follows that the bimodule $\bim{P}{K}{P_1 \ot 1}$ is coarse. Since
$$H \ot_{P_1 \ovt P_2} (L_1 \ot L_2) = K \ot_{P_1} L_1$$
as $P$-$M_1$-bimodules, we conclude that $\bim{P}{(H \ot_{P_1 \ovt P_2} (L_1 \ot L_2))}{M_1 \ot 1}$ is coarse.
\end{proof}

The following straightforward lemma is used throughout this article.

\begin{lemma}\label{lem.coarse-gives-nonintertwining}
Let $M$, $P$ and $Q$ be tracial von Neumann algebras, $p \in M_n(\C) \ovt P \ovt Q$ a projection and $\psi : M \to p(M_n(\C) \ovt P \ovt Q)p$ an embedding. Assume that the bimodule
$$\bim{\psi(M)}{p(\C^n \ot L^2(P \ovt Q))}{P \ot 1}$$
is coarse.
\begin{enumlist}
\item\label{lem.coarse-gives-nonintertwining.i} If $A \subset M$ is a diffuse von Neumann subalgebra, then $\psi(A) \not\prec_{P \ovt Q} P \ot 1$.
\item\label{lem.coarse-gives-nonintertwining.ii} If $A \subset M$ is a von Neumann subalgebra without amenable direct summand, then $\psi(A)$ is strongly nonamenable relative to $P \ot 1$.
\end{enumlist}
\end{lemma}
\begin{proof}
By the definition of a coarse bimodule, we can choose an isometry
$$V : p(\C^n \ot L^2(P \ovt Q)) \to \ell^2(\N) \ot L^2(M) \ot L^2(P)$$
such that $V(\psi(a)\xi (b \ot 1)) = (1 \ot a \ot 1) V(\xi) (1 \ot b)$ for all $a \in M$, $b \in P$ and $\xi \in p(\C^n \ot L^2(P \ovt Q))$. Denote $q = VV^*$ and note that $q \in B(\ell^2(\N)) \ovt JMJ \ovt P$, where $J : L^2(M) \to L^2(M) : J(x) = x^*$.

Write $\cM = q(B(\ell^2(\N)) \ovt B(L^2(M)) \ovt P)q$. Then $V$ induces an isomorphism between the inclusions
$$\psi(M) \subset p(M_n(\C) \ot \langle P \ovt Q,e_{P \ot 1} \rangle)p \quad\text{and}\quad (1 \ot M \ot 1)q \subset \cM \; .$$
We use Remark \ref{rem.everything-in-terms-of-basic-construction}. If $\psi(A) \prec_{P \ovt Q} P \ot 1$, we find a normal state $\Om$ on $\cM$ that is $(1 \ot A \ot 1)q$-central. Then $\om(T) = \Om(q(1 \ot T \ot 1)q)$ defines a normal $A$-central state on $B(L^2(M))$. So, $A$ is not diffuse.

Similarly, if $\psi(A)$ is not strongly nonamenable relative to $P \ot 1$, we find a state $\Om$ on $\cM$ that is $(1 \ot A \ot 1)q$-central and whose restriction to $(1 \ot A \ot 1)q$ is normal. Then $\om(T) = \Om(q(1 \ot T \ot 1)q)$ defines an $A$-central state on $B(L^2(M))$ whose restriction to $A$ is normal. So, $A$ has an amenable direct summand.
\end{proof}

We also have the following obvious lemma.

\begin{lemma}\label{lem.stable-coarse}
Let $M$, $P$, $Q$, $A$ and $B$ be tracial von Neumann algebras. Let
$$\psi : M \to p(M_n(\C) \ovt P \ovt Q)p \quad\text{and}\quad \vphi : Q \to q (M_k(\C) \ovt A \ovt B)q$$
be embeddings. Define $e = (\id \ot \vphi)(p)$.

If both $\bim{\psi(M)}{p(\C^n \ot L^2(P \ovt Q))}{P \ot 1}$ and $\bim{\vphi(Q)}{q(\C^k \ot L^2(A \ovt B))}{A \ot 1}$ are coarse bimodules, then
$$\bim{(\id \ot \vphi)\psi(M)}{e (\C^n \ot \C^k \ot L^2(P \ovt A \ovt B))}{P \ovt A \ovt 1}$$
is also a coarse bimodule.
\end{lemma}
\begin{proof}
If $V : p(\C^n \ot L^2(P \ovt Q)) \to (L^2(M) \ot L^2(P))^{\oplus \infty}$ is an $M$-$P$-bimodular isometry and $W : q(\C^k \ot L^2(A \ovt B)) \to (L^2(Q) \ot L^2(A))^{\oplus \infty}$ is a $Q$-$A$-bimodular isometry, one can interpret $(V \ot 1)(1 \ot W)$ as the required $M$-$(P \ovt A)$-bimodular isometry
\begin{equation*}
e (\C^n \ot \C^k \ot L^2(P \ovt A \ovt B)) \to (L^2(M) \ot L^2(P \ovt A))^{\oplus \infty} \; . \qedhere
\end{equation*}
\end{proof}

\subsection{Bernoulli like crossed products}\label{sec.bernoulli-like}

Throughout this section, we fix an action $\cG \actson I$ of a countable group $\cG$ on a countable set $I$. In order to simultaneously treat the generalized Bernoulli action $\cG \actson (A_0,\tau_0)^I$ and the action $\cG \actson^\si L_\mu(\Lambda_0^{(I)}) : \si_g = \Ad u^\mu_g$ given a $2$-cocycle $\mu \in Z^2(\Lambda_0^{(I)} \rtimes \cG,\T)$, we introduce the following ad hoc notion.

We say that a trace preserving action $\cG \actson^\si (B,\tau)$ on an amenable tracial von Neumann algebra $(B,\tau)$ is \emph{Bernoulli like over $\cG \actson I$} if the following holds. There exists a trace preserving action $\cG \actson^\zeta (C,\tau)$, a tracial amenable $(A_0,\tau_0)$ with corresponding generalized Bernoulli action $\cG \actson^\gamma (A,\tau) = (A_0,\tau_0)^I$ and a unital $*$-homomorphism $\be : B \to C \ovt A$ satisfying
\begin{enumlist}
\item $(\zeta_g \ot \gamma_g) \circ \be = \be \circ \si_g$ for all $g \in \cG$,
\item $(\id \ot \tau)\be(b) = \tau(b) 1$ for all $b \in B$,
\item the linear span of $(C \ot 1)\be(B)$ is $\|\cdot\|_2$-dense in $C \ovt A$,
\item for every finite subset $\cF \subset I$, we have a von Neumann subalgebra $B_\cF \subset B$ such that, writing $A_\cF = A_0^\cF$, we have $(\id \ot E_{A_\cF}) \circ \be = \be \circ E_{B_\cF}$.
\end{enumlist}

Note that by (ii) and (iii), we get that
\begin{equation}\label{eq.comm-square-C-B}
\begin{array}{ccc}
C \ot 1 & \subset & C \ovt A \\ \cup & & \cup \\ \C 1 & \subset & \be(B)
\end{array}
\end{equation}
is a nondegenerate commuting square.

\begin{example}
When $\cG \actson^\si (B,\tau)$ equals the generalized Bernoulli action $\cG \actson^\si (A_0,\tau_0)^I$, we take $C = \C 1$, $\be$ the identity map and $\gamma = \si$. When $\mu \in Z^2(\Lambda_0^{(I)} \rtimes \cG,\T)$ and $\cG \actson^\si B = L_\mu(\Lambda_0^{(I)})$ is given by $\si_g = \Ad u^\mu_g$, we take $C = B$, $\zeta_g = \si_g$, $\cG \actson^\gamma L(\Lambda_0^{(I)})$ the generalized Bernoulli action, and
$$\be : L_\mu(\Lambda_0^{(I)}) \to L_\mu(\Lambda_0^{(I)}) \ovt L(\Lambda_0^{(I)}) : \be(u^\mu_a) = u^\mu_a \ot u_a \quad\text{for all $a \in \Lambda_0^{(I)}$,}$$
with $B_\cF = L_\mu(\Lambda_0^\cF)$.
\end{example}

For the rest of this section, we fix such a Bernoulli like action $\cG \actson^\si (B,\tau)$. We also fix a $2$-cocycle $\mu \in Z^2(\cG,\T)$ and write $M = B \rtimes_{\si,\mu} \cG$, $\cN = C \rtimes_{\zeta,\mu} \cG$ and $\cM = A \rtimes_\gamma \cG$. We extend $\be$ to an embedding
\begin{equation}\label{eq.embedding-beta}
\be : M \to \cN \ovt \cM : \be(u^\mu_r) = u^\mu_r \ot u_r \quad\text{for all $r \in \cG$.}
\end{equation}

Note that the following properties follow from the assumptions (i)--(iv) above. Since $(\id \ot E_{A_\cF}) \circ \be = \be \circ E_{B_\cF}$, we have that $b \in B_\cF$ if and only if $\be(b) \in C \ovt A_\cF$. In particular, $\si_g(B_\cF) = B_{g \cdot \cF}$. It also follows that $E_{B_{\cF_1}} \circ E_{B_{\cF_2}} = E_{B_{\cF_1 \cap \cF_2}}$.

Since the linear span of $(C \ot 1)\be(B)$ is dense in $C \ovt A$, applying $\id \ot E_{A_{\cF_1}}$, it follows that the linear span of $(C \ot 1)\be(B_{\cF_1})$ is dense in $C \ovt A_{\cF_1}$. Then $(C \ot 1)\be(B_{\cF_1} B_{\cF_2})$ has the same closed linear span as $(C \ot A_{\cF_1}) \be(B_{\cF_2}) = (1 \ot A_{\cF_1}) (C \ot 1)\be(B_{\cF_2})$, which thus equals $C \ovt A_{\cF_1 \cup \cF_2}$. We conclude that the linear span of $B_{\cF_1} B_{\cF_2}$ is dense in $B_{\cF_1 \cup \cF_2}$. This means that
\begin{equation}\label{eq.commuting-square-with-BF}
\begin{array}{ccc}
B_{\cF_1} & \subset & B_{\cF_1 \cup \cF_2} \\ \cup & & \cup \\ B_{\cF_1 \cap \cF_2} & \subset & B_{\cF_2}
\end{array}
\end{equation}
is a nondegenerate commuting square.

We use $\be$ to transfer several known results on the generalized Bernoulli crossed product $\cM = A_0^I \rtimes_\gamma \cG$ to the twisted Bernoulli like crossed product $M = B \rtimes_{\si,\mu} \cG$. The following elementary lemma helps us in this respect.

\begin{lemma}\label{lem.transfer-1}
Let $(N,\tau)$ be a tracial von Neumann, $p \in N \ovt M$ a projection and $P \subset p (N \ovt M)p$ a von Neumann subalgebra. Let $\Lambda < \cG$ be a subgroup.
\begin{enumlist}
\item\label{lem.transfer-1.i} We have $P \prec_{N \ovt M} N \ovt (B \rtimes_{\si,\mu} \Lambda)$ if and only if $(\id \ot \be)(P) \prec_{N \ovt \cN \ovt \cM} N \ovt \cN \ovt (A \rtimes_\gamma \Lambda)$. The same holds for full embedding $\prec^f$ instead of $\prec$.
\item\label{lem.transfer-1.ii} We have $P \prec_{N \ovt M} N \ovt L_\mu(\Lambda)$ if and only if $(\id \ot \be)(P) \prec_{N \ovt \cN \ovt \cM} N \ovt \cN \ovt L(\Lambda)$. The same holds for full embedding $\prec^f$ instead of $\prec$.
\item\label{lem.transfer-1.iii} We have that $P$ is amenable relative to $N \ovt (B \rtimes_{\si,\mu} \Lambda)$ if and only if $(\id \ot \be)(P)$ is amenable relative to $N \ovt \cN \ovt (A \rtimes_\gamma \Lambda)$. The same holds for strong relative amenability.
\end{enumlist}
\end{lemma}
\begin{proof}
Since \eqref{eq.comm-square-C-B} is a nondegenerate commuting square, also
$$\begin{array}{ccc}
\cN \ovt (A \rtimes_\gamma \Lambda) & \subset & \cN \ovt \cM \\ \cup & & \cup \\ \beta(B \rtimes_{\si,\mu} \Lambda) & \subset & \beta(M)
\end{array}\qquad\text{and}\qquad
\begin{array}{ccc}
\cN \ovt L(\Lambda) & \subset & \cN \ovt \cM \\ \cup & & \cup \\ \beta(L_\mu(\Lambda)) & \subset & \beta(M)
\end{array}
$$
are nondegenerate commuting squares. So also taking throughout the tensor product with $N$, we get nondegenerate commuting squares. The lemma then follows from Lemma \ref{lem.commuting-square-transfer-embedding}.
\end{proof}

The restriction of $\be$ to $L_\mu(\cG) \subset M$ is the comultiplication $\Delta : L_\mu(\cG) \to L_\mu(\cG) \ovt L(\cG) : \Delta(u^\mu_r) = u^\mu_r \ot u_r$ for all $r \in \cG$. We then also have the following result, which may be viewed as a cocycle twisted version of \cite[Proposition 7.2]{IPV10}. As with Lemma \ref{lem.transfer-1}, we give a short commuting square proof.

\begin{lemma}\label{lem.transfer-2}
Let $(N,\tau)$ be a tracial von Neumann, $p \in N \ovt L_\mu(\cG)$ a projection and $P \subset p (N \ovt L_\mu(\cG))p$ a von Neumann subalgebra. Let $\Lambda < \cG$ be a subgroup and write $S = L_\mu(\cG)$.
\begin{enumlist}
\item\label{lem.transfer-2.i} We have $P \prec_{N \ovt S} N \ovt L_\mu(\Lambda)$ if and only if $(\id \ot \be)(P) \prec_{N \ovt S \ovt L(\cG)} N \ovt S \ovt L(\Lambda)$. The same holds for full embedding $\prec^f$ instead of $\prec$.
\item\label{lem.transfer-2.ii} We have that $P$ is amenable relative to $N \ovt L_\mu(\Lambda)$ if and only if $(\id \ot \be)(P)$ is amenable relative to $N \ovt S \ovt L(\Lambda)$. The same holds for strong relative amenability.
\end{enumlist}
\end{lemma}
\begin{proof}
Note that
$$\begin{array}{ccc}
S \ovt L(\Lambda) & \subset & S \ovt L(\cG) \\ \cup & & \cup \\ \beta(L_\mu(\Lambda)) & \subset & \beta(S)
\end{array}$$
is a nondegenerate commuting square. Taking throughout the tensor product with $N$, we still have a nondegenerate commuting square, and the conclusion follows from Lemma \ref{lem.commuting-square-transfer-embedding}.
\end{proof}

The proof of the following result is almost identical to the proof of \cite[Lemma 4.2.1]{Vae07} and thus omitted.

\begin{lemma}\label{lem.weak-mixing}
Let $(N,\tau)$ be a tracial von Neumann, $p \in N \ovt L_\mu(\cG)$ a projection and $P \subset p (N \ovt L_\mu(\cG))p$ a von Neumann subalgebra.

If for all $i \in I$, we have that $P \not\prec_{N \ovt L_\mu(\cG)} N \ovt L_\mu(\Stab i)$, then $P' \cap p (N \ovt M) p \subset p (N \ovt L_\mu(\cG))p$.
\end{lemma}

For later use, we prove the following elementary lemma.

\begin{lemma}\label{lem.further-commutant-result}
Let $(N,\tau)$ be a tracial von Neumann, $\Lambda$ a group, $\pi : \Lambda \to \cU(N)$ a projective representation and $\delta : \Lambda \to \cG$ a group homomorphism. Write
$$\theta : \Lambda \to \cU(N \ovt M) : \theta(g) = \pi(g) \ot u^\mu_{\delta(g)} \; .$$
\begin{enumlist}
\item If $\delta(\Lambda) < \cG$ is relatively icc, then $\theta(\Lambda)' \cap N \ovt M \subset N \ovt B$.
\item If moreover the action $\delta(\Lambda) \actson I$ has infinite orbits, then $\theta(\Lambda)' \cap N \ovt M \subset N \ot 1$.
\end{enumlist}
\end{lemma}
\begin{proof}
Fix $a \in \theta(\Lambda)' \cap N \ovt M$.

(i) Write $a = \sum_{g \in \cG} (a)_g (1 \ot u^\mu_g)$ with $(a)_g \in N \ovt B$. Since $a$ commutes with $\theta(\Lambda)$, we get that the function $g \mapsto \|(a)_g\|_2$ is invariant under conjugation by $\delta(\Lambda)$. Since this function is also square summable and $\delta(\Lambda) < \cG$ is relatively icc, it follows that $(a)_g = 0$ for all $g \neq e$. This means that $a \in N \ovt B$.

(ii) For every finite subset $\cF \subset I$, define $K_\cF \subset L^2(B_\cF)$ as the orthogonal complement in $L^2(B_\cF)$ of all $B_{\cF'}$ where $\cF' \subset \cF$ is a proper subset. By convention, $K_\emptyset = \C 1$.
Since \eqref{eq.commuting-square-with-BF} is a commuting square, $L^2(B)$ is the orthogonal direct sum of the subspaces $K_\cF$, $\cF \subset I$ finite. Define $a_\cF \in L^2(N) \ot K_\cF$ as the projection of $a \in N \ovt B$ onto $L^2(N) \ot K_\cF$. Since $a$ commutes with $\theta(\Lambda)$, we find that
$$\|a_{\delta(g) \cdot \cF}\|_2 = \|a_\cF\|_2 \quad\text{for all finite subsets $\cF \subset I$ and $g \in \Lambda$.}$$
Since the sum of all $\|a_\cF\|_2^2$ is finite and the action $\delta(\Lambda) \actson I$ has infinite orbits, it follows that $a_\cF = 0$ for all finite nonempty subsets $\cF \subset I$. So, $a = a_\emptyset$, meaning that $a \in N \ot 1$.
\end{proof}

\subsection{General coarse embeddings for Bernoulli like crossed products}

One of the key results in \cite{PV21} is \cite[Theorem 3.2]{PV21} determining all possible embeddings between left-right Bernoulli crossed products with groups in the class $\cC$. One of our key technical ingredients is to determine all possible \emph{coarse} embeddings from such a left-right Bernoulli crossed product to a \emph{tensor product} of left-right Bernoulli crossed products. Since we moreover also want to include cocycle twists, we work in the more general context of Bernoulli like actions introduced in Section \ref{sec.bernoulli-like}.

We say that a group homomorphism $\delta : \Gamma \times \Gamma \to \Gamma_1 \times \Gamma_1$ is \emph{symmetric} if there exists a group homomorphism $\delta_0 : \Gamma \to \Gamma_1$ such that either $\delta(g,h) = (\delta_0(g),\delta_0(h))$ for all $g,h \in \Gamma$, or $\delta(g,h) = (\delta_0(h),\delta_0(g))$ for all $g,h \in \Gamma$.

\begin{theorem}\label{thm.coarse-embedding-Bernoulli-like}
Let $\Gamma_1,\ldots,\Gamma_k$ be groups in the class $\cC$. Let $\Gamma_i \times \Gamma_i \actson (B_i,\tau)$ be Bernoulli like actions over the left-right action $\Gamma_i \times \Gamma_i \actson \Gamma_i$, with all $B_i$ amenable and nontrivial. Let $\mu_i \in Z^2(\Gamma_i \times \Gamma_i,\T)$ and consider $M_i = B_i \rtimes_{\mu_i} (\Gamma_i \times \Gamma_i)$.

Let $\Gamma$ be a nonamenable icc group and $(A_0,\tau_0)$ a nontrivial amenable tracial von Neumann algebra. Consider the left-right Bernoulli action $\Gamma \times \Gamma \actson (A,\tau) = (A_0,\tau_0)^\Gamma$ and write $M = A \rtimes (\Gamma \times \Gamma)$.

Let $t > 0$ and $\psi : M \to (M_1 \ovt \cdots \ovt M_k)^t$ a coarse embedding. Then $t \in \N$ and $\psi$ is unitarily conjugate to a direct sum of embeddings of the form $\psi_0 : M \to M_n(\C) \ovt M_1 \ovt \cdots \ovt M_k$ satisfying
\begin{equation}
\begin{split}\label{eq.my.special.form}
& \psi_0(u_r) = \pi(r) \ot u_{\delta_1(r)} \ot \cdots \ot u_{\delta_k(r)} \quad\text{for all $r \in \Gamma \times \Gamma$,} \\
& \psi_0(\pi_e(A_0)) \subset M_n(\C) \ovt B_{1,e} \ovt \cdots \ovt B_{k,e} \; ,
\end{split}
\end{equation}
where $\pi$ is a projective representation with $2$-cocycle $\om_\pi$ and $\delta_i : \Gamma \times \Gamma \to \Gamma_i \times \Gamma_i$ are faithful symmetric homomorphisms such that $\om_\pi \, (\mu_1 \circ \delta_1) \, \cdots (\mu_k \circ \delta_k) = 1$.

If for every $i \in \{1,\ldots,k\}$ and every irreducible subfactor $P \subset M_i$ of infinite index, $\psi(M) \not\prec M_1 \ovt \cdots \ovt M_{i-1} \ovt P \ovt M_{i+1} \ovt \cdots \ovt M_k$, then the homomorphisms $\delta_i$ above are all isomorphisms.
\end{theorem}

Note that it follows automatically from \eqref{eq.my.special.form} where $\psi_0(\pi_s(A_0))$ is located for arbitrary $s \in \Gamma$. More precisely, defining $\al_i : \Gamma \to \Gamma_i$ and $\eps_i \in \{\pm 1\}$ such that either $\delta_i(g,h) = (\al_i(g),\al_i(h))$, $\eps_i=1$, or $\delta_i(g,h) = (\al_i(h),\al_i(g))$, $\eps_i=-1$, we get that
$$\psi_0(\pi_s(A_0)) \subset M_n(\C) \ovt B_{1,\al_1(s^{\eps_1})} \ovt \cdots \ovt B_{k,\al_k(s^{\eps_k})} \quad\text{for all $s \in \Gamma$.}$$

Before proving Theorem \ref{thm.coarse-embedding-Bernoulli-like}, we need the following general lemma.

\begin{lemma}\label{lem.key}
Let $\Gamma$ be a group in the class $\cC$ and $\Gamma \times \Gamma \actson (B,\tau)$ a Bernoulli like action over the left-right action $\Gamma \times \Gamma \actson \Gamma$, with $B$ amenable. Let $\mu \in Z^2(\Gamma \times \Gamma,\T)$ and write $M = B \rtimes_{\mu} (\Gamma \times \Gamma)$ with subalgebra $S = L_\mu(\Gamma \times \Gamma)$. Let $N$ be a finite factor.

\begin{enumlist}
\item\label{lem.key.i} If $P,Q \subset p (N \ovt M) p$ commute and are both strongly nonamenable relative to $N \ot 1$, then $P \vee Q$ can be unitarily conjugated into $N \ovt L_\mu(\Gamma \times \Gamma)$.
\item\label{lem.key.ii} If $P,Q \subset p (N \ovt S)p$ commute and are both strongly nonamenable relative to $N \ot 1$ and if we denote by $p_l$ and $p_r$ the maximal projections in $P' \cap p(N \ovt S)p$ such that
$$P p_l \prec_{N \ovt S}^f N \ovt L_\mu(\Gamma \times e) \quad\text{and}\quad P p_r \prec_{N \ovt S}^f N \ovt L_\mu(e \times \Gamma) \; ,$$
then $p_l + p_r = p$ and $p_l,p_r \in \cZ(P' \cap p(N \ovt S)p) \cap \cZ(Q' \cap p (N \ovt S)p)$ and
$$Q p_l \prec_{N \ovt S}^f N \ovt L_\mu(e \times \Gamma) \quad\text{and}\quad Q p_r \prec_{N \ovt S}^f N \ovt L_\mu(\Gamma \times e) \; .$$
\item\label{lem.key.iii} If $D \subset p(N \ovt M)p$ is amenable relative to $N \ot 1$ and if the normalizer of $D$ contains commuting von Neumann subalgebras $P,Q$ that are both strongly nonamenable relative to $N \ot 1$, then $D \prec^f_{N \ovt M} N \ovt B$.
\end{enumlist}
\end{lemma}
\begin{proof}
Denote $\be : M \to \cN \ovt \cM$ where $\cN = C \rtimes_\mu (\Gamma \times \Gamma)$ and $\cM = A \rtimes (\Gamma \times \Gamma)$ with $(A,\tau) = (A_0,\tau_0)^\Gamma$ as in \eqref{eq.embedding-beta}. We still denote as $\be$ the embedding $\id \ot \be : N \ovt M \to N \ovt \cN \ovt \cM$.

(i)\ Since $B$ is amenable, $P$ and $Q$ are strongly nonamenable relative to $N \ovt B$. By Lemma \ref{lem.transfer-1.iii}, $\be(P)$ and $\be(Q)$ are strongly nonamenable relative to $N \ovt \cN \ovt A$. A fortiori, $\be(P)$ and $\be(Q)$ are strongly nonamenable relative to $N \ovt \cN \ovt 1$. By \cite[Lemma 5.3]{KV15}, the subalgebra $\be(P \vee Q)$ can be unitarily conjugated into $N \ovt \cN \ovt L(\Gamma \times \Gamma)$. So certainly, $\be(P \vee Q) \prec^f N \ovt \cN \ovt L(\Gamma \times \Gamma)$. By Lemma \ref{lem.transfer-1.ii}, we get that $P \vee Q \prec^f N \ovt L_\mu(\Gamma \times \Gamma)$.

Denote by $\Gamma_d \subset \Gamma \times \Gamma$ the diagonal subgroup. By Corollary \ref{cor.relative-omega-solidity}, we get that $P \vee Q \not\prec N \ovt L_\mu(\Gamma_d)$. As in the proof of \cite[Lemma 5.3]{KV15}, using Lemmas \ref{lem.intertwine-transitivity.i} and \ref{lem.weak-mixing}, and the factoriality of $L_\mu(\Gamma \times \Gamma)$, we find that $P \vee Q$ can be unitarily conjugated into $N \ovt L_\mu(\Gamma \times \Gamma)$.

(ii)\ Define $p_l$ and $p_r$ as in the statement of the lemma. By definition, $p_l,p_r \in \cZ(P' \cap p(N \ovt S)p)$. We similarly define $q_l,q_r \in Q' \cap p (N \ovt S)p$ as the maximal projections such that
$$Q q_l \prec_{N \ovt S}^f N \ovt L_\mu(\Gamma \times e) \quad\text{and}\quad Q q_r \prec_{N \ovt S}^f N \ovt L_\mu(e \times \Gamma) \; .$$
Note that $q_l,q_r \in \cZ(Q' \cap p (N \ovt S)p)$. To conclude the proof of (ii), it suffices to prove that $p_r = q_l$, $p_l = q_r$ and $p_l + p_r = p$.

By Lemma \ref{lem.transfer-2.i}, we have that $\be(P (p-p_l)) \not\prec N \ovt \cN \ovt L(\Gamma \times e)$. Since $Q \subset P' \cap p(N \ovt S)p$, we find that $p-p_l$ commutes with $Q$. By Theorem \ref{thm.relative-omega-solidity.iii}, we find that $\be(Q(p-p_l))$ is amenable relative to $N \ovt \cN \ovt L(\Gamma \times e)$. By Lemma \ref{lem.transfer-2.ii}, we get that $Q(p-p_l)$ is amenable relative to $N \ovt L_\mu(\Gamma \times e)$. Since $Q q_r \prec^f N \ovt L_\mu(e \times \Gamma)$, a fortiori $Q q_r$ is amenable relative to $N \ovt L_\mu(e \times \Gamma)$. Since $p-p_l \in Q' \cap p(N \ovt S)p$, the projections $p-p_l$ and $q_r$ commute. By Lemma \ref{lem.intertwining-intersection.i}, we find that $Q (p-p_l) q_r$ is amenable relative to $N \ot 1$. By our assumption on $Q$, we conclude that $(p-p_l) q_r = 0$.

So, $q_r \leq p_l$. By symmetry, we also find the other inequalities $p_r \leq q_l$ and $q_l \leq p_r$ and $p_l \leq q_r$. So, $p_l = q_r$ and $p_r = q_l$. By Lemma \ref{lem.intertwining-intersection.ii}, $P p_l p_r \prec N \ot 1$, so that the projections $p_l$ and $p_r$ must orthogonal.

It remains to prove that $p_l + p_r = p$. Write $e = p - (p_l + p_r) = p - (q_l + q_r)$. Since $e \leq p-p_l$, we have seen above that $Q e$ is amenable relative to $N \ovt L_\mu(\Gamma \times e)$. Since $e \leq p-p_r$, we similarly have that $Q e$ is amenable relative to $N \ovt L_\mu(e \times \Gamma)$. By Lemma \ref{lem.intertwining-intersection.i}, $Qe$ is amenable relative to $N \ot 1$, so that $e=0$.

(iii)\ Denote by $\cD$ the normalizer of $D$ inside $p(N \ovt M)p$. By point (i), we may assume that $p \in N \ovt S$ and that $P \vee Q \subset p (N \ovt S) p$. We then take $p_l$ and $p_r$ as in point (ii). Assume that $p_l \neq 0$. We claim that $P p_l$ is strongly nonamenable relative to $N \ovt (B \rtimes_\mu (e \times \Gamma))$ inside $N \ovt M$. Assume the contrary. Since $P p_l \prec^f N \ovt L_\mu(\Gamma \times e)$, a fortiori $P p_l$ is amenable relative to $N \ovt L_\mu(\Gamma \times e)$. By Lemma \ref{lem.intertwining-intersection.i}, we arrive at the contradiction that a nonzero corner of $P p_l$ is amenable relative to $N \ot 1$. So the claim is proven.

By symmetry, $Q p_l$ is strongly nonamenable relative to $N \ovt (B \times_\mu (\Gamma \times e))$. Since similar statements hold for $P p_r$ and $Q p_r$ when $p_r \neq 0$, it follows that $\cD$ is strongly nonamenable relative to both $N \ovt (B \rtimes_\mu (\Gamma \times e))$ and $N \ovt (B \rtimes_\mu (e \times \Gamma))$.

By Lemma \ref{lem.transfer-1.iii}, $\be(\cD)$ is strongly nonamenable relative to both $N \ovt \cN \ovt (A \rtimes (\Gamma \times e))$ and $N \ovt \cN \ovt (A \rtimes (e \times \Gamma))$. Applying twice Theorem \ref{thm.relative-omega-solidity.i}, we conclude that $\be(D) \prec^f N \ovt \cN \ovt A$. By Lemma \ref{lem.transfer-1.i}, also $D \prec^f N \ovt B$.
\end{proof}

We are now ready to prove Theorem \ref{thm.coarse-embedding-Bernoulli-like}.

\begin{proof}[{Proof of Theorem \ref{thm.coarse-embedding-Bernoulli-like}}]
We only prove the theorem in the case $k=2$. This does not hide any substantial aspect of the argument, but makes the notations much easier. We write $P = \psi(L(\Gamma \times e))$ and $Q = \psi(L(e \times \Gamma))$. We also write $S_i = L_{\mu_i}(\Gamma_i \times \Gamma_i)$. By coarseness of $\psi$ and Lemma \ref{lem.coarse-gives-nonintertwining.ii}, we get that $P$ and $Q$ are strongly nonamenable relative to both $M_1 \ot 1$ and $1 \ot M_2$.

{\bf Step 1.} After a unitary conjugacy, we have $P \vee Q \subset (S_1 \ovt S_2)^t$. This follows by applying twice Lemma \ref{lem.key.i}.

For the rest of the proof, we realize $(M_1 \ovt M_2)^t$ and $(S_1 \ovt S_2)^t$ explicitly as $p (M_m(\C) \ovt M_1 \ovt M_2)p$, resp.\ $p (M_m(\C) \ovt S_1 \ovt S_2)p$, where $p \in M_m(\C) \ovt S_1 \ovt S_2$ is a projection, and we assume that $\psi(L(\Gamma \times \Gamma)) \subset (S_1 \ovt S_2)^t$.

{\bf Step 2.} Denote by $p_{1l},p_{1r} \in P' \cap (S_1 \ovt S_2)^t$ the maximal projections such that in $S_1 \ovt S_2$,
$$P p_{1l} \prec^f L_{\mu_1}(\Gamma_1 \times e) \ovt S_2 \quad\text{and}\quad P p_{1r} \prec^f L_{\mu_1}(e \times \Gamma_1) \ovt S_2 \; .$$
We similarly denote by $p_{2l},p_{2r} \in P' \cap (S_1 \ovt S_2)^t$ the maximal projections such that in $S_1 \ovt S_2$,
$$P p_{2l} \prec^f S_1 \ovt L_{\mu_2}(\Gamma_2 \times e) \quad\text{and}\quad P p_{2r} \prec^f S_1 \ovt L_{\mu_2}(e \times \Gamma_2) \; .$$
By Lemma \ref{lem.key.ii}, these four projections belong to $\cZ(P' \cap (S_1 \ovt S_2)^t) \cap \cZ(Q' \cap (S_1 \ovt S_2)^t)$ and, in particular, all commute. Writing
$$p_{ll} = p_{1l} p_{2l} \;\; , \;\; p_{rl} = p_{1r} p_{2l} \;\; , \;\; p_{lr} = p_{1l} p_{2r} \;\; , \;\; p_{rr} = p_{1r} p_{2r} \;\; ,$$
it follows from Lemma \ref{lem.key.ii} that $p_{ll} + p_{rl} + p_{lr} + p_{rr} = 1$. By Lemma \ref{lem.intertwining-intersection.ii}, we find that inside $S_1 \ovt S_2$,
$$P p_{ll} \prec^f L_{\mu_1}(\Gamma_1 \times e) \ovt L_{\mu_2}(\Gamma_2 \times e) \quad\text{and}\quad Q p_{ll} \prec^f L_{\mu_1}(e \times \Gamma_1) \ovt L_{\mu_2}(e \times \Gamma_2) \; .$$
We get similar results under $p_{rl}, p_{lr}$ and $p_{rr}$.

{\bf Step 3.} Applying twice Lemma \ref{lem.key.iii} and using Lemma \ref{lem.intertwining-intersection.ii}, we find that $\psi(A) \prec^f B_1 \ovt B_2$.

{\bf Step 4.} Assume that $p_{ll} \neq 0$, that $\Lambda_1 < \Gamma$ is a finite index subgroup and that
\begin{equation}\label{eq.theta-V}
\begin{split}
& q \in M_k(\C) \ovt L_{\mu_1}(\Gamma_1 \times e) \ovt L_{\mu_2}(\Gamma_2 \times e) \quad\text{is a projection,}\\
& \theta : \Lambda_1 \to \cU(q (M_k(\C) \ovt L_{\mu_1}(\Gamma_1 \times e) \ovt L_{\mu_2}(\Gamma_2 \times e))q) \quad\text{is a group homomorphism,}\\
& V \in p_{ll}(M_{m,k}(\C) \ovt S_1 \ovt S_2)q \quad\text{is a nonzero element such that}\\
& \psi(u_{(g,e)}) V = V \theta(g) \quad\text{for all $g \in \Lambda_1$.}
\end{split}
\end{equation}
We prove in step 4 that $h_{\Gamma_1 \times \Gamma_2}(\theta(\Lambda_1)) > 0$. Similar statements hold under each of the projections $p_{ll}, p_{lr}, p_{rl}$ and $p_{rr}$, and also for the intertwining of $\psi(u_{(e,g)})$ instead of $\psi(u_{(g,e)})$, $g \in \Lambda_1$.

Define $K \subset L^2(M_1 \ovt M_2)$ as the closed linear span of $M_1 \ovt S_2$ and $S_1 \ovt M_2$. For every finite subset $\cF \subset (\Gamma_1 \times \Gamma_1) \times (\Gamma_2 \times \Gamma_2)$, denote by $P_\cF$ the orthogonal projection of $L^2(M_1 \ovt M_2)$ onto the closed linear span of $B_1 u_{r_1} \ot B_2 u_{r_2}$, $(r_1,r_2) \in \cF$. The key estimate is to prove that for every sequence $a_n \in L_{\mu_1}(\Gamma_1 \times e) \ovt L_{\mu_2}(\Gamma_2 \times e)$ with $\|a_n\| \leq 1$ for all $n$ and $h_{\Gamma_1 \times \Gamma_2}(a_n) \to 0$, every vector $\xi \in K^\perp$, every sequence $b_n \in S_1 \ovt S_2$ with $\|b_n\| \leq 1$ for all $n$ and every finite subset $\cF \subset (\Gamma_1 \times \Gamma_1) \times (\Gamma_2 \times \Gamma_2)$, we have
\begin{equation}\label{eq.clustering}
\|P_\cF(a_n \xi b_n)\|_2 \to 0 \; .
\end{equation}
The proof of \eqref{eq.clustering} is almost identical to \cite[Lemma 5.10]{KV15}, but given the difference in notations and context, we repeat the argument here.

For every nonempty finite subset $\cF_i \subset \Gamma_i$ and using the convention that $B_{i,\emptyset} = \C 1$, we define
$$B^\circ_{i,\cF_i} = B_{i,\cF_i} \ominus \lspan \{B_{i,\cF'_i} \mid \cF'_i \subset \cF_i \;\;\text{and}\;\; \cF'_i \neq \cF_i \} \; .$$
By construction, the subspaces $B^\circ_{i,\cF_i}$ are orthogonal and densely span $B_i \ominus \C 1$, when $\cF_i$ runs over all nonempty finite subsets $\cF_i \subset \Gamma_i$. So the elements of the form $d_1 u_{r_1} \ot d_2 u_{r_2}$ with $d_i \in B^\circ_{i,\cF_i}$, $\|d_i\| \leq 1$ and $r_i \in \Gamma_i \times \Gamma_i$, form a total subset of $K^\perp$. Since we can replace $b_n \in S_1 \ovt S_2$ by $(u_{r_1} \ot u_{r_2}) b_n$, to prove \eqref{eq.clustering}, it suffices to prove that
\begin{equation}\label{eq.clustering-enough}
\|P_\cF(a_n (d_1 \ot d_2) b_n)\|_2 \to 0
\end{equation}
when $d_i \in B^\circ_{i,\cF_i}$ with $\|d_i\| \leq 1$ and $\cF_i \subset \Gamma_i$ are nonempty finite subsets. For $r_i \in \Gamma_i \times \Gamma_i$, we denote by $P_{(r_1,r_2)}$ the projection $P_{\cF}$ corresponding to the singleton $\cF = \{(r_1,r_2)\}$. Since for any finite set $\cF$, we have $P_\cF = \sum_{(r_1,r_2) \in \cF} P_{(r_1,r_2)}$, we may furthermore assume in \eqref{eq.clustering-enough} that $\cF = \{(r_1,r_2)\}$. Note that
\begin{multline*}
P_{(r_1,r_2)}(a_n(d_1 \ot d_2)b_n) \\= \sum_{k_1 \in \Gamma_1,k_2 \in \Gamma_2} (a_n)_{(k_1,e,k_2,e)} \; (\si_{(k_1,e)}(d_1) \ot \si_{(k_2,e)}(d_2)) \; P_{(r_1,r_2)}((u_{(k_1,e)} \ot u_{(k_2,e)})b_n) \; .
\end{multline*}
Denote by $L_i < \Gamma_i$ the, necessarily finite, subgroup of elements $g_i \in \Gamma_i$ such that $(g_i,e) \cF_i = \cF_i$. Choose sets of representatives $T_i \subset \Gamma_i$ for $\Gamma_i/L_i$. For all $g_1 \in L_1$ and $g_2 \in L_2$, define
$$\xi_{n,g_1,g_2} := \sum_{k_1 \in T_1 g_1, k_2 \in T_2 g_2} (a_n)_{(k_1,e,k_2,e)} \; (\si_{(k_1,e)}(d_1) \ot \si_{(k_2,e)}(d_2)) \; P_{(r_1,r_2)}((u_{(k_1,e)} \ot u_{(k_2,e)})b_n) \; .$$
Since
$$P_{(r_1,r_2)}(a_n(d_1 \ot d_2)b_n) = \sum_{g_1 \in L_1,g_2 \in L_2} \xi_{n,g_1,g_2}$$
and $L_1$, $L_2$ are finite, it suffices to show that for all $g_i \in L_i$, $\|\xi_{n,g_1,g_2}\|_2 \to 0$. But the terms of the sum defining $\xi_{n,g_1,g_2}$ are orthogonal. Thus,
\begin{align*}
&\|\xi_{n,g_1,g_2}\|_2^2 \\
&\hspace{1cm} = \sum_{k_1 \in T_1 g_1, k_2 \in T_2 g_2} |(a_n)_{(k_1,e,k_2,e)}|^2 \; \|\si_{(k_1,e)}(d_1)\|_2^2 \; \|\si_{(k_2,e)}(d_2)\|_2^2 \; |(b_n)_{((k_1^{-1},e)r_1,(k_2^{-1},e)r_2)}|^2 \\
&\hspace{1cm} \leq h_{\Gamma_1 \times \Gamma_2}(a_n)^2 \sum_{k_1 \in T_1 g_1, k_2 \in T_2 g_2}  |(b_n)_{((k_1^{-1},e)r_1,(k_2^{-1},e)r_2)}|^2 \\
&\hspace{1cm} \leq h_{\Gamma_1 \times \Gamma_2}(a_n)^2 \; \|b_n\|_2^2 \leq h_{\Gamma_1 \times \Gamma_2}(a_n)^2 \to 0 \; .
\end{align*}
So, \eqref{eq.clustering} is proven.

Now assume that $p_{ll} \neq 0$, that $\Lambda_1 < \Gamma$ is a finite index subgroup and that $\theta$, $V$ and $q$ are as in \eqref{eq.theta-V}. Assume that $h_{\Gamma_1 \times \Gamma_2}(\theta(\Lambda_1)) = 0$ and take a sequence $g_n \in \Lambda_1$ such that $h_{\Gamma_1 \times \Gamma_2}(\theta(g_n)) \to 0$.

We extend the above notations to $m \times k$ matrices over $M_1 \ovt M_2$. So we replace $K$ by $M_{m,k}(\C) \ot K$ and we still denote by $P_\cF$ the orthogonal projection onto $M_{m,k}(\C) \ot P_\cF(L^2(M_1 \ovt M_2))$. Write $v_n = \psi(u_{(g_n,e)})$. By our assumption that $h_{\Gamma_1 \times \Gamma_2}(\theta(g_n)) \to 0$ and by \eqref{eq.clustering}, we get that
\begin{equation}\label{eq.new-clustering}
\|P_\cF(\theta(g_n) \xi v_n^*)\|_2 \to 0 \quad\text{whenever $\cF \subset (\Gamma_1 \times \Gamma_1) \times (\Gamma_2 \times \Gamma_2)$ is finite and $\xi \in K^\perp$.}
\end{equation}

We know from step 3 that $\psi(A) \prec^f B_1 \ovt B_2$. Since $\psi$ is coarse and $A$ is diffuse, we also know from Lemma \ref{lem.coarse-gives-nonintertwining.i} that $\psi(A) \not\prec M_1 \ot 1$ and $\psi(A) \not\prec 1 \ot M_2$. It then follows from Lemma \ref{lem.intertwining-intersection.ii} that $\psi(A) \not\prec M_1 \ovt S_2$ and $\psi(A) \not\prec S_1 \ovt M_2$. Fix an arbitrary $\eps > 0$.
By e.g.\ \cite[Remark 3.3]{Vae07}, we can choose a unitary $a \in \cU(A)$ such that $\|\psi(a) - P_{K^\perp}(\psi(a))\|_2 < \eps$. Since $\psi(A) \prec^f B_1 \ovt B_2$, we can take a finite subset $\cF \subset (\Gamma_1 \times \Gamma_1) \times (\Gamma_2 \times \Gamma_2)$ such that $\|V^*\psi(b) - P_\cF(V^*\psi(b))\|_2 < \eps$ for all $b \in A$ with $\|b\| \leq 1$.

Write $\xi = V^* P_{K^\perp}(\psi(a))$. Since $V$ is a matrix over $S_1 \ovt S_2$, we have that $\xi \in K^\perp$. It follows from \eqref{eq.new-clustering} that
$$\|P_\cF(\theta(g_n) V^* P_{K^\perp}(\psi(a)) v_n^*)\|_2 \to 0 \; .$$
Since $\|\psi(a) - P_{K^\perp}(\psi(a))\|_2 < \eps$, it follows that
\begin{equation}\label{eq.preceding-uniform}
\limsup_n  \|P_\cF(\theta(g_n) V^* \psi(a) v_n^*)\|_2 \leq \eps \; .
\end{equation}
Writing $a_n = \si_{(g_n,e)}(a)$, we have $\theta(g_n) V^* \psi(a) v_n^* = V^* v_n \psi(a) v_n^* = V^* \psi(a_n)$. It thus follows from \eqref{eq.preceding-uniform} that $\limsup_n  \|P_\cF(V^* \psi(a_n))\|_2 \leq \eps$. But $\|V^*\psi(a_n) - P_\cF(V^*\psi(a_n))\|_2 < \eps$ for all $n$, so that $\limsup_n \|V^* \psi(a_n)\|_2 \leq 2 \eps$. Since $a_n$ is a unitary, we conclude that $\|V^*\|_2 \leq 2 \eps$ for all $\eps > 0$. We arrive at the contradiction $V = 0$.

{\bf Step 5.} If $p_{ll} \neq 0$, we prove that there exists a finite index subgroup $\Lambda < \Gamma$ such that $\Lambda \times \Lambda \ni (g,h) \mapsto \psi(u_{(g,h)}) p_{ll}$ is unitarily conjugate, with a unitary in an amplification of $S_1 \ovt S_2$, to a direct sum of homomorphisms of the form
$$
\Lambda \times \Lambda \ni r \mapsto \pi(r) \ot u_{\delta_1(r)} \ot u_{\delta_2(r)}
$$
where $\pi$ is a finite dimensional projective representation of $\Lambda \times \Lambda$ and the $\delta_i : \Lambda \times \Lambda \to \Gamma_i \times \Gamma_i$ are faithful group homomorphisms of the form $\delta_i(g,h) = (\delta_{i,1}(g),\delta_{i,2}(h))$. Similar statements hold for the summands of $\psi(u_{(g,h)})$ under $p_{lr}, p_{rl}$ and $p_{rr}$.

To prove the statement in step~5, assume that $p_{ll} \neq 0$. Assume that $\Lambda_1 < \Gamma$ is any finite index subgroup and $p_1 \in \psi(L(\Lambda_1 \times \Lambda_1))' \cap p_{ll}(M_m(\C) \ovt S_1 \ovt S_2)p_{ll}$ is a nonzero projection. Below we prove the existence of a smaller finite index subgroup $\Lambda_2 < \Lambda_1$, a projective representation $\pi : \Lambda_2 \times \Lambda_2 \to \cU(n)$, faithful group homomorphisms $\delta_i : \Lambda_2 \times \Lambda_2 \to \Gamma_i \times \Gamma_i$ of the form $\delta_i(g,h) = (\delta_{i,1}(g),\delta_{i,2}(h))$ and an element $V \in p_1(M_{m,n}(\C) \ovt S_1 \ovt S_2)$ such that $V^* V = 1$ and
$$\psi(u_r) V = V (\pi(r) \ot u_{\delta_1(r)} \ot u_{\delta_2(r)}) \quad\text{for all $r \in \Lambda_2 \times \Lambda_2$.}$$
Note as follows that it is indeed sufficient to prove this statement. We apply it a first time to $\Lambda_1 = \Gamma$ and $p_1 = p_{ll}$. Then we apply it to $\Lambda_2$ and the projection $p_{ll} - VV^*$, and so on. Since each of the projections $VV^*$ has integer trace, this process must end in finitely many steps.

Since $\psi(L(\Lambda_1 \times e))p_{ll} \prec^f L_{\mu_1}(\Gamma_1 \times e) \ovt L_{\mu_2}(\Gamma_2 \times e)$, we can take $\theta$, $V$ and $q$ satisfying \eqref{eq.theta-V} and $V = p_1 V$. Replacing $q$ by the support projection of the conditional expectation of $V^* V$ on $M_k(\C) \ovt L_{\mu_1}(\Gamma_1 \times e) \ovt L_{\mu_2}(\Gamma_2 \times e)$, which commutes with $\theta(\Lambda_1)$, we may further assume that $V q_0 \neq 0$ for every nonzero projection $q_0 \in q(M_k(\C) \ovt L_{\mu_1}(\Gamma_1 \times e) \ovt L_{\mu_2}(\Gamma_2 \times e))q$.

By step~4, $\theta(\Lambda_1)q_0$ has positive height as a subgroup of
$$\cU(q_0(M_k(\C) \ovt L_{\mu_1}(\Gamma_1 \times e) \ovt L_{\mu_2}(\Gamma_2 \times e))q_0)$$
for every nonzero projection $q_0$ that commutes with $\theta(\Lambda_1)$. By Lemma \ref{lem.automatic-atomic} and Corollary \ref{cor.height-no-weak-mixing}, we find a nonzero $W \in q(M_{k,k'}(\C) \ovt L_{\mu_1}(\Gamma_1 \times e) \ovt L_{\mu_2}(\Gamma_2 \times e))$, a finite index subgroup $\Lambda_0 < \Lambda_1$, an irreducible projective representation $\pi : \Lambda_0 \to \cU(k')$ and group homomorphisms $\delta_{1,1} : \Lambda_0 \to \Gamma_1$ and $\delta_{2,1} : \Lambda_0 \to \Gamma_2$ such that
$$\theta(g) W = W (\pi(g) \ot u_{(\delta_{1,1}(g),e)} \ot u_{(\delta_{2,1}(g),e)}) \quad\text{for all $g \in \Lambda_0$.}$$
We now replace $V$ by the partial isometry given by the polar decomposition of $VW$. Note that $V$ remains nonzero and that $VV^* \leq p_1$. We replace $\theta$ by $\theta(g) = \pi(g) \ot u_{(\delta_{1,1}(g),e)} \ot u_{(\delta_{2,1}(g),e)}$ for all $g \in \Lambda_0$ and replace $k$ by $k'$. We still have $\psi(u_{(g,e)}) V = V \theta(g)$ for all $g \in \Lambda_0$.

It follows that $\psi(L(\Ker \delta_{1,1} \times e)) \prec 1 \ot M_2$ and $\psi(L(\Ker \delta_{2,1} \times e)) \prec M_1 \ot 1$. By the coarseness of $\psi$ and Lemma \ref{lem.coarse-gives-nonintertwining.i}, the subgroups $\Ker \delta_{i,1}$ must be finite. Since $\Lambda_0$ is an icc group, these kernels must be trivial, so that the group homomorphisms $\delta_{i,1} : \Lambda_0 \to \Gamma$ are faithful.

We now need a precise description of the relative commutant $\theta(\Lambda_0)' \cap M_k(\C) \ovt S_1 \ovt S_2$.
Define the bicharacters
$$\Om_i : \Gamma_i \times \Gamma_i \to \T : \Om_i(g,g') = \mu_i((g,e),(e,g')) \overline{\mu_i}((e,g'),(g,e)) \; .$$
Then define the bicharacter
$$\al : \Lambda_0 \times (\Gamma_1 \times \Gamma_2) \to \T : \al(g,h,h') = \Om_1(\delta_{1,1}(g),h) \; \Om_2(\delta_{2,1}(g),h') \; .$$
By definition, we have that
$$\Ad(u_{(\delta_{1,1}(g),e)} \ot u_{(\delta_{2,1}(g),e)})(u_{(e,h)} \ot u_{(e,h')}) = \al(g,h,h') \, ( u_{(e,h)} \ot u_{(e,h')})$$
for all $g \in \Lambda_0$, $h,h' \in \Gamma_2$. Define the subset $T \subset \Gamma_1 \times \Gamma_2$ by
$$T = \{(h,h') \in \Gamma_1 \times \Gamma_2 \mid \exists v \in M_k(\C) : v \neq 0 \;\;\text{and for all $g \in \Lambda_0$,}\;\; \al(g,h,h') \pi(g) v \pi(g)^* = v \} \; .$$
Since $\pi$ is irreducible, when $(h,h') \in T$, such an invariant element $v$ must be a multiple of a unitary in $M_k(\C)$. It follows that $T$ is a subgroup of $\Gamma_1 \times \Gamma_2$ and that for every $(h,h') \in T$, we can choose a corresponding unitary $v(h,h') \in \cU(k)$. Then $v : T \to \cU(k)$ is a projective representation. We denote its $2$-cocycle as $\om_v$. Define $\om \in Z^2(T,\T)$ by
$$\om((h,h'),(s,s')) = \om_v((h,h'),(s,s')) \; \mu_1((e,h),(e,s)) \; \mu_2((e,h'),(e,s')) \; .$$
Then
$$\Phi : L_\om(T) \to M_k(\C) \ovt S_1 \ovt S_2 : \Phi(u_{(h,h')}) = v(h,h') \ot u_{(e,h)} \ot u_{(e,h')}$$
is a well-defined embedding. By construction, $\theta(\Lambda_0)$ and $\Phi(T)$ commute. We claim that $\theta(\Lambda_0)' \cap M_k(\C) \ovt S_1 \ovt S_2 = \Phi(L_\om(T))$.

Since the group homomorphisms $\delta_{i,1}$ are faithful, their images are nonamenable subgroups of $\Gamma_i$ and hence, relatively icc. It follows that
$$\theta(\Lambda_0)' \cap M_k(\C) \ovt S_1 \ovt S_2 \subset M_k(\C) \ovt L_{\mu_1}(e \times \Gamma_1) \ovt L_{\mu_2}(e \times \Gamma_2) \; .$$
Then the claim follows from the construction of $T$.

Note that the commutator subgroup $[\Gamma_1,\Gamma_1] \times [\Gamma_2,\Gamma_2] \subset T$. Since $[\Gamma_i,\Gamma_i]$ is a nonamenable subgroup of $\Gamma_i$, it is relatively icc. In particular, $T$ is an icc group and $L_\om(T)$ is a II$_1$ factor.

Consider the von Neumann algebra
$$\cR = \psi(L(\Lambda_0 \times e))' \cap p_1 (M_m(\C) \ovt S_1 \ovt S_2) p_1 \; .$$
Note that $VV^*$ is a projection in $\cR$ and denote by $p_0 \in \cZ(\cR)$ its central support. We can then choose a sequence of partial isometries $a_n \in \cR$ with $a_1 = VV^*$, $a_n^* a_n \leq VV^*$ for all $n$ and $\sum_n a_n a_n^* = p_0$.

Note that $\psi(L(e \times \Lambda_1)) p_1 \subset \cR$. In particular, $p_0$ commutes with $\psi(L(e \times \Lambda_1))$ and $V^* a_n^* a_n V$ commutes with $\theta(\Lambda_0)$. We can thus write $V^* a_n^* a_n V = \Phi(q_n)$, where $q_n \in L_\om(T)$ is a nonzero projection. Note that $\sum_n \tau(q_n) = \tau(p_0)$. Since $L_\om(T)$ is II$_1$ factor, we can choose a projection $q \in M_{k'}(\C) \ot L_\om(T)$ and elements $b_n \in \C^{k'} \ot L_\om(T)$ such that $b_n^* b_n = q_n$ for all $n$ and $\sum_n b_n b_n^* = q$. Then,
$$W = \sum_n a_n V (\id \ot \Phi)(b_n^*) \in M_{m,k'k}(\C) \ovt S_1 \ovt S_2$$
is a partial isometry satisfying $WW^* = p_0$ and $W^* W = (\id \ot \Phi)(q)$.

Since for every $h \in \Lambda_1$, $\psi(u_{(e,h)})$ commutes with $\psi(L(\Lambda_0 \times e))$, we find the group homomorphism
$$\zeta : \Lambda_1 \to \cU(q(M_{k'}(\C) \ot L_\om(T))q) : W^* \psi(u_{(e,h)}) W = (\id \ot \Phi)(\zeta(h)) \quad\text{for all $h \in \Gamma$.}$$
We now want to apply Corollary \ref{cor.height-no-weak-mixing} to the group homomorphism $\zeta$. Define $\cA_0$ as the set of elements $a \in q(M_{k'}(\C) \ot L_\om(T))q$ such that $\lspan\{\zeta(h) a \zeta(h)^* \mid h \in \Lambda_1\}$ is finite dimensional and denote by $\cA$ the weak closure of $\cA_0$. Since $(\id \ot \Phi)(\cA)$ is normalized by $W^* \psi(u_{(e,h)}) W$ for all $h \in \Lambda_1$, it follows from Lemma \ref{lem.automatic-atomic} that $(\id \ot \Phi)(\cA)$ is atomic. So, $\cA$ is atomic.

Since $(\id \ot \Phi)(\zeta(\Lambda_1)) = W^* \psi(L(e \times \Lambda_1)) W$, it follows from step~4 that $(\id \ot \Phi)(\zeta(\Lambda_1)q_0)$ has positive height for every nonzero projection $q_0$ that commutes with $\zeta(\Lambda_1)$. Given the explicit form of $\Phi$, this implies that also $\zeta(\Lambda_1)q_0$ has positive height. Since $W^* \psi(L(e \times \Lambda_1)) W$ has no intertwining with any nontrivial centralizer, a fortiori $\zeta(\Lambda_1)\dpr$ has no intertwining with any nontrivial centralizer. So we can apply Corollary \ref{cor.height-no-weak-mixing} and conclude that $\zeta$ is standard.

This means that we can find an element $X \in q(M_{k',d}(\C) \ot L_\om(T))$ with $X^* X = 1$, a finite index subgroup $\Lambda_0' < \Lambda_1$, a projective representation $\pi' : \Lambda_0' \to \cU(d)$ and a group homomorphism $\eta: \Lambda_0' \to T$ such that
$$\zeta(h) X = X (\pi'(h) \ot u_{\eta(h)}) \quad\text{for all $h \in \Lambda_0'$.}$$
Write $\eta(h) = (\delta_{1,2}(h),\delta_{2,2}(h))$. Then
$$Y = W (\id \ot \Phi)(X) \in p_1 (M_{m,dk}(\C) \ovt S_1 \ovt S_2)$$
is a partial isometry satisfying $Y^* Y = 1$ and by construction
$$\psi(u_{(g,h)}) Y = Y (\pi(g,h) \ot u_{(\delta_{1,1}(g),\delta_{1,2}(h))} \ot u_{(\delta_{2,1}(g),\delta_{2,2}(h))}) \quad\text{for all $g \in \Lambda_0$ and $h \in \Lambda_0'$,}$$
where $\pi : \Lambda_0 \times \Lambda_0' \to \cU(dk)$ is a projective representation.

In the same way as we did for $\delta_{i,1}$, also the group homomorphisms $\delta_{i,2}$ are faithful. So taking $\Lambda_2 = \Lambda_0 \cap \Lambda_0'$, the proof of step~5 is complete.

{\bf Step 6.} By step 5, the amplification $t$ is an integer $n$ and after a unitary conjugacy of the embedding $\psi$, we may assume that $\psi(L(\Gamma \times \Gamma)) \subset M_n(\C) \ovt S_1 \ovt S_2$ and that for some finite index subgroup $\Lambda < \Gamma$, we have
\begin{equation}\label{eq.direct-sum}
\psi(u_{r}) = \bigoplus_j (\pi_j(r) \ot u_{\delta^j_1(r)} \ot u_{\delta^j_2(r)}) \quad\text{for all $r \in \Lambda \times \Lambda$,}
\end{equation}
where each $\pi_j : \Lambda \times \Lambda \to \cU(n_j)$ is a projective representation, $\sum_j n_j = n$ and each $\delta^j_i : \Lambda \times \Lambda \to \Gamma_i \times \Gamma_i$ is a faithful group homomorphism of the form $\delta^j_i(g,h) = (\delta^j_{i,1}(g),\delta^j_{i,2}(h))$ or the form $\delta^j_i(g,h) = (\delta^j_{i,1}(h),\delta^j_{i,2}(g))$.

Write $\cG = (\Gamma_1 \times \Gamma_1) \times (\Gamma_2 \times \Gamma_2)$ and $\delta^j : \Lambda \times \Lambda \to \cG : \delta^j(r) = (\delta^j_1(r),\delta^j_2(r))$. By unitarily conjugating and regrouping direct summands, we may assume that for $j \neq j'$, there is no $s \in \cG$ with $\delta^{j'} = \Ad s \circ \delta^j$. We denote by $p_j$ the projection corresponding to the $j$'th direct summand in \eqref{eq.direct-sum}. We prove in step~6 that $\psi(A)$ commutes with all $p_j$.

Take $j \neq j'$ and $a \in A$. We prove that $p_j \psi(a) p_{j'} = 0$. Since $\delta^j(\Lambda \times \Lambda) \subset \cG$ is relatively icc and since $\delta^{j'}$ and $\delta^j$ are not conjugate, it follows from Lemma \ref{lem.remarks-icc.ii} that for every $s \in \cG$, the set $\{\delta^j(r) \, s \, \delta^{j'}(r)^{-1} \mid r \in \Lambda \times \Lambda\}$ is infinite. So we can take $r_n \in \Lambda \times \Lambda$ such that
\begin{equation}\label{eq.to-infty}
\delta^j(r_n) \, s \, \delta^{j'}(r_n)^{-1} \to \infty \quad\text{for all $s \in \cG$.}
\end{equation}
Take $\eps > 0$ arbitrary. By step~3, we have that $\psi(A) \prec^f B_1 \ovt B_2$. Using the notation $P_\cF$ as in the proof of step~4, take a finite subset $\cF \subset \cG$ such that $\|P_{\cF}(p_j \psi(b) p_{j'}) - p_j \psi(b) p_{j'}\|_2 < \eps$ for all $b \in A$ with $\|b\| \leq 1$. Write $a_n = \si_{r_n}(a)$ and note that
$$p_j \psi(a_n) p_{j'} = \bigl(\pi_j(r_n) \ot u_{\delta^j_1(r_n)} \ot u_{\delta^j_2(r_n)}\bigr) \; p_j \psi(a) p_{j'} \; \bigl(\pi_{j'}(r_n) \ot u_{\delta^{j'}_1(r_n)} \ot u_{\delta^{j'}_2(r_n)}\bigr)^* \; .$$
First note that it follows that $\|p_j \psi(a_n) p_{j'}\|_2 = \|p_j \psi(a) p_{j'}\|_2$ for all $n$. Then note that \eqref{eq.to-infty} implies that
$$\|P_{\cF}(p_j \psi(a_n) p_{j'})\|_2 \to 0 \; .$$
We conclude that $\limsup_n \|p_j \psi(a_n) p_{j'}\|_2 \leq \eps$ and thus $\|p_j \psi(a) p_{j'}\|_2 \leq \eps$ for all $\eps > 0$. So $p_j \psi(a) p_{j'} = 0$.

{\bf Step 7.} Using the notation of step~6, we prove that for every $j$ and $i$, there exists a $k \in \Gamma_i$ such that $\delta^j_{i,2} = \Ad k \circ \delta^j_{i,1}$.

Assume the contrary. By the relative icc property of $\delta^j_{i,1}(\Lambda)$ in $\Gamma_i$ and Lemma \ref{lem.remarks-icc.ii}, we can then take $g_n \in \Lambda$ such that
\begin{equation}\label{eq.use-this}
\delta^j_{i,1}(g_n) \; k \; \delta^j_{i,2}(g_n)^{-1} \to \infty \quad\text{whenever $k \in \Gamma_i$.}
\end{equation}
Denote by $\Lambda_d = \{(g,g) \mid g \in \Lambda\}$ the diagonal subgroup in $\Lambda \times \Lambda$. For every $h \in \Gamma$, denote by $\pi_h : A_0 \to A$ the embedding as the $h$-th tensor factor. Note that $\psi(\pi_e(A_0)) p_j$ commutes with $\psi(L(\Lambda_d))$. Using that $\delta^j(\Lambda_d) \subset \cG$ is relatively icc and using \eqref{eq.direct-sum} and \eqref{eq.use-this}, it follows from Lemma \ref{lem.further-commutant-result} that $\psi(\pi_e(A_0))p_j \subset M_{n_j}(\C) \ot 1 \ot B_2$ (when $i=1$) or $\psi(\pi_e(A_0))p_j \subset M_{n_j}(\C) \ot B_1 \ot 1$ (when $i=2$). Conjugating with $\psi(u_r)$, $r \in \Lambda \times \Lambda$, and writing $A_1 = (A_0,\tau_0)^\Lambda$, we arrive at $\psi(A_1) p_j \subset M_{n_j}(\C) \ot 1 \ot B_2$ or $\psi(A_1)p_j \subset M_{n_j}(\C) \ot B_1 \ot 1$. Both give a contradiction with the coarseness of $\psi$ and Lemma \ref{lem.coarse-gives-nonintertwining.i}.

{\bf Step 8.} After a further conjugacy by step 7, we may thus assume that $\delta^j_{i,1} = \delta^j_{i,2}$ for all $i$, $j$. We denote this homomorphism as $\al_{j,i} : \Lambda \to \Gamma_i$. We next prove that each of these homomorphisms $\al_{j,i}$ uniquely extends to a faithful homomorphism from $\Gamma$ to $\Gamma_i$. Thus each $\delta^j_i$ uniquely extends to a faithful symmetric homomorphism $\delta^j_i : \Gamma \times \Gamma \to \Gamma_i \times \Gamma_i$.

Note that we may assume that $\Lambda$ is a finite index \emph{normal} subgroup of $\Gamma$. Fix $k \in \Gamma$. Since $\al_{j,i}(\Lambda) \subset \Gamma_i$ is relatively icc, by Lemma \ref{lem.remarks-icc.iii}, we only need to prove that there exists an $s \in \Gamma_i$ such that the set
$$\{\al_{j,i}(k g k^{-1}) \, s \, \al_{j,i}(g)^{-1} \mid g \in \Lambda\}$$
is finite. Assume that this set is infinite for every $s \in \Gamma_i$. We can then take $g_n \in \Lambda$ such that
$$\al_{j,i}(k g_n k^{-1}) \, s \, \al_{j,i}(g_n)^{-1} \to \infty \quad\text{for all $s \in \Gamma_i$.}$$
Noting that $\psi(\pi_k(A_0))p_j$ commutes with $\psi(u_{(kg_n k^{-1},g_n)}) p_j$ for all $n$, the same computation as in step~7 leads to a contradiction.

{\bf Step 9.} Using step~8, we uniquely extend $\delta^j : \Lambda \times \Lambda \to \cG$ to homomorphisms $\delta^j : \Gamma \times \Gamma \to \cG$. We prove that \eqref{eq.direct-sum} now automatically holds for all $r \in \Gamma \times \Gamma$, where $\pi_j : \Gamma \times \Gamma \to \cU(n_j)$ are still projective representations.

Fix $r \in \Gamma \times \Gamma$. Define the unitary $X_r \in M_n(\C) \ovt S_1 \ovt S_2$ by
$$X_r = \psi(u_r) \; \Bigl(\bigoplus_j (1 \ot u_{\delta^j_1(r)} \ot u_{\delta^j_2(r)})\Bigr)^* \; .$$
Also define $D_1 = \bigoplus_j (M_{n_j}(\C) \ot 1 \ot 1)$. Since $\Lambda$ is normal in $\Gamma$, for every $s \in \Lambda \times \Lambda$ and $r \in \Gamma \times \Gamma$, we have
$$\psi(u_s) \psi(u_r) = \psi(u_r) \psi(u_{r^{-1} s r}) \quad\text{and}\quad r^{-1} s r \in \Lambda \times \Lambda \; .$$
Using \eqref{eq.direct-sum}, it follows that for all $j$,
$$\psi(u_s) \psi(u_r) p_j = \psi(u_r) \psi(u_{r^{-1} s r}) p_j = \psi(u_r) p_j \; \bigl(\pi_j(r^{-1} s r) \ot u_{\delta^j_1(r^{-1} s r)} \ot u_{\delta^j_2(r^{-1} s r)}\bigr) \; .$$
Multiplying on the right with $(1 \ot u_{\delta^j_1(r)} \ot u_{\delta^j_2(r)})^*$, it follows that
$$\psi(u_s) X_r p_j \in X_r p_j \bigl(M_{n_j}(\C) \ot u_{\delta^j_1(s)} \ot u_{\delta^j_2(s)}\bigr) \; ,$$
so that
$$\psi(u_s) X_r \in X_r \, D_1 \, \psi(u_s) \quad\text{for all $s \in \Lambda$.}$$
We conclude that $\psi(u_s) X_r \psi(u_s)^* \in X_r \, D_1$ for all $s \in \Lambda$. On the other hand, it follows from \eqref{eq.direct-sum} and Lemma \ref{lem.remarks-icc} that $(\Ad \psi(u_s))_{s \in \Lambda \times \Lambda}$ is weakly mixing on the orthogonal complement of $D_1$ inside $L^2(M_n(\C) \ovt S_1 \ovt S_2)$. So, for every $r \in \Gamma \times \Gamma$, we must have that $X_r \in D_1$. Denoting $X_r = \bigoplus (\pi_j(r) \ot 1 \ot 1)$, we find that \eqref{eq.direct-sum} holds for all $r \in \Gamma \times \Gamma$.

{\bf Step 10.} Since we already saw in step~6 that also $\psi(A)$ commutes with all the projections $p_j$, we have written $\psi$ as a direct sum of embeddings indexed by $j$. To conclude the proof, we may analyze each of them separately. We may thus assume that $\psi : M \to M_n(\C) \ovt M_1 \ovt M_2$ is an embedding satisfying
$$\psi(u_r) = \pi(r) \ot u_{\delta_1(r)} \ot u_{\delta_2(r)} \quad\text{for all $r \in \Gamma \times \Gamma$,}$$
where $\pi : \Gamma \times \Gamma \to \cU(n)$ is a projective representation and each $\delta_i : \Gamma \times \Gamma \to \Gamma_i \times \Gamma_i$ is a symmetric faithful homomorphism.

Denote by $\om_\pi$ the $2$-cocycle associated with $\pi$. Since $\Gamma \times \Gamma \ni r \mapsto \psi(u_r)$ is multiplicative, we get that $\om_\pi \, (\mu_1 \circ \delta_1) \, (\mu_2 \circ \delta_2) = 1$.

Analyzing $\psi(\pi_e(A_0))$ as in step~8, we find that $\psi(\pi_e(A_0)) \subset M_n(\C) \ovt B_{1,e} \ovt B_{2,e}$. It only remains to prove the last statement of the theorem.

Let $\al_i : \Gamma \to \Gamma_i$ be the injective group homomorphism such that either $\delta_i(g,h) = (\al_i(g),\al_i(h))$ or $\delta_i(g,h) = (\al_i(h),\al_i(g))$ for all $g,h \in \Gamma$. Write $\Gamma'_i = \al_i(\Gamma)$. Assume that $\Gamma'_1$ is a proper subgroup of $\Gamma_1$. Define $P_1 \subset M_1$ as the von Neumann subalgebra generated by $B_{1,e}$ and $u_{(g,h)}$, $g,h \in \Gamma'_1$. Then $P_1 \subset M_1$ is an irreducible subfactor and the diffuse von Neumann subalgebra generated by $B_{1,k}$, $k \in \Gamma_1 \setminus \Gamma'_1$ is orthogonal to $P_1$. So, $P_1 \subset M_1$ has infinite index. By definition, $\psi(M) \subset M_n(\C) \ovt P_1 \ovt M_2$. In particular, $\psi(M) \prec P_1 \ovt M_2$. This concludes the proof of the theorem.
\end{proof}

\subsection{Coarse embeddings for cocycle twisted wreath products}

In the special case where the Bernoulli like crossed products in Theorem \ref{thm.coarse-embedding-Bernoulli-like} are all cocycle twisted wreath product groups $L_{\mu_i}((\Z/2\Z)^{(\Gamma_i)} \rtimes (\Gamma_i \times \Gamma_i))$, the conclusion of Theorem \ref{thm.coarse-embedding-Bernoulli-like} can be strengthened in the following way, to provide a truly complete classification of all coarse embeddings.

As above, we say that a group homomorphism $\delta : \Gamma \times \Gamma \to \Gamma_i \times \Gamma_i$ is symmetric if either $\delta(g,h) = (\al(g),\al(h))$ or $\delta(g,h) = (\al(h),\al(g))$ for all $g,h \in \Gamma$, where $\al : \Gamma \to \Gamma_i$.

\begin{theorem}\label{thm.coarse-embedding-cocycle-twisted-wreath-product}
Let $\Gamma_1,\ldots,\Gamma_k$ be groups in the class $\cC$ and define the generalized wreath product $G_i = (\Z/2\Z)^{(\Gamma_i)} \rtimes (\Gamma_i \times \Gamma_i)$ over the left-right action $\Gamma_i \times \Gamma_i \actson \Gamma_i$. Let $\mu_i \in Z^2(G_i,\T)$ be any $2$-cocycle and write $M_i = L_{\mu_i}(G_i)$

Let $\Gamma$ be any nonamenable icc group and define $G = (\Z/2\Z)^{(\Gamma)} \rtimes (\Gamma \times \Gamma)$. Let $\mu \in Z^2(G,\T)$ be any $2$-cocycle and write $M = L_\mu(G)$.

Let $t > 0$ and $\psi : M \to (M_1 \ovt \cdots \ovt M_k)^t$ a coarse embedding. Then $t \in \N$ and $\psi$ is unitarily conjugate to a direct sum of embeddings of the form
$$L_\mu(G) \to M_n(\C) \ot L_{\mu_1}(G_1) \ovt \cdots \ovt L_{\mu_k}(G_k) : u_g \mapsto \pi(g) \ot u_{\delta(g)} \quad\text{for all $g \in G$,}$$
where $\pi : G \to \cU(n)$ is a projective representation with $2$-cocycle $\om_\pi$ and $\delta : G \to G_1 \times \cdots \times G_k$ is a $k$-tuple of faithful group homomorphisms $\delta_i : G \to G_i$ of the following form:
\begin{itemlist}
\item $\delta_i(\pi_e(a)) = \pi_e(a)$ for all $a \in \Z/2\Z$, where $\pi_e$ is the embedding in position $e$~;
\item $\delta_i$ restricts to a faithful symmetric group homomorphism $\Gamma \times \Gamma \to \Gamma_i \times \Gamma_i$~;
\item we have $\mu = \om_\pi \, (\mu_1 \circ \delta_1) \, \cdots \, (\mu_k \circ \delta_k)$.
\end{itemlist}
If for every $i \in \{1,\ldots,k\}$ and every irreducible subfactor $P \subset M_i$ of infinite index, $\psi(M) \not\prec M_1 \ovt \cdots \ovt M_{i-1} \ovt P \ovt M_{i+1} \ovt \cdots \ovt M_k$, then the homomorphisms $\delta_i$ are all isomorphisms.
\end{theorem}

Note that $\delta_i$ is completely determined by the description above. Defining $\al_i : \Gamma \to \Gamma_i$ and $\eps_i \in \{\pm 1\}$ such that either $\delta_i(g,h) = (\al_i(g),\al_i(h))$, $\eps_i = 1$, or $\delta_i(g,h) = (\al_i(h),\al_i(g))$, $\eps_i = -1$, we get that $\delta_i(\pi_s(a)) = \pi_{\al_i(s^{\eps_i})}(a)$ for all $a \in \Z/2\Z$ and $s \in \Gamma$.

\begin{proof}
We start by proving the theorem in the case where $\mu=1$.
Write $B_i = L_{\mu_i}((\Z/2\Z)^{(\Gamma_i)})$. Write $K = (\Z/2\Z)^k$ and denote
$$\zeta_0 : K \to (\Z/2\Z)^{(\Gamma_1)} \times \cdots \times (\Z/2\Z)^{(\Gamma_k)} : \zeta_0(a) = (\pi_e(a_1),\ldots,\pi_e(a_k)) \; .$$
Define $\eta \in Z^2(K,\T)$ by $\eta = (\mu_1 \times \cdots \times \mu_k) \circ \zeta_0$. We then still denote by $\zeta_0$ the embedding $L_\eta(K) \to B_1 \ovt \cdots \ovt B_k$.

By Theorem \ref{thm.coarse-embedding-Bernoulli-like}, we have that $t \in \N$ and that we can write $\psi$ as a direct sum of embeddings of a special form. We analyze each summand separately and may thus assume that $\psi : M \to M_n(\C) \ovt M_1 \ovt \cdots \ovt M_k$ is of the following form.
\begin{itemlist}
\item $\psi(u_{(g,h)}) = \pi(g,h) \ot u_{\delta(g,h)}$ for all $(g,h) \in \Gamma \times \Gamma$, where $\delta$ is a $k$-tuple of symmetric homomorphisms $\delta_i : \Gamma \times \Gamma \to \Gamma_i \times \Gamma_i$ and $\om_\pi \, (\mu_1 \circ \delta_1) \, \cdots (\mu_k \circ \delta_k) = 1$.
\item $\psi(u_{\pi_e(1)}) \in M_n(\C) \ot \zeta_0(L_\eta(K))$.
\end{itemlist}
Moreover, if the final assumption in the formulation of the theorem holds, then all $\delta_i$ are isomorphisms.

To prove the theorem in the case where $\mu=1$, it suffices to prove that we automatically have that
\begin{equation}\label{eq.main-job}
\psi(u_{\pi_e(1)}) \in \cU(n) \ot u_{\pi_e(1)} \ot \cdots \ot u_{\pi_e(1)} \; .
\end{equation}
Indeed, once we have proven this property, we extend $\delta_i$ to a group homomorphism $\delta_i : G \to G_i$ by putting $\delta_i(\pi_e(1)) = \pi_e(1)$ and find that $\psi(u_g) \in \cU(n) \ot u_{\delta(g)}$ for all $g \in G$. This forces $\psi(u_g) = \pi(g) \ot u_{\delta(g)}$ for all $g \in G$, where $\pi : G \to \cU(n)$ is a projective representation satisfying all the required conclusions.

Write $v_0 = u_{\pi_e(1)} \in L((\Z/2\Z)^{\Gamma}) \subset M$. We have $\psi(v_0) = (\id \ot \zeta_0)(v)$ with $v \in M_n(\C) \ot L_\eta(K)$. To prove \eqref{eq.main-job}, we have to show that
\begin{equation}\label{eq.actual-job}
v \in \cU(n) \ot u_{(1,\ldots,1)} \; .
\end{equation}
For every $a \in K$ and $g \in \Gamma$, we define $v_a,v_{g,a} \in M_n(\C)$ by
\begin{equation}\label{eq.def-va}
v_a = (\id \ot \tau)(v (1 \ot u_a^*)) = (\id \ot \tau)(\psi(v_0)(1 \ot \zeta_0(u_a)^*)) \quad\text{and}\quad v_{g,a} = \pi(g,e) v_a \pi(g,e)^* \; .
\end{equation}
We thus have to prove that $v_a = 0$ for all $a \in K \setminus \{(1,\ldots,1)\}$.

{\bf Step 1.} We prove that for all $g,h \in \Gamma$ and $a,b \in K$, we have that $v_{g,a} v_{h,b} = \pm v_{h,b} v_{g,a}$.

For every $g \in \Gamma$, we write $\zeta_g : L_\eta(K) \to B_1 \ovt \cdots \ovt B_k : \zeta_g = \Ad(u_{\delta(g,e)}) \circ \zeta_0$. Whenever $\cF \subset \Gamma$, we denote by $B_\cF \subset B_1 \ovt \cdots \ovt B_k$ the von Neumann subalgebra generated by all $\zeta_g(L_\eta(K))$, $g \in \cF$. We denote by $E_{\cF} : B_1 \ovt \cdots \ovt B_k \to B_\cF$ the unique trace preserving conditional expectation. Note that $E_{\cF}(\zeta_g(b)) = \tau(b)1$ whenever $b \in L_\eta(K)$ and $g \not\in \cF$.

As before, denote by $(\si_{(g,h)})_{(g,h) \in \Gamma \times \Gamma}$ the left-right Bernoulli action of $\Gamma \times \Gamma$ on $L((\Z/2\Z)^{\Gamma})$. Note that $\psi(\si_{(g,e)}(v_0)) = (\Ad \pi(g,e) \ot \zeta_g)(v)$. Therefore,
$$v_{g,a} = (\id \ot \tau)(\psi(\si_{(g,e)}(v_0)) (1 \ot \zeta_g(u_a)^*)) \quad\text{for all $g \in \Gamma$, $a \in K$.}$$
When $g \neq h$ are distinct elements in $\Gamma$ and $a,b \in K$ are arbitrary, we make the following computation, where in the first equality, we use that $\tau \circ E_{\{g\}} = \tau$.
\begin{equation}\label{eq.to-repeat}
\begin{split}
(\id \ot \tau)\bigl(\psi(\si_{(g,e)}(v_0)) & \, \psi(\si_{(h,e)}(v_0)) \, (1 \ot \zeta_h(u_b)^*) \, (1 \ot \zeta_g(u_a)^*) \bigr) \\
 &= (\id \ot \tau)\bigl(\psi(\si_{(g,e)}(v_0)) \, (v_{h,b} \ot \zeta_g(u_a))^*\bigr) = v_{g,a} \, v_{h,b} \; .
\end{split}
\end{equation}
The group $T = (\Z/2\Z)^{(\Gamma_1)} \times \cdots \times (\Z/2\Z)^{(\Gamma_k)}$ is commutative and $B_1 \ovt \cdots \ovt B_k$ is its twisted group von Neumann algebra w.r.t.\ the $2$-cocycle $\mu_1 \times \cdots \times \mu_k$. Therefore, $u_c u_d = \Om(c,d) u_d u_c$ for all $c,d \in T$, where $\Om : T \times T \to \T$ is a bicharacter. Since every nontrivial element of $T$ has order $2$, we have that $\Om(c,d) \in \{\pm 1\}$ for all $c,d \in T$, so that $u_c u_d = \pm u_d u_c$ for all $c,d \in T$. Then also $\zeta_g(u_a) \, \zeta_h(u_b) = \pm \, \zeta_h(u_b) \, \zeta_g(u_a)$.

Because $\psi(\si_{(g,e)}(v_0)) \, \psi(\si_{(h,e)}(v_0)) = \psi(\si_{(h,e)}(v_0)) \, \psi(\si_{(g,e)}(v_0))$, we repeat the computation in \eqref{eq.to-repeat} and conclude that $v_{g,a} \, v_{h,b} = \pm \, v_{h,b} \, v_{g,a}$ when $g \neq h$.

To conclude step 1, fix $g,h \in \Gamma$ and $a,b \in K$. By compactness of $\cU(n)$, we can take a sequence $g_n \in \Gamma$ such that $g_n \to \infty$ and $\Ad \pi(g_n,e) \to \id$. Then, $v_{g_ng,a} = (\Ad \pi(g_n))(v_{g,a}) \to v_{g,a}$. For all $n$ large enough, we have $g_n g \neq h$. So for all $n$ large enough, $v_{g_ng,a} v_{h,b} = \pm \, v_{h,b} v_{g_ng,a}$. But then also $v_{g,a} \, v_{h,b} = \pm \, v_{h,b} \, v_{g,a}$.

{\bf Step 2.} We prove that there exist projections $p_1,\ldots,p_m \in M_n(\C)$ and a finite index subgroup $\Lambda_1 < \Gamma$ such that
\begin{itemlist}
\item $\sum_{j=1}^m p_j = 1$,
\item for all $j \in \{1,\ldots,m\}$, we have that $p_j$ commutes with $\pi(h,e)$ for all $h \in \Lambda_1$ and with $v_{g,a}$ for all $g \in \Gamma$ and $a \in K$,
\item for all $j \in \{1,\ldots,m\}$, $g \in \Gamma$, $a \in K$, the matrix $v_{g,a} p_j$ is either zero or a multiple of a self-adjoint unitary.
\end{itemlist}

Since for every $a \in K$, we have $a^2 = e$, $u_a^*$ is a multiple of $u_a$ and it follows from \eqref{eq.def-va} that all $v_{g,a}$ are multiples of self-adjoint matrices. We can thus define $s_{g,a} = v_{g,a}^* v_{g,a} = v_{g,a} v_{g,a}^*$ and note that it follows from step~1 that $\{s_{g,a} \mid g \in \Gamma , a \in K\}$ is a family of commuting self-adjoint matrices, which moreover commute with all $v_{h,b}$. By definition, this family is normalized by $(\Ad \pi(g,e))_{g \in \Gamma}$. It now suffices to define the $p_j$ as the minimal projections in the von Neumann algebra generated by the $s_{g,a}$ and to define $\Lambda_1 < \Gamma$ as the subgroup of $h \in \Gamma$ such that $\pi(h,e)$ commutes with all $p_j$.

{\bf Step 3.} We prove that there exists an infinite subgroup $\Lambda_2 < \Gamma$ such that $v_{g,a} = v_a$ for all $g \in \Lambda_2$ and $a \in K$.

Fix $j \in \{1,\ldots,m\}$. Applying Lemma \ref{lem.projective-order-2} below to the family of self-adjoint unitaries that arise as nonzero multiples of $v_{g,a} p_j$, $g \in \Gamma$, $a \in K$, we find a finite index subgroup $\Lambda_{0,j} < \Lambda_1$ such that $\pi(h,e)$ commutes with $v_{g,a} p_j$ for all $h \in [\Lambda_{0,j},\Lambda_{0,j}]$, $g \in \Gamma$ and $a \in K$. Define $\Lambda_0 = \bigcap_{j=1}^m \Lambda_{0,j}$. Then $\Lambda_0 < \Lambda_1$ is still of finite index and $[\Lambda_0,\Lambda_0] \subset \bigcap_{j=1}^m [\Lambda_{0,j},\Lambda_{0,j}]$. So, for all $h \in [\Lambda_0,\Lambda_0]$ and all $g \in \Gamma$, $a \in K$ and $j \in \{1,\ldots,m\}$, we have that $\pi(h,e)$ commutes with $v_{g,a} p_j$. Summing over $j$, we find that $\pi(h,e)$ commutes with $v_{g,a}$.

Defining $\Lambda_2 = [\Lambda_0,\Lambda_0]$, we find in particular that $v_{g,a} = v_a$ for all $g \in \Lambda_2$ and $a \in K$. Since $\Gamma$ is nonamenable, also $\Lambda_0$ is nonamenable and hence, $\Lambda_2$ is nonamenable. In particular, $\Lambda_2$ is infinite.

{\bf Step 4.} We conclude the proof by showing that \eqref{eq.actual-job} holds. We thus have to show that $v_a = 0$ for all $a \in K \setminus \{(1,\ldots,1)\}$. By symmetry, it suffices to prove that $v_a = 0$ for all $a \in \{0\} \times (\Z/2\Z)^{k-1}$, because we can make a similar reasoning in each of the coordinates.

Write $K_0 = \{0\} \times (\Z/2\Z)^{k-1}$. Define $\cB \subset 1 \ovt B_2 \ovt \cdots \ovt B_k$ as the von Neumann subalgebra generated by $\zeta_g(L_\eta(K_0))$, $g \in \Lambda_2$. Define $\cA = L((\Z/2\Z)^{(\Lambda_2)}) \subset L((\Z/2\Z)^{(\Gamma)})$, which is generated by $\si_{(g,e)}(v_0)$, $g \in \Lambda_2$.

Since $\psi$ is a coarse embedding, there is a normal positive functional $\om$ on $\cA \ovt (M_n(\C) \ot \cB)\op$ satisfying
$$\om(a \ot (S \ot b)\op) = (\Tr \ot \tau)(\psi(a)(S \ot b)) \quad\text{for all $a \in \cA$, $b \in \cB$ and $S \in M_n(\C)$.}$$
We define the action $(\be_g)_{g \in \Lambda_2}$ on $\cA \ovt (M_n(\C) \ot \cB)\op$ by $\be_g = \si_{(g,e)} \ot (\id \ot \Ad u_{\delta(g,e)})\op$ and prove that $\om \circ \be_g = \om$ for all $g \in \Lambda_2$.

Whenever $\cF \subset \Lambda_2$, define $\cB_\cF \subset \cB$ as the von Neumann subalgebra generated by $\zeta_g(L_\eta(K))$, $g \in \cF$. By density, it suffices to prove that
\begin{align}\label{eq.my-step-here}
\om(x & \ot  (S \ot y)\op) = \om(\si_{(g,e)}(x) \ot (S \ot u_{\delta(g,e)} y u_{\delta(g,e)}^*)\op) \\
&\text{when}\;\; x = \si_{(h_1,e)}(v_0) \cdots \si_{(h_k,e)}(v_0) \;\;\text{and}\;\; y = \zeta_{h_k}(u_{b_m})^* \cdots \zeta_{h_1}(u_{b_1})^* y_0 \notag\\
&\text{with $h_1,\ldots,h_k \in \Lambda_2$ distinct, $b_1,\ldots,b_k \in K_0$ and $y_0 \in \cB_{\Lambda_2 \setminus \{h_1,\ldots,h_k\}}$.}\notag
\end{align}
Write $\cF_j = \{h_1,\ldots,h_j\}$ for all $1 \leq j \leq k$. Using the conditional expectations $E_\cF$ introduced in step~1, we find that the left hand side of \eqref{eq.my-step-here} equals
\begin{align*}
& (\Tr \ot (\tau \circ E_{\cF_k})) (\psi(\si_{(h_1,e)}(v_0) \cdots \si_{(h_k,e)}(v_0)) (S \ot \zeta_{h_k}(u_{b_k})^* \cdots \zeta_{h_1}(u_{b_1})^* y_0)) \\
&\hspace{0.5ex} = \tau(y_0) \; (\Tr \ot (\tau \circ E_{\cF_{k-1}}))(\psi(\si_{(h_1,e)}(v_0) \cdots \si_{(h_k,e)}(v_0)) (S \ot \zeta_{h_k}(u_{b_k})^* \cdots \zeta_{h_1}(u_{b_1})^*)) \\
&\hspace{0.5ex} = \tau(y_0) \; (\Tr \ot (\tau \circ E_{\cF_{k-2}}))(\psi(\si_{(h_1,e)}(v_0) \cdots \si_{(h_{k-1},e)}(v_0)) (v_{h_k,b_k} S \ot \zeta_{h_{k-1}}(u_{b_{k-1}})^* \cdots \zeta_{h_1}(u_{b_1})^*)) \\
&\hspace{0.5ex} = \cdots = \Tr(v_{h_1,b_1} \cdots v_{h_k,b_k} S) \, \tau(y_0) \; .
\end{align*}
Using step~3, it thus follows that the left hand side of \eqref{eq.my-step-here} equals
$$\Tr(v_{b_1} \cdots v_{b_k} S) \, \tau(y_0) \; .$$
By definition,
\begin{align*}
& \si_{(g,e)}(x) = \si_{(gh_1,e)}(v_0) \cdots \si_{(gh_k,e)}(v_0) \quad\text{and}\\
& u_{\delta(g,e)} y u_{\delta(g,e)}^* = \zeta_{gh_k}(u_{b_k})^* \cdots \zeta_{gh_1}(u_{b_1})^* \, y_1 \quad\text{with}\quad y_1 = u_{\delta(g,e)} y_0 u_{\delta(g,e)}^* \in \cB_{\Lambda_2 \setminus \{gh_1,\ldots,gh_k\}} \; .
\end{align*}
So the same computation says that the right hand side of \eqref{eq.my-step-here} equals
$$\Tr(v_{b_1} \cdots v_{b_k} S) \, \tau(u_{\delta(g,e)} y_0 u_{\delta(g,e)}^*) =  \Tr(v_{b_1} \cdots v_{b_k} S) \, \tau(y_0) \; .$$
So we have proven that $\om \circ \be_g = \om$ for all $g \in \Lambda_2$.

Since $\Lambda_2$ is an infinite group, the action $\si_{(g,e)} \ot (\Ad u_{\delta(g,e)})\op$ on $\cA \ovt \cB\op$ is ergodic, so that $\tau \ot \tau\op$ is the unique normal state on $\cA \ovt \cB\op$ that is invariant under this action. It follows that $\tau \ot (\Tr \ot \tau)\op$ is the unique normal $(\be_g)_{g \in \Lambda_2}$-invariant functional on $\cA \ovt (M_n(\C) \ot \cB)\op$ whose restriction to $(M_n(\C) \ot 1)\op$ equals $\Tr\op$. Since by definition, the restriction of $\om$ to $(M_n(\C) \ot 1)\op$ equals $\Tr\op$, we conclude that $\om = \tau \ot (\Tr \ot \tau)\op$.

In particular,
$$\Tr(v_a S) = (\Tr \ot \tau)(\psi(v_0)(S \ot \zeta_0(u_a)^*)) = \om(v_0 \ot (S \ot \zeta_0(u_a)^*)\op) = \tau(v_0) \, \Tr(S) \, \tau(u_a^*) = 0$$
for all $a \in K_0$ and $S \in M_n(\C)$, because $\tau(v_0) = 0$.

So \eqref{eq.actual-job} is proven, which concludes the proof of the theorem in the case where $\mu=1$.

In the general case, take a coarse embedding $\psi : L_\mu(G) \to (M_1 \ovt \cdots \ovt M_k)^t$. Define the coarse embedding
$$\Delta : L(G) \to L_\mu(G) \ovt L_\mu(G)\op : \Delta(u_g) = u_g \ot \overline{u_g} \quad\text{for all $g \in G$.}$$
By Lemma \ref{lem.stable-coarse}, $\Psi = (\psi \ot \psi\op) \circ \Delta$ is a coarse embedding to which the case $\mu=1$ can be applied.

It follows that $\Psi$ is a finite direct sum of irreducible embeddings, so that also the relative commutant of $\psi(L_\mu(G))$ must be finite dimensional. We may thus assume that $\psi$ is irreducible. We realize $(M_1 \ovt \cdots \ovt M_k)^t = p (M_d(\C) \ovt M_1 \ovt \cdots \ovt M_k) p$. By the case $\mu = 1$, we find
$$X \in (p \ot p\op)(M_{d,m}(\C) \ovt M_1 \ovt \cdots \ovt M_k \ovt \C^d \ovt M_1\op \ovt \cdots \ovt M_k\op) \quad\text{with $X^* X = 1$,}$$
a projective representation $\pi_0 : G \to \cU(m)$ and a $2k$-tuple of faithful group homomorphisms $\delta_i : G \to G_i$ of the form as described in the theorem, such that
$$(\psi(u_g) \ot \overline{\psi(u_g)}) X = X (\pi_0(g) \ot u_{\delta_1(g)} \ot \cdots \ot u_{\delta_{k}(g)} \ot \overline{u_{\delta_{k+1}(g)}} \ot \cdots \ot \overline{u_{\delta_{2k}(g)}}) \quad\text{for all $g \in G$.}$$
Write $\delta(g) = (\delta_1(g),\ldots,\delta_k(g))$ and $\delta'(g) = (\delta_{k+1}(g),\ldots,\delta_{2k}(g))$. Define the projective representations
\begin{align*}
& \Phi : G \to \cU(p (\C^n \ot L^2(M_1 \ovt \cdots \ovt M_k))) : \Phi(g) \xi = \psi(u_g) \xi u_{\delta(g)}^* \; ,\\
& \Phi' : G \to \cU(p (\C^n \ot L^2(M_1 \ovt \cdots \ovt M_k))) : \Phi'(g) \xi = \psi(u_g) \xi u_{\delta'(g)}^* \; .
\end{align*}
The isometry $X$ shows that the projective representation $\Phi \ot \overline{\Phi'}$ admits a nonzero invariant subspace. So also $\Phi$ cannot be weakly mixing. We then find an irreducible projective representation $\pi : G \to \cU(n)$ and a nonzero element $Y \in p(M_{d,n}(\C) \ovt M_1 \ovt \cdots \ovt M_k)$ such that
$$\psi(u_g) Y = Y (\pi(g) \ot u_{\delta(g)}) \quad\text{for all $g \in G$.}$$
By the irreducibility of $\psi$, we may assume that $YY^*=1$. Since $\delta(G)$ is relatively icc in $G_1 \times \cdots \times G_k$, we first find that $Y^* Y \in M_n(\C) \ot 1$ and then that $Y^* Y = 1$ by the irreducibility of $\pi$. So, $\psi$ is unitarily conjugate to an embedding of the required form.
\end{proof}

\begin{lemma}\label{lem.projective-order-2}
Let $J \subset \cU(n)$ be a set of unitary matrices and let $\pi : \Lambda_1 \to \cU(n)$ be a projective representation of a countable group $\Lambda_1$ such that
$$\pi(g) J \pi(g)^* = J \quad\text{for all $g \in \Lambda_1$, and}\quad w_1 w_2 \in \T \cdot w_2 w_1 \quad\text{for all $w_1,w_2 \in J$.}$$
There then exists a finite index subgroup $\Lambda_0 < \Lambda_1$ such that for all $g$ in the commutator subgroup $[\Lambda_0,\Lambda_0]$ and all $w \in J$, we have $\pi(g) w = w \pi(g)$.
\end{lemma}
\begin{proof}
Define the abelian subgroup $W < \PU(n) = \cU(n)/\T \cdot 1$ generated by $w \T$, $w \in J$. Then $(\Ad \pi(g))_{g \in \Lambda_1}$ induces an action $(\al_g)_{g \in \Lambda_1}$ on $W$. Choosing a lift, we get the projective representation $\zeta : W \to \cU(n)$. By our assumptions, we find a bicharacter $\Om : W \times W \to \T$ such that $\zeta(w) \zeta(w') = \Om(w,w') \zeta(w') \zeta(w)$ for all $w,w' \in W$. Note that $\Om(w,w') = \overline{\Om}(w',w)$ for all $w \in W$.

It also follows from our assumptions that $\pi(g) \zeta(w) \pi(g)^* \in \T \cdot \zeta(\al_g(w))$ for all $g \in \Lambda_1$ and $w \in W$. It follows that $\Om(\al_g(w),\al_g(w')) = \Om(w,w')$ for all $g \in \Lambda_1$ and $w,w' \in W$.

By Lemma \ref{lem.finite-type-cocycles}, the subgroup $W_0 = \{w \in W \mid \forall w' \in W : \Om(w,w')=1\}$ has finite index in $W$. Since $\al_g(W_0)=W_0$ for all $g \in \Lambda_1$, $(\al_g)_{g \in \Lambda_1}$ induces an action of $\Lambda_1$ on the finite quotient $W/W_0$. We define the finite index subgroup $\Lambda_2 < \Lambda_1$ so that this action is trivial for all $g \in \Lambda_2$.

Note that $\zeta(W_0)$ is an abelian group of unitaries that commute with all $\zeta(w)$, $w \in W$. Then $B = \lspan \zeta(W_0)$ is an abelian $*$-subalgebra of $M_n(\C)$ that commutes with all $\zeta(w)$, $w \in W$. Since $\Ad \pi(g)$ normalizes $\zeta(W_0) \T$, also $\pi(g)B\pi(g)^* = B$ for all $g \in \Lambda_1$. Denote by $p_1,\ldots,p_k$ the minimal projections in $B$. Define the finite index subgroup $\Lambda_3 < \Lambda_1$ such that $\pi(g)$ commutes with $p_i$ for all $g \in \Lambda_3$ and $i \in \{1,\ldots,k\}$. We write $\Lambda_0 = \Lambda_2 \cap \Lambda_3$.

By construction, for every $w \in W$, $\zeta(w)$ commutes with all $p_i$ and $\zeta(w) p_i \in \T p_i$ when $w \in W_0$. Fix $i \in \{1,\ldots,k\}$, $g \in \Lambda_0$ and $w \in W$. Since $\al_g$ induces the identity automorphism on $W/W_0$, we find $w_0 \in W_0$ such that $\al_g(w) = w w_0$. So,
$$\pi(g) \, \zeta(w) p_i \, \pi(g)^* \in \T \cdot \zeta(\al_g(w)) p_i = \T \cdot \zeta(w w_0) p_i = \T \cdot \zeta(w) \zeta(w_0) p_i = \T \cdot \zeta(w) p_i \; .$$
We thus find for every $w \in W$ and every $i \in \{1,\ldots,k\}$, a character $\om_{w,i} : \Lambda_0 \to \T$ such that
$$\pi(g) \, \zeta(w) p_i \, \pi(g)^* = \om_{w,i}(g) \, \zeta(w) p_i$$
for all $g \in \Lambda_0$. So for all $g \in [\Lambda_0,\Lambda_0]$, we get that
$$\pi(g) \, \zeta(w) p_i \, \pi(g)^* = \zeta(w) p_i \quad\text{for all $w \in W$ and $i \in \{1,\ldots,k\}$.}$$
Summing over $i$, we conclude that for $g \in [\Lambda_0,\Lambda_0]$, $\pi(g)$ commutes with all $\zeta(w)$, $w \in W$. Since $J \subset \zeta(W) \T$, we get that $\pi(g)$ commutes with all $w \in J$.
\end{proof}

\subsection{Coarse embeddings for generalized Bernoulli crossed products}

Also in the special case where the Bernoulli like crossed products in Theorem \ref{thm.coarse-embedding-Bernoulli-like} are ordinary generalized Bernoulli crossed products $(A_i,\tau_i)^{\Gamma_i} \rtimes (\Gamma_i \times \Gamma_i)$, the conclusion of Theorem \ref{thm.coarse-embedding-Bernoulli-like} can be strengthened in the following way.

\begin{theorem}\label{thm.coarse-embedding-generalized-Bernoulli}
Let $\Gamma_1,\ldots,\Gamma_k$ be groups in the class $\cC$ and let $(A_i,\tau_i)$ be nontrivial tracial amenable von Neumann algebras. Write $(B_i,\tau) = (A_i,\tau_i)^{\Gamma_i}$, consider the left-right Bernoulli action $\Gamma_i \times \Gamma_i \actson B_i$ and write $M_i = B_i \rtimes (\Gamma_i \times \Gamma_i)$.

Let $\Gamma$ be a nonamenable icc group and $(A_0,\tau_0)$ a nontrivial amenable tracial von Neumann algebra. Consider the left-right Bernoulli action $\Gamma \times \Gamma \actson (A,\tau) = (A_0,\tau_0)^\Gamma$ and write $M = A \rtimes (\Gamma \times \Gamma)$.

Let $t > 0$ and $\psi : M \to (M_1 \ovt \cdots \ovt M_k)^t$ a coarse embedding. Then $t \in \N$ and $\psi$ is unitarily conjugate to a direct sum of embeddings of the form $\psi_0 : M \to M_n(\C) \ovt M_1 \ovt \cdots \ovt M_k$ satisfying
\begin{itemlist}
\item $\psi_0(u_r) = \pi(r) \ot u_{\delta_1(r)} \ot \cdots \ot u_{\delta_k(r)}$ for all $r \in \Gamma \times \Gamma$, where $\pi : \Gamma \times \Gamma \to \cU(n)$ is a unitary representation and the $\delta_i : \Gamma \times \Gamma \to \Gamma_i \times \Gamma_i$ are faithful symmetric homomorphisms,
\item $\psi_0(\pi_e(A_0)) \subset D \ovt \pi_e(A_1) \ovt \cdots \ovt \pi_e(A_k)$, where $D \subset M_n(\C)$ is an abelian $*$-subalgebra that is normalized by $(\Ad \pi(r))_{r \in \Gamma \times \Gamma}$,
\item for every $a \in A_0$ and $i \in \{1,\ldots,k\}$, we have that $(\id^{\ot i} \ot \tau \ot \id^{\ot (k-i)})\psi_0(\pi_e(a)) = \tau(a) 1$.
\end{itemlist}
If for every $i \in \{1,\ldots,k\}$ and every irreducible subfactor $P \subset M_i$ of infinite index, $\psi(M) \not\prec M_1 \ovt \cdots \ovt M_{i-1} \ovt P \ovt M_{i+1} \ovt \cdots \ovt M_k$, then the homomorphisms $\delta_i$ above are all isomorphisms.
\end{theorem}

\begin{proof}
We repeat a few words from the start of the proof of Theorem \ref{thm.coarse-embedding-cocycle-twisted-wreath-product}. By Theorem \ref{thm.coarse-embedding-Bernoulli-like}, we have that $t \in \N$ and that we can write $\psi$ as a direct sum of embeddings of a special form. We again analyze each summand separately and may thus assume that $\psi : M \to M_n(\C) \ovt M_1 \ovt \cdots \ovt M_k$ is of the following form.
\begin{itemlist}
\item $\psi(u_r) = \pi(r) \ot u_{\delta(r)}$ for all $r \in \Gamma \times \Gamma$, where $\delta$ is a $k$-tuple of symmetric homomorphisms $\delta_i : \Gamma \times \Gamma \to \Gamma_i \times \Gamma_i$,
\item $\psi(\pi_e(A_0)) \subset M_n(\C) \ovt \pi_e(A_1) \ovt \cdots \ovt \pi_e(A_k)$.
\end{itemlist}
Moreover, if the final assumption in the formulation of the theorem holds, then all $\delta_i$ are isomorphisms. For simplicity of notation, we will assume that none of the $\delta_i$ involves a flip, i.e.\ $\delta_i(g,h) = (\al_i(g),\al_i(h))$ where $\al_i : \Gamma \to \Gamma_i$ are faithful homomorphisms.

Note that for every $g \in \Gamma$, $\psi(\pi_g(A_0)) \subset M_n(\C) \ovt \pi_{\al_1(g)}(A_0) \ovt \cdots \ovt \pi_{\al_k(g)}(A_0)$. Denote by $D_g \subset M_n(\C)$ the smallest $*$-algebra such that
\begin{equation}\label{eq.live-in-these-coordinates}
\psi(\pi_g(A_0)) \subset D_g \ovt \pi_{\al_1(g)}(A_0) \ovt \cdots \ovt \pi_{\al_k(g)}(A_0) \; .
\end{equation}
Note that $D_g$ equals the algebra generated by the elements
$$(\id \ot \om_1 \ot \cdots \ot \om_k)(\psi(\pi_g(a))) \quad\text{with $\om_i \in (M_i)_*$ and $a \in A_0$,}$$
and that $\pi(g,h) D_k \pi(g,h)^* = D_{gkh^{-1}}$ for all $g,h,k \in \Gamma$. We claim that $(D_g)_{g \in \Gamma}$ is a commuting family of abelian $*$-subalgebras of $M_n(\C)$ that thus generate an abelian $*$-subalgebra $D \subset M_n(\C)$ satisfying $\pi(g,h) D \pi(g,h)^* = D$ for all $g,h \in \Gamma$. To prove this claim, we first take $g \neq h$ in $\Gamma$ and show that $D_g$ and $D_h$ commute. Take $a,a' \in A_0$ and $\om_i,\om'_i \in (M_i)_*$. We have that
$$\psi(\pi_g(a)) \, \psi(\pi_h(a)) = \psi(\pi_h(a)) \, \psi(\pi_g(a)) \; .$$
Using \eqref{eq.live-in-these-coordinates} and applying the functionals $\om_i$ in the coordinates $\al_i(g)$ and $\om'_i$ in the coordinates $\al_i(h)$, which are all distinct because $g \neq h$, it follows that
$$(\id \ot \om_1 \ot \cdots \ot \om_k)(\psi(\pi_g(a))) \quad\text{commutes with}\quad (\id \ot \om'_1 \ot \cdots \ot \om'_k)(\psi(\pi_h(a'))) \; .$$
So the self-adjoint family of matrices generating $D_g$ commutes with the self-adjoint family of matrices generating $D_h$. We conclude that $D_g$ commutes with $D_h$.

Since $\pi(g,h) D_k \pi(g,h)^* = D_{gkh^{-1}}$ for all $g,h,k \in \Gamma$, all $D_g$ are isomorphic. So if one of the $D_g$ is nonabelian, all $D_g$ are nonabelian and we would find an infinite family of commuting nonabelian $*$-subalgebras of $M_n(\C)$, which is impossible. So all $D_g$ are abelian and the claim is proven.

It remains to prove that $(\id^{\ot i} \ot \tau \ot \id^{\ot (k-i)})\psi(\pi_e(a)) = \tau(a) 1$ for every $a \in A_0$ and $i \in \{1,\ldots,k\}$. By symmetry, it suffices to consider the case $i=1$.

Fix an arbitrary minimal projection $p \in D$. Define the finite index subgroup $\Gamma_0 < \Gamma$ such that $\pi(g,e) p = p \pi(g,e)$ for all $g \in \Gamma_0$. Define the coarse embedding $\gamma : A \to B_1 \ovt \cdots \ovt B_k$ such that $\psi(a)(p \ot 1^{\ot k}) = p \ot \gamma(a)$ for all $a \in A$. Define $\gamma_0 : A_0 \to A_1 \ovt \cdots \ovt A_k$ such that $\gamma \circ \pi_e = (\pi_e \ot \cdots \ot \pi_e) \circ \gamma_0$. Note that for all $g \in \Gamma_0$, we have $\gamma \circ \pi_g = (\pi_{\al_1(g)} \ot \cdots \ot \pi_{\al_k(g)}) \circ \gamma_0$.

Define $C_0 = A_2 \ovt \cdots \ovt A_k$ and write $(C,\tau) = (C_0,\tau)^{\Gamma_0}$. Define the embedding $\eta : C \to B_2 \ovt \cdots \ovt B_k$ such that $\eta(\pi_g(c)) = (\pi_{\al_2(g)} \ot \cdots \ot \pi_{\al_k(g)})(c)$ for all $g \in \Gamma_0$, $c \in C_0$. We also define $\cA \subset A$ by $(\cA,\tau) = (A_0,\tau_0)^{\Gamma_0}$. Since the embedding $\gamma$ is coarse, we can define a normal state $\Om$ on $\cA \ovt C\op$ satisfying
$$\Om(a \ot b\op) = \tau(\gamma(a)(1 \ot \eta(b))) \quad\text{for all $a \in \cA$, $b \in C$.}$$
By restricting $\Om$ to $\pi_e(A_0) \ovt (\pi_e(A_2) \ovt \cdots \ovt \pi_e(A_k))\op$, we also have the normal state $\Om_0$ on $A_0 \ovt C_0\op$ satisfying $\Om_0(a \ot b\op) = \tau(\gamma_0(a)(1 \ot b))$.

If $g_1,\ldots,g_m \in \Gamma_0$ are distinct elements and $a_1,\ldots,a_m \in A_0$, $b_1,\ldots,b_m \in C_0$, we consider
$$a = \pi_{g_1}(a_1) \cdots \pi_{g_m}(a_m) \in \cA \quad\text{and}\quad b = \pi_{g_1}(b_1) \cdots \pi_{g_m}(b_m) \in C \; .$$
A direct computation then gives
$$\Om(a \ot b\op) = \prod_{i=1}^m \Om_0(a_i \ot b_i\op) \; .$$
It follows that $\Om$ is invariant under the diagonal Bernoulli action of $\Gamma_0$ on $\cA \ovt C\op$. Since this action is ergodic, it follows that $\Om = \tau$. This implies that $(\tau \ot \id^{\ot (k-1)})\gamma_0(a) = \tau(a) 1$ for all $a \in A_0$. Since the minimal projection $p \in D$ was arbitrary, it also follows that
$$(\id \ot \tau \ot \id^{\ot (k-1)})\psi(\pi_e(a)) = \tau(a) 1$$
for every $a \in A_0$.
\end{proof}

\section{Proof of Theorem \ref{thm.A}}

In this section, we prove Theorem \ref{thm.A}. We actually prove a more precise result in Section \ref{sec.proof-theorem-A} below. To prepare for the proof, we need several preliminary results.

\subsection{The icc decomposition of a twisted group von Neumann algebra}

Let $\Lambda$ be a countable group and $\om \in Z^2(\Lambda,\T)$. We denote by $\Lambda\fc$ the virtual center of $\Lambda$, i.e.\ the set of all elements $g \in \Lambda$ having a finite conjugacy class, which always form a normal subgroup of $\Lambda$. Then the center of $L_\om(\Lambda)$ is contained in $L_\om(\Lambda\fc)$, so that $\cZ(L_\om(\Lambda))$ equals the space of $(\Ad u_g)_{g \in \Lambda}$-invariant elements $\cZ(L_\om(\Lambda\fc))$.

When $\Lambda$ is a countable group with \emph{finite} virtual center, the center $\cZ(L_\om(\Lambda))$ is finite dimensional, so that $L_\om(\Lambda)$ is a direct sum of factors. We make this direct sum decomposition more concrete and prove in Proposition \ref{prop.decompose-virtual-center} that
\begin{equation}\label{eq.decompose-direct-sum-icc}
L_\om(\Lambda) \cong \bigoplus_{i=1}^l \bigl(M_{n_i}(\C) \ot L_{\om_{z_i}}(\Lambda_{z_i} / \Lambda\fc)\bigr) \; ,
\end{equation}
where $\{z_1,\ldots,z_l\}$ are minimal projections in $\cZ(L_\om(\Lambda\fc))$ such that $z_i \leq q_i$ with $\{q_1,\ldots,q_l\}$ being the minimal projections of $\cZ(L_\om(\Lambda))$ and where, for each minimal projection $z \in \cZ(L_\om(\Lambda\fc))$, $\Lambda_{z} < \Lambda$ is a finite index subgroup containing $\Lambda\fc$ and $\om_{z} \in H^2(\Lambda_{z}/\Lambda\fc,\T)$. The integers $n_i$ are such that $n_i = d_i k_i$ with $L_\om(\Lambda\fc) q_i \cong M_{d_i}(\C) \ot \C^{k_i}$ and also $k_i = [\Lambda : \Lambda_{z_i}]$.

Since $\Lambda_z / \Lambda\fc$ is icc, the decomposition \eqref{eq.decompose-direct-sum-icc} writes $L_\om(\Lambda)$ as a direct sum of twisted group von Neumann algebras of icc groups, provided that $\Lambda\fc$ is finite.

\begin{proposition}\label{prop.decompose-virtual-center}
Let $\Lambda$ be a countable group with finite virtual center $\Lambda\fc$ and $\om \in Z^2(\Lambda,\T)$. Let $z \in \cZ(L_\om(\Lambda\fc))$ be a minimal projection.
\begin{enumlist}
\item Then $\Lambda_z = \{g \in \Lambda \mid u_g z = z u_g\}$ is a finite index subgroup of $\Lambda$ that contains $\Lambda\fc$.
\end{enumlist}
Choose a projective representation $\pi : \Lambda_z \to \cU(L_\om(\Lambda\fc)z)$ such that $\Ad u_g z = \Ad \pi(g)$ on $L_\om(\Lambda\fc)z$ for all $g \in \Lambda_z$. We write $\rho : \Lambda_z \to \cU(z L_\om(\Lambda) z) : \rho_g = u_g \pi(g)^* z$. Choose a lift $\phi : \Lambda_z/\Lambda\fc \to \Lambda_z$.
\begin{enumlist}[resume]
\item Then the map $\Lambda_z / \Lambda\fc \to \cU(z L_\om(\Lambda) z) : g \mapsto \rho_{\phi(g)}$ is a projective representation. We denote by $\om_z$ its $2$-cocycle.
\item For every minimal projection $p \in L_\om(\Lambda\fc) z$, the map $g \mapsto \rho_{\phi(g)} p$ realizes an isomorphism $p L_\om(\Lambda) p \cong L_{\om_z}(\Lambda_z / \Lambda\fc)$.
\item Denoting by $q \in \cZ(L_\om(\Lambda))$ the unique minimal projection with $z \leq q$, we have that $L_\om(\Lambda) q \cong M_n(\C) \ot L_{\om_z}(\Lambda_z / \Lambda\fc)$, where $n = d k$ and $L_\om(\Lambda\fc) q \cong M_d(\C) \ot \C^k$ and $k = [\Lambda:\Lambda_z]$.
\end{enumlist}
In particular, if $\{q_1,\ldots,q_l\}$ are the minimal projections of $\cZ(L_\om(\Lambda))$ and if we choose minimal projections $z_i \in \cZ(L_\om(\Lambda\fc))q_i$, the isomorphism \eqref{eq.decompose-direct-sum-icc} holds.
\end{proposition}
\begin{proof}
Since $\cZ(L_\om(\Lambda)) \subset \cZ(L_\om(\Lambda\fc))$, there is a unique minimal projection $q \in \cZ(L_\om(\Lambda))$ such that $z \leq q$. Then $(\Ad u_g)_{g \in \Lambda}$ defines an ergodic action of $\Lambda$ on the finite dimensional algebra $L_\om(\Lambda\fc)q$, so that $L_\om(\Lambda\fc)q \cong M_d(\C) \ot \C^k$ for some integers $d,k$. Defining $\Lambda_z$ as stated, we have $[\Lambda:\Lambda_z] = k$.

Also note that $L_\om(\Lambda\fc)z \cong M_d(\C)$. Since $L_\om(\Lambda\fc) z$ is a matrix algebra, the action $(\Ad u_g z)_{g \in \Lambda_z}$ on $L_\om(\Lambda\fc) z$ is inner and we can choose the projective representation $\pi$ as stated. Denote $\rho_g = u_g \pi(g)^* z$. When $g,h \in \Lambda_z$, the unitary $u_h \pi(h)^* z$ commutes with $\pi(g)^* z \in L_\om(\Lambda\fc) z$. Therefore,
\begin{equation}\label{eq.good-formula-for-rho-g}
\rho_g \rho_h = u_g \pi(g)^* z \, u_h \pi(h)^* z = u_g u_h \, \pi(h)^* \pi(g)^* z = \om(g,h) \, \overline{\om_\pi}(g,h) \, \rho_{gh} \; ,
\end{equation}
so that $\rho$ is a projective representation.

By definition, if $g \in \Lambda\fc$, $\rho_g \in \T z$. So for every lift $\phi : \Lambda_z/\Lambda\fc \to \Lambda_z$, the map $\Lambda_z / \Lambda\fc \to \cU(z L_\om(\Lambda) z) : g \mapsto \rho_{\phi(g)}$ is a projective representation. We denote by $\om_z$ its $2$-cocycle. Fix a minimal projection $p \in L_\om(\Lambda\fc) z$. By construction, for every $g \in \Lambda_z$, $\rho_g$ commutes with $L_\om(\Lambda\fc)z$, so that
$$\theta : \Lambda_z/\Lambda\fc \to \cU(p L_\om(\Lambda) p) : g \mapsto \rho_{\phi(g)} p$$
is a well-defined projective representation with the same $2$-cocycle $\om_z$. Denote by $\tau$ the canonical tracial state on $L_\om(\Lambda)$. Let $E : L_\om(\Lambda) \to L_\om(\Lambda\fc)$ be the unique trace preserving conditional expectation. Note that $E(u_g) = 0$ for all $g \in \Lambda \setminus \Lambda\fc$. It follows that for all $g \in \Lambda_z \setminus \Lambda\fc$,
$$E(\rho_g) = E(u_g \pi(g)^* z) = E(u_g) \pi(g)^* z = 0 \; ,$$
so that also $E(\theta(g)) = E(\rho_{\phi(g)} p) = E(\rho_{\phi(g)}) p = 0$ for all $g \in \Lambda_z / \Lambda\fc$ with $g \neq e$. In particular, $\tau(\theta(g)) = 0$ for all $g \in \Lambda_z / \Lambda\fc$ with $g \neq e$. It follows that $\theta$ uniquely extends to a faithful normal $*$-homomorphism from $L_{\om_z}(\Lambda_z / \Lambda\fc)$ to $p L_\om(\Lambda) p$. To conclude the proof of (iii), it suffices to show that the elements $\rho_g p$, $g \in \Lambda_z$, span a strongly dense subspace of $p L_\om(\Lambda) p$. Since $z u_g z = 0$ for all $g \in \Lambda \setminus \Lambda_z$, the elements $u_g z$, $g \in \Lambda_z$, span a strongly dense subspace of $z L_\om(\Lambda) z$. So, the subspaces $\rho_g L_\om(\Lambda\fc) z$, $g \in \Lambda_z$, span a strongly dense subspace of $z L_\om(\Lambda) z$. Multiplying on the left and on the right with $p$ and using that $p$ is a minimal projection in $L_\om(\Lambda\fc) z$, it follows that the elements $\rho_g p$, $g \in \Lambda_z$, span a strongly dense subspace of $p L_\om(\Lambda) p$.

Finally note that the statements in (iv) hold by construction.
\end{proof}

\begin{lemma}\label{lem.cohom}
Let $\Lambda$ be a countable group with finite virtual center $\Lambda\fc$ and $\om \in Z^2(\Lambda,\T)$. Let $z \in \cZ(L_\om(\Lambda\fc))$ be a minimal projection. Define $\Lambda_z < \Lambda$ and $\om_z \in Z^2(\Lambda_z/\Lambda\fc,\T)$ as in Proposition \ref{prop.decompose-virtual-center}. Denote by $q : \Lambda_z \to \Lambda_z / \Lambda\fc$ the quotient map. Let $\Lambda_1 < \Lambda_z /\Lambda\fc$ be a subgroup and $\mu \in Z^2(\Lambda_1,\T)$.

The $2$-cocycle $(\om_z)|_{\Lambda_1} \, \mu$ on $\Lambda_1$ is of finite type if and only if there exists a finite dimensional projective representation $\theta : q^{-1}(\Lambda_1) \to \cU(n)$ with $2$-cocycle $\om_\theta(g) = \om(g) \, \mu(q(g))$ for all $g \in q^{-1}(\Lambda_1)$ such that the restricted $\omega$-representation $\theta_0 : \Lambda\fc \to \cU(n)$ satisfies $\theta_0(z) \neq 0$.
\end{lemma}

\begin{proof}
Write $\Lambda_0 = q^{-1}(\Lambda_1)$. As in Proposition \ref{prop.decompose-virtual-center}, we choose a projective representation $\pi : \Lambda_z \to \cU(d)$ such that $L_\om(\Lambda\fc)z \cong M_d(\C)$ and $\Ad u_g = \Ad \pi(g)$ on $L_\om(\Lambda\fc)z$ for all $g \in \Lambda_z$.

To prove the first implication, assume that $\vphi : \Lambda_1 \to \cU(n)$ is a projective representation with $2$-cocycle $\om_\vphi = (\om_z)|_{\Lambda_1} \, \mu$. Then $\theta(g) = \vphi(q(g)) \ot \pi(g)$ for all $g \in \Lambda_0$ defines a finite dimensional projective representation of $\Lambda_0$ with $2$-cocycle
$$\om_\theta = (\om_\vphi \circ q)|_{\Lambda_0} \; \om_\pi|_{\Lambda_0} = ((\om_z \circ q) \; \om_\pi)|_{\Lambda_0} \; (\mu \circ q)|_{\Lambda_0} \; .$$
For $g \in \Lambda\fc$, we have that $\theta(g)$ is a multiple of $1 \ot u_g z$. To conclude the proof of the first implication, it thus suffices to prove that
\begin{equation}\label{eq.we-need-this-cohomology}
\om_z \circ q \quad\text{is cohomologous to}\quad \om|_{\Lambda_z} \, \overline{\om_\pi} \quad\text{in $Z^2(\Lambda_z,\T)$.}
\end{equation}
Defining $\rho_g = u_g \pi(g)^* z$ for $g \in \Lambda_z$ and choosing a lift $\phi : \Lambda_z/\Lambda\fc \to \Lambda_z$, we have by definition that $\om_z$ is the $2$-cocycle of the projective representation $g \mapsto \rho_{\phi(g)}$. Since $\rho_g \in \T z$ for all $g \in \Lambda\fc$, we find a map $c : \Lambda_z \to \T$ such that $\rho_{\phi(q(g))} = c(g) \rho_g$ for all $g \in \Lambda_z$. It now follows from \eqref{eq.good-formula-for-rho-g} that $c$ realizes the cohomology of \eqref{eq.we-need-this-cohomology}.

To prove the converse, assume that $\theta : \Lambda_0 \to \cU(n)$ is a projective representation with $2$-cocycle $\om_\theta = \om|_{\Lambda_0} \, (\mu \circ q)$ such that $\theta(z) \neq 0$. We linearize $\theta$ to a unital $*$-homomorphism $\theta : \C_{\om_{\theta}}[\Lambda_0]$ on the twisted group algebra. Since $\Lambda\fc < \Lambda_0$ and since the restriction of $\om_\theta$ to $\Lambda\fc$ equals $\om$, we may view $\C_{\om}[\Lambda\fc]$ as a unital $*$-subalgebra of $\C_{\om_{\theta}}[\Lambda_0]$.

Since $\om_\theta(g,h) = \om(g,h)$ when at least one of the $g,h$ belong to $\Lambda\fc$, the equality $u_g z = z u_g$ holds in $\C_{\om_{\theta}}[\Lambda_0]$ for all $g \in \Lambda_0$. So, $\theta(g)$ commutes with $\theta(z)$ for all $g \in \Lambda_0$. Reducing $\theta$ with the projection $\theta(z)$, we may assume that $\theta(z)=1$.

Then $\gamma(g) = \theta(u_g \pi(g)^* z)$ for $g \in \Lambda_0$ defines a finite-dimensional projective representation of $\Lambda_0$ with the property that $\gamma(g)$ is a multiple of $1$ whenever $g \in \Lambda\fc$. Choosing a lift $\phi : \Lambda_1 \to \Lambda_0$, it follows that $g \mapsto \gamma(\phi(g))$ is a finite-dimensional projective representation of $\Lambda_1$ with $2$-cocycle $\om_z \, \mu$.
\end{proof}

\subsection{The triple comultiplication of a twisted group von Neumann algebra}

A key ingredient in our approach is that every twisted group von Neumann algebra $N = L_\om(\Lambda)$ not only admits the canonical coarse embedding $\Delta_3 : N \to N \ovt N\op \ovt N$ given by the triple comultiplication $\Delta_3(u_g) = u_g \ot \overline{u_g} \ot u_g$, but also that for every nonzero central projection $z \in \cZ(N)$, we can restrict $\Delta_3$ to a nonzero coarse embedding from $N z_0$ to a corner of $Nz \ovt (Nz)\op \ovt Nz$ with $z_0 \in \cZ(N) z$.

\begin{proposition}\label{prop.triple-comult}
Let $\Lambda$ be a countable group and $\om \in Z^2(\Lambda,\T)$. Write $N = L_\om(\Lambda)$. Define
$$\Delta_3 : N \to N \ovt N\op \ovt N : \Delta_3(u_g) = u_g \ot \overline{u_g} \ot u_g \quad\text{for all $g \in \Lambda$.}$$
Let $z \in \cZ(N)$ be a nonzero central projection.
\begin{enumlist}
\item We have $p := \Delta_3(z) (z \ot z\op \ot z) \neq 0$.
\item If $z_0 \in \cZ(N)z$ is the smallest central projection such that $\Delta_3(z-z_0) (z \ot z\op \ot z)=0$,
$$\Delta_z : N z_0 \to p(N \ovt N\op \ovt N)p$$
is a coarse embedding.
\item If $\Lambda\fc$ is infinite, then $\Delta_z(Nz_0)' \cap p(N \ovt N\op \ovt N)p$ is infinite dimensional.
\end{enumlist}
\end{proposition}
\begin{proof}
We decompose $z = \sum_{g \in \Lambda} (z)_g u_g$. We also define $\Delta_2 : L(\Lambda) \to N\op \ovt N : \Delta_2(u_g) = \overline{u_g} \ot u_g$ for all $g \in \Lambda$. For all $g,k \in \Lambda$, using that $z=z^*$, we get that
\begin{align*}
(\tau \ot \tau \ot \tau)((z \ot z\op \ot z) (1 \ot \Delta_2(u_k)) \Delta_3(u_g)) & = \tau(zu_g) \, \tau(u_g^* u_k^* z) \, \tau(z u_k u_g) \\
& = \overline{\tau(u_g^* z)} \, \tau((u_ku_g)^* z) \, \overline{\tau((u_k u_g)^* z)} \\
& = \overline{(z)_g} \; |(z)_{kg}|^2 \; .
\end{align*}
Multiplying with $(z)_g$ and summing over $g$, we find that for all $k \in \Lambda$,
\begin{equation}\label{eq.formula-trace}
(\tau \ot \tau \ot \tau)((z \ot z\op \ot z) (1 \ot \Delta_2(u_k)) \Delta_3(z)) = \sum_{g \in \Lambda} |(z)_g|^2 \, |(z)_{kg}|^2 \; .
\end{equation}
In particular, we find that
$$(\tau \ot \tau \ot \tau)(\Delta_3(z) (z \ot z\op \ot z)) = \sum_{g \in \Lambda} |(z)_g|^4 > 0 \; .$$
So, $\Delta_3(z) (z \ot z\op \ot z) \neq 0$.

Since $\Delta_3$ is a coarse embedding, also $\Delta_z$ is a coarse embedding.

Now assume that $\Lambda\fc$ is infinite. Write $\al : N \to N \ovt L(\Lambda) : \al(u_g) = u_g \ot u_g$ for all $g \in \Lambda$, and note that $\Delta_3 = (\id \ot \Delta_2) \circ \al$. It follows that for all $a \in \cZ(L(\Lambda))$, the element $1 \ot \Delta_2(a)$ commutes with $\Delta_3(Nz) = (\id \ot \Delta_2)\al(Nz)$. Since $\Delta_2(a)$ also commutes with $z\op \ot z$, it follows that the linear map
$$V : \cZ(L(\Lambda)) \to N \ovt N\op \ovt N : V(a) = (z \ot z\op \ot z) (1 \ot \Delta_2(a)) \Delta_3(z)$$
takes values in $\Delta_z(Nz_0)' \cap p(N \ovt N\op \ovt N)p$. It also follows that for all $a \in \cZ(L(\Lambda))$,
$$\|V(a)\|_2^2 = (\tau \ot \tau \ot \tau)((z \ot z\op \ot z) (1 \ot \Delta_2(a^* a)) \Delta_3(z)) \; .$$
Define $y \in L(\Lambda)$ such that $\Delta_2(y) = E_{\Delta_2(L(\Lambda))}(z\op \ot z)$. Then $\|y\| \leq 1$ and a similar computation as above gives
$$y = \sum_{g \in \Lambda} |(z)_g|^2 \, u_g \; .$$
Using \eqref{eq.formula-trace}, we conclude that
$$\|V(a)\|_2 = \|a y\|_2 \quad\text{for all $a \in \cZ(L(\Lambda))$.}$$
Whenever $C \subset \Lambda$ is a finite conjugacy class, define $a_C \in \cZ(L(\Lambda))$ by
$$a_C = |C|^{-1/2} \sum_{h \in C} u_h \; .$$
Write $\delta = \sum_{g \in \Lambda} |(z)_g|^4$. Since $\|y\| \leq 1$ and $\|a_C\|_2 = 1$, we get that
$$1 \geq \|a_C y\|_2^2 = |C|^{-1} \sum_{g \in \Lambda} \Bigl( \sum_{h \in C} |(z)_{h^{-1}g}|^2\Bigr)^2 \geq |C|^{-1} \sum_{g \in \Lambda, h \in C} |(z)_{h^{-1}g}|^4 = \delta \; .$$
So, whenever $C_n$ is a sequence of finite conjugacy classes tending to infinity, we have the sequence of elements $V(a_{C_n})$ in $\Delta_z(Nz_0)' \cap p(N \ovt N\op \ovt N)p$ with the properties
$$\delta \leq \|V(a_{C_n})\|_2^2 \leq 1 \quad\text{for all $n$, and}\quad V(a_{C_n}) \to 0 \quad\text{weakly.}$$
This means that $\Delta_z(Nz_0)' \cap p(N \ovt N\op \ovt N)p$ is infinite dimensional.
\end{proof}

The proof of the following proposition is almost identical to the proof of \cite[Proposition 7.2.(3) and 7.2.(2)]{IPV10} and thus omitted.

\begin{proposition}\label{prop.properties-triple-comult}
Let $\Lambda$ be a countable icc group and $\om \in Z^2(\Lambda,\T)$. Denote $N = L_\om(\Lambda)$ and denote by
$$\Delta_3 : N \to N \ovt N\op \ovt N$$
the triple comultiplication given by Proposition \ref{prop.triple-comult}.
\begin{enumlist}
\item\label{prop.properties-triple-comult.i} If $K \subset L^2(N \ovt N\op \ovt N)$ is a $\Delta_3(N)$-$\Delta_3(N)$-subbimodule that is finitely generated as a right Hilbert $\Delta_3(N)$-module, then $K \subset \Delta_3(L^2(N))$.

In particular, if $P \subset N$ is a finite index subfactor, then $\Delta_3(P)' \cap (N \ovt N\op \ovt N) = \Delta_3(P' \cap N)$.

\item\label{prop.properties-triple-comult.ii} If $p \in M_k(\C) \ot N$ is a projection and $P \subset p (M_k(\C) \ot N)p$ is a subfactor such that $\Delta_3(N) \prec_{N \ovt N\op \ovt N} P \ovt N\op \ovt N$, there exists a nonzero projection $p_0 \in P' \cap p (M_k(\C) \ot N)p$ such that $Pp_0 \subset p_0 (M_k(\C) \ot N)p_0$ has finite index.
\end{enumlist}
\end{proposition}

\subsection{Intermediate subfactors for twisted group von Neumann algebras}

Let $G$ be a countable group and $\mu \in Z^2(G,\T)$ a $2$-cocycle. Whenever $\pi : G \to \cU(n)$ is a projective representation with $2$-cocycle $\om_\pi$ and $\delta \in \Aut G$ is an automorphism such that $\om_\pi \, (\mu \circ \delta) = \mu$, we denote
\begin{equation}\label{eq.psi-pi-delta}
\psi_{\pi,\delta} : L_\mu(G) \to M_n(\C) \ot L_\mu(G) : \psi_{\pi,\delta}(u_g) = \pi(g) \ot u_{\delta(g)} \; .
\end{equation}

Whenever $G$ is an icc group, we identify $G$ with a subgroup of $\Aut G$ via $g \mapsto \Ad g$.

\begin{proposition}\label{prop.intermediate}
Let $G$ be an icc group, $\mu \in Z^2(G,\T)$ and $\psi : L_\mu(G) \to M_n(\C) \ot L_\mu(G)$ an embedding that is unitarily conjugate to a direct sum of embeddings of the form \eqref{eq.psi-pi-delta}.

If $N$ is a factor and $\psi(L_\mu(G)) \subset N \subset M_n(\C) \ot L_\mu(G)$, there exists a subgroup $G_1 < \Aut G$ and a $2$-cocycle $\mu_1 \in Z^2(G_1,\T)$ such that
\begin{itemlist}
\item $G \cap G_1$ has finite index in both $G$ and $G_1$, and in particular $G_1$ is icc,
\item the $2$-cocycle $\mu_1|_{G \cap G_1} \, \overline{\mu}|_{G \cap G_1}$ on $G \cap G_1$ is of finite type,
\item $N \cong L_{\mu_1}(G_1)^t$ for some $t > 0$.
\end{itemlist}
\end{proposition}

\begin{proof}
We write $M = L_\mu(G)$. After a unitary conjugacy and regrouping, we may assume that $\psi = \bigoplus_{i=1}^k \psi_{\pi_i,\delta_i}$ where for $i \neq j$, $\delta_i$ is not conjugate to $\delta_j$. By Lemma \ref{lem.remarks-icc.ii},
\begin{equation}\label{eq.this-is-infinite}
\{\delta_i(g) h \delta_j(g)^{-1} \mid g \in G\} \quad\text{is infinite whenever $i \neq j$ and $h \in G$.}
\end{equation}
We have $\pi_i : G \to \cU(n_i)$ and $n = \sum_{i=1}^k n_i$. We write $B = \bigoplus_{i=1}^k M_{n_i}(\C)$. By \eqref{eq.this-is-infinite} and the icc property of $G$, we find that
$$\psi(M)' \cap M_n(\C) \ot M = \Bigl(\bigoplus_{i=1}^k \End \pi_i\Bigr) \ot 1 \subset B \ot 1 \; ,$$
where $\End \pi_i = M_{n_i}(\C) \cap \pi_i(G)'$. So, after cutting everything with a minimal projection in $\psi(M)' \cap N$, we may moreover assume that $\psi(M)' \cap N = \C 1$.

For every $\delta \in \Aut G$, we define the projective representation
$$\zeta_\delta : G \to \cU(L^2(M_n(\C) \ot M)) : \zeta_\delta(g) \xi = \psi(u_{\delta(g)}) \xi \psi(u_g)^* \; .$$
Note that the $2$-cocycle of $\zeta_\delta$ equals $(\mu \circ \delta) \, \overline{\mu}$. We denote by $W_\delta$ the closed linear span of all finite dimensional $\zeta_\delta$-invariant subspaces. So whenever $W_\delta \neq \{0\}$, the $2$-cocycle $(\mu \circ \delta) \, \overline{\mu}$ is of finite type. We also define the subset $I_\delta \subset \{1,\ldots,k\}^2$ by
$$I_\delta = \{(i,j) \in \{1,\ldots,k\}^2 \mid \exists h \in G : \delta_i \circ \delta = (\Ad h) \circ \delta_j \} \; .$$
For all $(i,j) \in I_\delta$, we then have the unique element $\rho(\delta,i,j) \in G$ such that $\delta_i \circ \delta = \Ad \rho(\delta,i,j) \circ \delta_j$.

We denote by $p_i \in M_n(\C)$ the projection that corresponds to the direct summand $M_{n_i}(\C)$. So, the $p_i$ are the minimal projections of $\cZ(B)$. Note that each $\zeta_\delta$ leaves each of the subspaces $p_i M_n(\C) p_j \ot L^2(M)$ globally invariant. By the form of $\psi$ and Lemma \ref{lem.remarks-icc.ii}, we find that
$$(p_i \ot 1) W_\delta (p_j \ot 1) = \begin{cases} \{0\} &\quad\text{if $(i,j) \not\in I_\delta$,}\\
p_i M_n(\C) p_j \ot u_{\rho(\delta,i,j)} &\quad\text{if $(i,j) \in I_\delta$.}\end{cases}$$
In particular, $(\mu \circ \delta) \, \overline{\mu}$ is of finite type when $I_\delta \neq \emptyset$. For all $i,j \in \{1,\ldots,k\}$ and all $g \in G$, we trivially have that $p_i M_n(\C) p_j \ot u_g \subset W_{\delta_i^{-1} \circ (\Ad g) \circ \delta_j}$. Also note that if $\delta \neq \delta'$ and $(i,j) \in I_\delta \cap I_{\delta'}$, then $\rho(\delta,i,j) \neq \rho(\delta',i,j)$. Therefore $W_\delta \perp W_{\delta'}$ whenever $\delta \neq \delta'$. Defining the subset $G_2 \subset \Aut G$ by $G_2 = \{\delta \in \Aut G \mid W_\delta \neq \{0\}\}$, we thus find the orthogonal direct sum decomposition
$$L^2(M_n(\C) \ot L(G)) = \bigoplus_{\delta \in G_2} W_\delta \; .$$
Since $\psi(u_{g}) \in N$ for all $g \in G$, the projective representation $\zeta_\delta$ globally preserves $L^2(N)$ and it follows from Lemma \ref{lem.projection-onto-non-weak-mixing} below that $P_{W_\delta}(N) \subset L^2(N)$. Since $W_\delta$ only contains bounded operators, we have $P_{W_\delta}(N) = N \cap W_\delta$. Write $V_\delta = N \cap W_\delta$. Defining $G_3 = \{\delta \in \Aut G \mid V_\delta \neq \{0\}\}$, we get the orthogonal direct sum decomposition
\begin{equation}\label{eq.direct-sum-N-Vdelta}
L^2(N) = \bigoplus_{\delta \in G_3} V_\delta \; .
\end{equation}
Note that by construction, for all $\delta,\delta' \in \Aut G$, we have $V_\delta^* = V_{\delta^{-1}}$ and $V_\delta V_{\delta'} = V_{\delta \circ \delta'}$. Since $W_{\id} = B \ot 1$, we find that $V_{\id} = N \cap (B \ot 1)$. We denote $A = N \cap (B \ot 1)$. By definition, $A$ is a finite dimensional $*$-algebra and with $\al_g =\Ad \psi(u_g)$, we have the action $(\al_g)_{g \in G}$ of $G$ on $A$, which is ergodic because of the irreducibility of $\psi(M) \subset N$. It follows that $A \cong M_d(\C) \ot \C^m$ for integers $d,m \in \N$.

We claim that for every $\delta \in G_3$, the set $V_\delta$ contains a unitary $Y_\delta \in V_\delta$. To prove this claim, fix $\delta \in G_3$. Since $V_\delta \neq \{0\}$, we can choose a minimal projection $z_1 \in \cZ(A)$ such that $z_1 V_\delta \neq \{0\}$. Since $V_\delta$ is globally invariant under the representation $\zeta_\delta$ and since the action $\al$ is ergodic, it follows that $z V_\delta \neq \{0\}$ for all minimal central projections $z \in \cZ(A)$.

Fix a minimal central projection $z \in \cZ(A)$. Since $z V_\delta \neq \{0\}$, we can choose a minimal central projection $z' \in \cZ(A)$ such that $z V_\delta z' \neq \{0\}$. Choose minimal projections $p \in Az$ and $q \in Az'$ such that $p V_\delta q \neq \{0\}$. Whenever $X \in p V_\delta q$, we have that $XX^* \in p A p = \C p$ and $XX^* \in q A q = \C q$. Since $p V_\delta q \neq \{0\}$, we can choose $X \in V_\delta$ with $XX^* = p$ and $X^* X = q$. Since $A z \cong M_d(\C) \cong A z'$ and $A V_\delta = V_\delta = V_\delta A$, we can use matrix units in $Az$ and $Az'$ to find $Y \in V_\delta$ with $YY^* = z$ and $Y^* Y = z'$.

We can find such an element $Y \in V_\delta$ whenever $z V_\delta z' \neq \{0\}$. If $z\dpr \in \cZ(A)$ would be another minimal central projection with $z V_\delta z\dpr \neq \{0\}$, we also find $Z \in V_\delta$ with $ZZ^* = z$ and $Z^* Z = z\dpr$. Since $Y^* Z \in V_\delta^* V_\delta = V_{\delta^{-1}} V_\delta \subset V_{\id} = A$, we get that $Y^* Z$ is a partial isometry in $A$ with left support $z'$ and right support $z\dpr$, which is absurd because $z',z\dpr$ are central and orthogonal. So, for every minimal central projection $z \in \cZ(A)$, there is a unique minimal central projection $z' \in \cZ(A)$ with $z V_\delta z' \neq \{0\}$ and a $Y_z \in V_\delta$ with $Y_zY_z^* = z$ and $Y_z^* Y_z = z'$.

By symmetry, also for every minimal central projection $z' \in \cZ(A)$, there is a unique minimal central projection $z \in \cZ(A)$ with $z V_\delta z' \neq \{0\}$. It follows that the sum $Y_\delta = \sum_z Y_z$ is a unitary element in $V_\delta$, which proves the claim.

Since $V_\delta^* = V_{\delta^{-1}}$ and $V_{\delta} V_{\delta'} \subset V_{\delta \circ \delta'}$, it then follows that $G_3$ is a subgroup of $\Aut G$. Note that for all $\delta,\delta' \in G_3$, the unitary $Y_\delta$ normalizes $A = V_{\id}$. Also $Y_\delta^* \in \cU(A) Y_{\delta^{-1}}$ and $Y_{\delta} Y_{\delta'} \in \cU(A) Y_{\delta \circ \delta'}$. In particular, $(\Ad Y_\delta)_{\delta \in G_3}$ is an action of $G_3$ on $\cZ(A)$.
Fixing a minimal central projection $z \in \cZ(A)$, we can thus define the finite index subgroup $G_1 < G_3$ by $G_1 = \{\delta \in G_3 \mid Y_\delta z = z Y_\delta\}$. By \eqref{eq.direct-sum-N-Vdelta}, we find that
\begin{equation}\label{eq.direct-sum-N-Ydelta}
L^2(zNz) = \bigoplus_{\delta \in G_1} Y_\delta A z \; .
\end{equation}
For every $\delta \in G_1$, $\Ad Y_\delta$ defines an automorphism of $Az \cong M_d(\C)$. We can thus choose $a_\delta \in \cU(Az)$ such that $\Ad Y_\delta z = \Ad a_\delta$ on $Az$. Fix a minimal projection $p \in Az$. Write $U_\delta = Y_\delta a_\delta^* p$. Then, $U_\delta \in \cU(pNp)$ and it follows from \eqref{eq.direct-sum-N-Ydelta} that
$$L^2(pNp) = \bigoplus_{\delta \in G_1} \C U_\delta \; .$$
We conclude that $pNp \cong L_{\mu_1}(G_1)$ where $\mu_1 \in Z^2(G_1,\T)$ is defined by $U_\delta U_{\delta'} = \mu_1(\delta,\delta') \, U_{\delta \circ \delta'}$ for all $\delta,\delta' \in G_1$.

We identify $G$ with a subgroup of $\Aut G$ through the adjoint action $\Ad$. In this way, $G < G_3$ and $V_{g} = \psi(u_g) A$ for all $g \in G$. Also,
$$G_3 \subset G_2 \subset \{ \delta_i^{-1} \circ \delta_j \circ \Ad h \mid i,j \in \{1,\ldots,k\}, h \in G\} \; .$$
So, $G$ is a finite index subgroup of $G_3$. Since $G_1 < G_3$ has finite index, it follows that $G \cap G_1$ has finite index in both $G$ and $G_1$.

It only remains to prove that the $2$-cocycle ${\mu_1}|_{G \cap G_1} \, \overline{\mu}|_{G \cap G_1}$ is of finite type.

Since $V_g = \psi(u_g) A$ for all $g \in G$, also $Y_g = \psi(u_g) b_g$ with $b_g \in \cU(A)$. When $g \in G \cap G_1$, both $b_g$ and $Y_g$ commute with $z$, so that also $\psi(u_g)$ commutes with $z$. We have chosen $a_g \in \cU(Az)$ such that $\Ad Y_g z = \Ad a_g$ on $Az$. Thus, $\Ad \psi(u_g)z = \Ad a_g b_g^* z$ on $Az$ for all $g \in G \cap G_1$. Since $(\Ad \psi(u_g))_{g \in G \cap G_1}$ is an action, it follows that $G \cap G_1 \to \cU(Az) : \theta(g) = a_g b_g^* z$ is a finite dimensional projection representation. Denote its $2$-cocycle by $\om_\theta$.

By definition, $U_g = Y_g a_g^* p = \psi(u_g) \theta(g)^* p$ and $\psi(u_h)\theta(h)^*$ commutes with $Az$ for all $g,h \in G \cap G_1$. Therefore, for all $g,h \in G \cap G_1$,
\begin{align*}
U_g \, U_h &= \psi(u_g) \theta(g)^* p \; \psi(u_h) \theta(h)^* p = \psi(u_g) \; \psi(u_h) \theta(h)^* \; \theta(g)^* p
\\ &= \mu(g,h) \, \overline{\om_\theta}(g,h) \, \psi(u_{gh}) \, \theta(gh)^* \, p = \mu(g,h) \, \overline{\om_\theta}(g,h) \, U_{gh} \; .
\end{align*}
This means that $\mu_1|_{G \cap G_1} \, \overline{\mu}|_{G \cap G_1} = \overline{\om_\theta}$ is of finite type. So the theorem is proven.
\end{proof}

\begin{theorem}\label{thm.virtual-isom-factor}
Let $\Gamma$ be a group in $\cC$ and write $G = (\Z/2\Z)^{(\Gamma)} \rtimes (\Gamma \times \Gamma)$. Let $\mu \in Z^2(G,\T)$.

If $N$ is a II$_1$ factor that admits a nonzero bifinite $L_\mu(G)$-$N$-bimodule, there exists a subgroup $G_1 < \Aut G$ with $G \cap G_1$ having finite index in both $G$ and $G_1$ and a $2$-cocycle $\mu_1 \in Z^2(G_1,\T)$ such that
\begin{itemlist}
\item the $2$-cocycle $\mu_1|_{G \cap G_1} \, \overline{\mu}|_{G \cap G_1}$ on $G \cap G_1$ is of finite type,
\item $N \cong L_{\mu_1}(G_1)^t$ for some $t > 0$.
\end{itemlist}
\end{theorem}

\begin{proof}
Let $\bim{L_\mu(G)}{H}{N}$ be a nonzero bifinite bimodule. Then $K = H \ot_N \overline{H}$ is a bifinite $L_\mu(G)$-$L_\mu(G)$-bimodule. By Theorem \ref{thm.coarse-embedding-cocycle-twisted-wreath-product}, $K$ is isomorphic with a finite direct sum of bimodules given by $\psi_{\pi,\delta}$ as in \eqref{eq.psi-pi-delta}. The conclusion then follows from Proposition \ref{prop.intermediate}.
\end{proof}

We also have the following constructive converse of Proposition \ref{prop.intermediate}.

\begin{lemma}\label{lem.good-bimodule}
Let $\cG$ be a countable group with subgroups $G,G_1 < \cG$ such that $G \cap G_1$ has finite index in both $G$ and $G_1$. Let $\mu \in Z^2(G,\T)$ and $\mu_1 \in Z^2(G_1,\T)$ be scalar $2$-cocycles such that the cocycle $\mu_1|_{G \cap G_1} \, \overline{\mu}|_{G \cap G_1}$ on $G \cap G_1$ is of finite type.

There exist $n \in \N$ and an embedding $\vphi : L_{\mu_1}(G_1) \to M_n(\C) \ot L_\mu(G)$ of finite index with the following property: whenever $(a_n)$ is a bounded sequence in $L_{\mu_1}(G_1)$ and $h_{G_1}(a_n) \to 0$, also $h_G(\vphi(a_n)) \to 0$.
\end{lemma}
\begin{proof}
Let $\{s_1,\ldots,s_k\}$ be representatives of $G_1 / (G \cap G_1)$. Define $G_1 \overset{\ast}{\actson} \{1,\ldots,k\}$ and the $1$-cocycle $\Om : G_1 \times \{1,\ldots,k\} \to G \cap G_1$ such that
$$g s_i = s_{g \ast i} \, \Om(g,i) \quad\text{for all $g \in G_1$ and $i \in \{1,\ldots,k\}$.}$$
Then
\begin{multline}\label{eq.induction-cocycle}
\qquad\vphi_1 : L_{\mu_1}(G_1) \to M_k(\C) \ot L_{\mu_1}(G \cap G_1) : \\ \vphi_1(u_g) = \sum_{i=1}^k \mu_1(g,s_i) \, \overline{\mu_1}(s_{g \ast i},\Om(g,i)) \, (e_{g \ast i,i} \ot u_{\Om(g,i)})\qquad
\end{multline}
is a finite index embedding. Choosing a projective representation $\pi : G \cap G_1 \to \cU(d)$ with $2$-cocycle $\om_\pi = \mu_1|_{G \cap G_1} \, \overline{\mu}|_{G \cap G_1}$, we also have the finite index embedding
$$\vphi_2 : L_{\mu_1}(G \cap G_1) \to M_d(\C) \ot L_\mu(G) : \vphi_2(u_g) = \pi(g) \ot u_g \; .$$
Defining $n = kd$ and $\vphi = (\id  \ot \vphi_2) \circ \vphi_1$, the lemma is proven.
\end{proof}

The same induction formula \eqref{eq.induction-cocycle} also gives the following.

\begin{lemma}\label{lem.induce-finite-type}
Let $G$ be a countable group and $G_1 < G$ a finite index subgroup. Let $\mu \in Z^2(G,\T)$. If $\mu|_{G_1}$ is of finite type on $G_1$, then $\mu$ is of finite type.
\end{lemma}
\begin{proof}
Choose a finite dimensional projective representation $\theta_1 : G_1 \to \cU(n)$ with $2$-cocycle $\om_{\theta_1} = \mu|_{G_1}$. Using the same notation as in \eqref{eq.induction-cocycle}, where $\{s_1,\ldots,s_k\}$ are representatives for $G/G_1$, the formula
$$\theta : G \to \cU(M_k(\C) \ot M_n(\C)) : \theta(g) = \sum_{i=1}^k \mu(g,s_i) \, \overline{\mu}(s_{g \ast i},\Om(g,i)) \, (e_{g \ast i,i} \ot \theta_1(\Om(g,i)))$$
is a projective representation with $2$-cocycle $\om_\theta = \mu$.
\end{proof}

In the proof of Proposition \ref{prop.intermediate}, we needed the following elementary lemma.

\begin{lemma}\label{lem.projection-onto-non-weak-mixing}
Let $G$ be a group and $\rho : G \to \cU(K)$ a projective representation. Denote by $K_0$ the closed linear span of all finite dimensional $\rho(G)$-invariant subspaces. For every $\xi \in K$, we have that
$$P_{K_0}(\xi) \in \overline{\operatorname{span}} \{\rho(g) \xi \mid g \in G\} \; .$$
\end{lemma}
\begin{proof}
For every $V \in \cU(K) \cap \rho(G)'$, we have that $V(K_0) = K_0$, so that $V P_{K_0} = P_{K_0} V$. This means that $P_{K_0}$ belongs to the center of $\rho(G)' \cap B(K)$. Fix $\xi \in K$. Denote by $P$ the projection onto $\overline{\operatorname{span}} \{\rho(g) \xi \mid g \in G\}$. Then $P \in \rho(G)' \cap B(K)$, so that $P$ commutes with $P_{K_0}$. We conclude that $P(P_{K_0}(\xi)) = P_{K_0}(P(\xi)) = P_{K_0}(\xi)$.
\end{proof}

\subsection{Commensurators of left-right wreath products}

Let $\Gamma$ be an icc group. The semidirect product group $\Gamma \rtimes \Aut \Gamma$ with group operation
$$(g,\al)(g',\al') = (g\al(g'),\al \circ \al')$$
contains $\Gamma \times \Gamma$ as the normal subgroup of elements $(gh^{-1},\Ad h)$, $(g,h) \in \Gamma \times \Gamma$ and admits the period $2$ automorphism $\zeta$ given by $\zeta(g,\al) = (g^{-1},\Ad g \circ \al)$. Consider the semidirect product $H = (\Gamma \rtimes \Aut \Gamma) \rtimes \Z/2\Z$.

Note that $H$ is naturally identified with the subgroup of $\Aut(\Gamma \times \Gamma)$ consisting of the automorphisms $\delta$ that are either of the form $\delta(k,k') = (g \al(k) g^{-1},\al(k'))$ or the form $\delta(k,k') =  (\al(k'),g \al(k) g^{-1})$, for some $g \in \Gamma$ and $\al \in \Aut \Gamma$. Also note that $H$ acts on the set $\Gamma$ by $(g,\delta) \cdot k = g \delta(k)$ and $\zeta \cdot k = k^{-1}$. Writing
\begin{equation}\label{eq.G-cG}
G = (\Z/2\Z)^{(\Gamma)} \rtimes (\Gamma \times \Gamma) \quad\text{and}\quad \cG = (\Z/2\Z)^{(\Gamma)} \rtimes ((\Gamma \rtimes \Aut \Gamma) \rtimes \Z/2\Z) \; ,
\end{equation}
we have that $G$ is a normal subgroup of $\cG$ and $G < \cG$ is relatively icc. For many groups $\Gamma$, one has $\cG = \Aut G$. This statement comes down to saying that for most icc groups $\Gamma$, every automorphism of $G$ can be written as the composition of an inner automorphism and an automorphism that is given by a symmetric automorphism of $\Gamma \times \Gamma$. We only need this result when $\Gamma \in \cC$.

The following result can certainly be proven by elementary means and for a much larger class of groups $\Gamma$, but for brevity, we give a short proof as a consequence of Theorem \ref{thm.coarse-embedding-cocycle-twisted-wreath-product}.

\begin{proposition}\label{prop.commensurators-wreath}
Let $\Gamma$ be a group in $\cC$ and denote $G = (\Z/2\Z)^{(\Gamma)} \rtimes (\Gamma \times \Gamma)$. If $G_0,G_1 < G$ are finite index subgroups and $\delta_0 : G_0 \to G_1$ is an isomorphism, then $\delta_0$ uniquely extends to an automorphism $\delta \in \Aut G$. Also, $\Aut G = \cG$ with $\cG$ defined by \eqref{eq.G-cG}.
\end{proposition}
\begin{proof}
Choose representatives $\{s_1,\ldots,s_n\}$ for $G / G_0$. Define the action $G \overset{\ast}{\actson} \{1,\ldots,n\}$ and the $1$-cocycle $\Om : G \times \{1,\ldots,n\} \to G_0$ by $g s_i = s_{g \ast i} \Om(g,i)$ for all $g \in G$ and $i \in \{1,\ldots,n\}$. Then
$$\vphi : L(G) \to M_n(\C) \ot L(G) : \vphi(u_g) = \sum_{i=1}^n e_{g \ast i,i} \ot u_{\delta_0(\Om(g,i))}$$
is a finite index embedding. By Theorem \ref{thm.coarse-embedding-cocycle-twisted-wreath-product} (and since we only have one tensor factor at the right, this actually also follows from \cite[Theorem 3.2]{PV21}), we find a finite dimensional representation $\pi : G \to \cU(d)$, an automorphism $\delta \in \Aut G$ that belongs to $\cG$ and a nonzero element $X \in M_{n,d}(\C) \ot L(G)$ such that $\vphi(u_g) X = X (\pi(g) \ot u_{\delta(g)})$ for all $g \in G$.

Define the finite index subgroup $G_2 < G$ by $G_2 = \bigcap_{i=1}^n s_i G_0 s_i^{-1}$. For all $g \in G_2$ and $i \in \{1,\ldots,n\}$, we have that $g \ast i = i$ and $\Om(g,i) = s_i^{-1} g s_i$. So,
$$\vphi(u_g) = \sum_{i=1}^n e_{i,i} \ot u_{\delta_0(s_i^{-1} g s_i)} \quad\text{for all $g \in G_2$.}$$
Because we have the intertwiner $X$, there must exist an $i \in \{1,\ldots,n\}$ such that the unitary representation
$$\zeta_i : G_2 \to \cU(\ell^2(G)) : \zeta_i(g) \xi = u_{\delta_0(s_i^{-1} g s_i)} \xi u_{\delta(g)}^*$$
is not weakly mixing. There thus exists an $h \in G$ such that the set $\{\delta_0(s_i^{-1} g s_i) h \delta(g)^{-1} \mid g \in G_2\}$ is finite. By Lemma \ref{lem.remarks-icc.ii}, there exists an $h \in G$ such that $\delta_0(s_i^{-1} g s_i) = (\Ad h \circ \delta)(g)$ for all $g \in G_2$. Defining $G_3 = s_i^{-1} G_2 s_i$ and replacing $\delta$ by $\Ad h \circ \delta \circ \Ad s_i$, we find that $G_3 < G_0$ is a finite index subgroup and $\delta_0(g) = \delta(g)$ for all $g \in G_3$. By Lemma \ref{lem.remarks-icc.ii}, we get that $\delta_0(g) = \delta(g)$ for all $g \in G_0$. So $\delta_0$ can be extended to an automorphism $\delta$ of $G$, which moreover belongs to $\cG$. Also by Lemma \ref{lem.remarks-icc.ii}, this extension is unique.
\end{proof}

\subsection{Proof of Theorem \ref{thm.A}}\label{sec.proof-theorem-A}

The two main results in this section are Theorems \ref{thm.virtual-iso-superrigidity-precise} and \ref{thm.iso-superrigidity-precise}. The first statement of Theorem \ref{thm.A} is contained in Theorem \ref{thm.iso-superrigidity-precise} below. The second statement of Theorem \ref{thm.A} is contained in Corollary \ref{cor.nonzero-or-faithful-bimodules.i} below.

\begin{lemma}\label{lem.virtual-center-finite}
Let $\Gamma$ be a group in $\cC$ and write $G = (\Z/2\Z)^{(\Gamma)} \rtimes (\Gamma \times \Gamma)$. Let $\Lambda$ be any countable group and let $\mu \in Z^2(G,\T)$ and $\om \in Z^2(\Lambda,\T)$ be arbitrary scalar $2$-cocycles.

If there exists a nonzero bifinite $L_\mu(G)$-$L_\om(\Lambda)$-bimodule, then the virtual center $\Lambda\fc$ is finite.
\end{lemma}
\begin{proof}
Write $M = L_\mu(G)$ and $N = L_\om(\Lambda)$. Let $\bim{N}{L}{M}$ be a nonzero bifinite $N$-$M$-bimodule. Let $z \in \cZ(N)$ be the support projection of the left action of $N$ on $L$, so that $\bim{Nz}{L}{M}$ has faithful left and right actions. Denote by $\Delta_z : N z_0 \to p(N \ovt N\op \ovt N)p$ the coarse embedding given by Proposition \ref{prop.triple-comult}. Denote by
$$\bim{N z_0}{H}{N \ovt N\op \ovt N} = \bim{\Delta_z(N z_0)}{pL^2(N \ovt N\op \ovt N)}{N \ovt N\op \ovt N}$$
the corresponding tensor coarse bimodule (see Definition \ref{def.tensor-coarse-bimodule}). By Lemma \ref{lem.stability-of-coarse},
$$\cH := \overline{L} z_0 \ot_{N z_0} H \ot_{N \ovt N\op \ovt N} (L \ot L\op \ot L)$$
is an $M$-$(M \ovt M\op \ovt M)$-bimodule that is tensor coarse. It follows from Theorem \ref{thm.coarse-embedding-cocycle-twisted-wreath-product} that $\cH$ is a finite direct sum of irreducible bimodules. So, the space of endomorphisms of $\cH$ is finite dimensional. This implies that $\Delta_z(Nz_0)' \cap p(N \ovt N\op \ovt N)p$, which is the endomorphism algebra of $H$, is finite dimensional. It follows from Proposition \ref{prop.triple-comult} that $\Lambda\fc$ is finite.
\end{proof}

In the following result, we use the notations introduced in Proposition \ref{prop.decompose-virtual-center}. As above, we identify $G$ with a subgroup of $\Aut G$ via $g \mapsto \Ad g$.

\begin{theorem}\label{thm.virtual-iso-superrigidity-precise}
Let $\Gamma$ be a group in $\cC$ and write $G = (\Z/2\Z)^{(\Gamma)} \rtimes (\Gamma \times \Gamma)$. Let $\Lambda$ be any countable group and let $\mu \in Z^2(G,\T)$ and $\om \in Z^2(\Lambda,\T)$ be arbitrary scalar $2$-cocycles.

Let $\bim{L_\mu(G)}{\cH}{L_\om(\Lambda)}$ be any nonzero bifinite bimodule. Then the virtual center $\Lambda\fc$ is finite and for every minimal projection $z \in \cZ(L_\om(\Lambda\fc))$ with $\cH z \neq 0$, there exists
\begin{itemlist}
\item a subgroup $G_1 < \Aut G$ such that $G \cap G_1$ has finite index in both $G$ and $G_1$,
\item a $2$-cocycle $\mu_1 \in Z^2(G_1,\T)$ s.t.\ the $2$-cocycle $\mu|_{G \cap G_1} \, \overline{\mu_1}|_{G \cap G_1}$ on $G \cap G_1$ is of finite type,
\item an isomorphism of groups $\delta : \Lambda_z / \Lambda\fc \to G_1$ satisfying $\mu_1 \circ \delta \sim \om_z$.
\end{itemlist}
\end{theorem}
\begin{proof}
By Lemma \ref{lem.virtual-center-finite}, the virtual center $\Lambda\fc$ is finite. Fix a minimal projection $z \in \cZ(L_\om(\Lambda\fc))$ with $\cH z \neq 0$. By Proposition \ref{prop.decompose-virtual-center}, $z L_\om(\Lambda) z$ is isomorphic with an amplification of $L_{\om_z}(\Lambda_z/\Lambda\fc)$. Since the group $\Lambda_z/\Lambda\fc$ is icc, we may thus assume from the start that $\Lambda$ is an icc group.

By Theorem \ref{thm.virtual-isom-factor}, we can choose $G_1 < \Aut G$ and $\mu_1 \in Z^2(G_1,\T)$ satisfying the first two statements in the theorem and such that there exists a $*$-isomorphism $\al : L_\om(\Lambda) \to L_{\mu_1}(G_1)^t$ for some $t > 0$. We denote by $(v_h)_{h \in \Lambda}$ the canonical unitaries in $L_\om(\Lambda)$. Below we prove that $h_{G_1}(\al(\Lambda)) > 0$.

Since $\Lambda$ is an icc group, the action $(\Ad v_h)_{h \in \Lambda}$ on $L^2(L_\om(\Lambda)) \ominus \C 1$ is weakly mixing. Since $G_1$ is icc, we have that $\al(L_\om(\Lambda)) \not\prec L_{\mu_1}(C_{G_1}(k))$ whenever $k \in G_1 \setminus \{e\}$. So we can apply Theorem \ref{thm.height-1}. We conclude that $t=1$ and that there exists a unitary $W \in L_{\mu_1}(G_1)$ and a group isomorphism $\delta : \Lambda \to G_1$ such that $W \al(v_h) W^* \in \T \cdot u_{\delta(h)}$ for all $h \in \Lambda$. Writing $\al(v_h) = \vphi(h) \, u_{\delta(h)}$ with $\vphi(h) \in \T$, it follows that $\mu_1 \circ \delta \sim \om$. It thus remains to prove that $h_{G_1}(\al(\Lambda)) > 0$.

Since $G < \Aut G$ is relatively icc, also $G_1$ is an icc group. We denote by $\vphi : L_{\mu_1}(G_1) \to M_n(\C) \ot L_\mu(G)$ the finite index embedding given by Lemma \ref{lem.good-bimodule}.

From now on, we write $N = L_\om(\Lambda)$, $M = L_\mu(G)$ and $M_1 = L_{\mu_1}(G_1)$. We amplify $\vphi$ to the finite index embedding $\Phi : M_1^t \to p (M_k(\C) \ot M)p$, where $p$ is a projection with $(\Tr \ot \tau)(p) = nt$. We denote by
$$\bim{M_1^t}{K}{M} = \bim{\Phi(M_1^t)}{p(\C^k \ot L^2(M))}{M}$$
the corresponding bifinite bimodule.

We denote by $\Delta_3 : N \to N \ovt N\op \ovt N$ the triple comultiplication given by Proposition \ref{prop.triple-comult} and let
$$\bim{N}{\bigl(H_{\Delta_3}\bigr)}{N \ovt N\op \ovt N} = \bim{\Delta_3(N)}{L^2(N \ovt N\op \ovt N)}{N \ovt N\op \ovt N}$$
be the corresponding tensor coarse bimodule.

We translate the $*$-isomorphism $\al : N \to M_1^t$ to the bifinite bimodule
$$\bim{M_1^t}{(H_\al)}{N} = \bim{M_1^t}{L^2(M_1^t)}{\al(N)} \; .$$
Define the bifinite $N$-$M$-bimodule $L = \overline{H_\al} \ot_{M_1^t} K$. Since $H_{\Delta_3}$ is tensor coarse, by Lemma \ref{lem.stability-of-coarse}, also the $M$-$(M \ovt M\op \ovt M)$-bimodule
$$H = \overline{L} \ot_N H_{\Delta_3} \ot_{N \ovt N\op \ovt N} (L \ot L\op \ot L)$$
is tensor coarse.

We claim that whenever $P \subset M$ is an irreducible subfactor such that $H$ admits a nonzero $M$-$(P \ovt M\op \ovt M)$-subbimodule $H_0$ that is finitely generated as a right $(P \ovt M\op \ovt M)$-module, then $P \subset M$ must be of finite index. To prove this claim, choose a finite index embedding $\gamma : M \to q (M_k(\C) \ot N)q$ such that
$$\bim{N}{L}{M} \cong \bim{N}{(\overline{\C^k} \ot L^2(N))q}{\gamma(M)} \; .$$
We can then identify $H$ with
$$H = (\id \ot \Delta_3)(q) \; (M_{k}(\C) \ot L^2(N) \ot M_{1,k}(\C) \ot L^2(N\op) \ot M_{1,k}(\C) \ot L^2(N)) \; (q \ot q\op \ot q)$$
such that the left action of $a \in M$ is given by left multiplication with $(\id \ot \Delta_3)\gamma(a)$, while the right action of $a_1 \ot a_2\op \ot a_3$ is given by right multiplication with $\gamma(a_1) \ot \gamma(a_2)\op \ot \gamma(a_3)$. Since $\gamma(M) \subset q (M_k(\C) \ot N)q$ has finite index, the closed linear span $H_1$ of
$$(\overline{\C^k} \ot \Delta_3(N)) \cdot H_0 \cdot (1 \ot 1 \ot (\C^k \ot N\op) \ot (\C^k \ot N))$$
then is a $\Delta_3(N)$-$(\gamma(P) \ovt N\op \ovt N)$-subbimodule of $(M_{1,k}(\C) \ot L^2(N \ovt N\op \ovt N))(q \ot 1 \ot 1)$ that is finitely generated as a right $(\gamma(P) \ovt N\op \ovt N)$-module.

Since $\Delta_3$ is the triple comultiplication, it follows from Proposition \ref{prop.properties-triple-comult.ii} that there exists a nonzero projection $q_0 \in \gamma(P)' \cap q (M_k(\C) \ot N)q$ such that $\gamma(P) q_0 \subset q_0 (M_k(\C) \ot N)q_0$ has finite index. The bounded linear map
$$V : L^2(M) \to q_0(M_k(\C) \ot L^2(N))q : V(x) = q_0 \gamma(x)$$
satisfies $V(axb) = \gamma(a) V(x) \gamma(b)$ for all $a \in P$, $b \in M$ and $x \in L^2(M)$. So the kernel $\Ker V$ is a Hilbert $P$-$M$-subbimodule of $L^2(M)$. Since $P \subset M$ is irreducible, either $\Ker V = \{0\}$ or $\Ker V = L^2(M)$. Since $V \neq 0$, it follows that $V$ is injective. Since $V(ax) = \gamma(a) V(x)$ and $\gamma(P) q_0 \subset q_0 (M_k(\C) \ot N)q_0$ has finite index, we conclude that $L^2(M)$ is finitely generated as a left $P$-module. This means that $P \subset M$ has finite index and the claim is proven.

Having proven the claim, it follows from Theorem \ref{thm.coarse-embedding-cocycle-twisted-wreath-product} that $H$ is unitarily conjugate to a finite direct sum of bimodules of the form
\begin{equation}\label{eq.H-pi-delta-ref-here}
\begin{split}
& H_{\pi,\delta} = \bim{\psi_{\pi,\delta}(M)}{(\C^d \ot L^2(M \ovt M\op \ovt M))}{M \ovt M\op \ovt M} \;, \quad\text{where}\\
& \psi_{\pi,\delta} : M \to M_d(\C) \ovt M \ovt M\op \ovt M : u_g \mapsto \pi(g) \ot u_{\delta_1(g)} \ot \overline{u_{\delta_2(g)}} \ot u_{\delta_3(g)}
\end{split}
\end{equation}
and $\delta$ is a triple of automorphisms of $G$ and $\pi : G \to \cU(d)$ a projective representation with $2$-cocycle $\om_\pi = \mu \, (\overline{\mu} \circ \delta_1) \, (\mu \circ \delta_2) \, (\overline{\mu} \circ \delta_3)$.

We are now ready to prove that $h_{G_1}(\al(\Lambda))>0$. Assume the contrary. As above, we denote by $(v_h)_{h \in \Lambda}$ the canonical unitaries in $L_\om(\Lambda)$. We can then take a sequence $(g_n)$ in $\Lambda$ such that $h_{G_1}(\al(v_{g_n})) \to 0$.

Let $H_{\pi,\delta}$ be the bimodule given by \eqref{eq.H-pi-delta-ref-here} and let $a_n,a_{1,n},a_{2,n}$ and $a_{3,n}$ be bounded sequences in $M$ such that $h_G(a_{3,n}) \to 0$. We claim that
\begin{equation}\label{eq.claim-to-control-height}
\langle a_n \cdot \xi \cdot (a_{1,n} \ot a_{2,n}\op \ot a_{3,n}) , \eta \rangle \to 0 \quad\text{for all $\xi,\eta \in H_{\pi,\delta}$.}
\end{equation}
By totality, it suffices to prove this claim for $\xi = \xi_0 \ot u_{g_1} \ot u_{g_2}\op \ot u_{g_3}$ and $\eta = \eta_0 \ot u_{h_1} \ot u_{h_2}\op \ot u_{h_3}$ with $\|\xi_0\| \leq 1$ and $\|\eta_0\|\leq 1$. We may also assume that all the elements $a_n,a_{1,n},a_{2,n}$ and $a_{3,n}$ belong to the unit ball of $M$. Decomposing an arbitrary element $a \in M$ as $a = \sum_{g \in G} (a)_g u_g$, one computes that for these special $\xi$ and $\eta$ and all $k \in G$,
\begin{align*}
|\langle u_k \cdot \xi & \cdot (a_{1,n} \ot a_{2,n}\op \ot a_{3,n}) , \eta \rangle|
\\ &= |\langle \pi(k) \xi_0,\eta_0\rangle| \; |(a_{1,n})_{g_1^{-1} \delta_1(k^{-1}) h_1}| \; |(a_{2,n})_{h_2 \delta_2(k^{-1}) g_2^{-1}}| \; |(a_{3,n})_{g_3^{-1} \delta_3(k^{-1}) h_3}|
\\ &\leq |(a_{1,n})_{g_1^{-1} \delta_1(k^{-1}) h_1}| \; h_G(a_{3,n}) \; .
\end{align*}
Since
$$|\langle a_n \cdot \xi  \cdot (a_{1,n} \ot a_{2,n}\op \ot a_{3,n}) , \eta \rangle| \leq \sum_{k \in G} |(a_n)_k| \; |\langle u_k \cdot \xi  \cdot (a_{1,n} \ot a_{2,n}\op \ot a_{3,n}) , \eta \rangle| \; ,$$
it follows that
\begin{align*}
|\langle a_n \cdot \xi  \cdot (a_{1,n} \ot a_{2,n}\op \ot a_{3,n}) , \eta \rangle|^2 &\leq \Bigl(\sum_{k \in G} |(a_n)_k|^2 \Bigr) \; \Bigl(\sum_{k \in G} |(a_{1,n})_{g_1^{-1} \delta_1(k^{-1}) h_1}|^2 \; h_G(a_{3,n})^2\Bigr) \\
& \leq h_G(a_{3,n})^2 \to 0 \; ,
\end{align*}
so that the claim is proven.

Since $H$ is unitarily conjugate to a direct sum of bimodules of the form $H_{\pi,\delta}$, it follows that
$$\langle a_n \cdot \xi \cdot (a_{1,n} \ot a_{2,n}\op \ot a_{3,n}) , \eta \rangle \to 0 \quad\text{for all $\xi,\eta \in H$,}$$
whenever $a_n,a_{1,n},a_{2,n}$ and $a_{3,n}$ are bounded sequences in $M$ such that $h_G(a_{3,n}) \to 0$.

By Lemma \ref{lem.good-bimodule}, $h_G(\Phi(\al(v_{g_n}))) \to 0$. We thus find that
\begin{equation}\label{eq.my-contra}
\langle v_{g_n}^* \cdot \xi_1 \cdot (v_{g_n} \ot \overline{v_{g_n}} \ot v_{g_n}),\xi_2 \rangle \to 0 \quad\text{for all $\xi_1,\xi_2 \in L \ot_M H \ot_{M \ovt M\op \ovt M} (\overline{L} \ot {\overline{L}}\op \ot \overline{L})$.}
\end{equation}
By construction, this last bimodule contains $H_{\Delta_3}$, so that the same holds for $\xi_1,\xi_2 \in H_{\Delta_3}$. But taking $\xi_1=\xi_2 = 1 \ot 1 \ot 1 \in H_{\Delta_3}$, the left hand side of \eqref{eq.my-contra} equals $1$ for all $n$. So we have reached a contradiction and the theorem is proven.
\end{proof}

\begin{corollary}\label{cor.nonzero-or-faithful-bimodules}
Let $\Gamma$ be a group in $\cC$ and write $G = (\Z/2\Z)^{(\Gamma)} \rtimes (\Gamma \times \Gamma)$. Let $\Lambda$ be any countable group and let $\mu \in Z^2(G,\T)$ and $\om \in Z^2(\Lambda,\T)$ be arbitrary scalar $2$-cocycles. Denote by $\cZ \subset Z^2(G,\T)$ the set of $2$-cocycles $\cZ = \{\mu \circ \delta \mid \delta \in \Aut G\}$.
\begin{enumlist}
\item\label{cor.nonzero-or-faithful-bimodules.i} There exists a \emph{nonzero} bifinite $L_\mu(G)$-$L_\om(\Lambda)$-bimodule if and only if there exist a finite index subgroup $\Lambda_0 < \Lambda$ and a group homomorphism $\delta : \Lambda_0 \to G$ with $\Ker \delta$ finite, $\delta(\Lambda_0) < G$ finite index and $\om \, (\overline{\mu} \circ \delta)$ of finite type on $\Lambda_0$.

\item\label{cor.nonzero-or-faithful-bimodules.ii} There exists a \emph{faithful} bifinite $L_\mu(G)$-$L_\om(\Lambda)$-bimodule if and only if there exist a finite index subgroup $\Lambda_0 < \Lambda$ and a group homomorphism $\delta : \Lambda_0 \to G$ with $\Ker \delta$ finite, $\delta(\Lambda_0) < G$ finite index and such that every irreducible $\om$-representation of $\Ker \delta$ appears in a finite dimensional projective representation of $\Lambda_0$ with $2$-cocycle $\om \, (\overline{\mu'} \circ \delta)$ for some $\mu' \in \cZ$.
\end{enumlist}
\end{corollary}

In Section \ref{sec.examples}, we will see that the distinction between the existence of a nonzero bifinite bimodule and the existence of a faithful bifinite bimodule is essential. We will actually construct for every $\Gamma \in \cC$ and $G = (\Z/2\Z)^{(\Gamma)} \rtimes (\Gamma \times \Gamma)$ a group $\Gtil$ such that $L(G)$ is isomorphic with a corner of $L(\Gtil)$ and such that there does not exist a faithful bifinite $L(G)$-$L(\Gtil)$-bimodule.

\begin{proof}
Throughout the proof, we say that two tracial von Neumann algebras are \emph{virtually isomorphic} if they admit a \emph{faithful} bifinite bimodule.

(i)\ If there exists a nonzero bifinite $L_\mu(G)$-$L_\om(\Lambda)$-bimodule, we apply Theorem \ref{thm.virtual-iso-superrigidity-precise}. We get that $\Lambda\fc$ is finite. Denote $G_1 < \Aut G$ and $\mu_1 \in Z^2(G_1,\T)$ as in the conclusion of Theorem \ref{thm.virtual-iso-superrigidity-precise}, and denote by $\delta_1 : \Lambda_z/\Lambda\fc \to G_1$ the isomorphism such that $\mu_1 \circ \delta_1 \sim \om_z$. We write $\Lambda_1 = \delta_1^{-1}(G \cap G_1)$. Denote by $q : \Lambda \to \Lambda/\Lambda\fc$ the quotient map, define $\Lambda_0 = q^{-1}(\Lambda_1)$ and put $\delta = \delta_1 \circ q$.

Since $\mu_1 \, \overline{\mu}$ is of finite type on $G \cap G_1$, also $(\mu_1 \circ \delta_1) \, (\overline{\mu} \circ \delta_1)$ is of finite type on $\Lambda_1$. Since $\mu_1 \circ \delta_1 \sim \om_z$, it follows that $\om_z \, (\overline{\mu} \circ \delta_1)$ is of finite type on $\Lambda_1$. By Lemma \ref{lem.cohom}, the $2$-cocycle $\om \, (\overline{\mu} \circ \delta)$ on $\Lambda_0$ is of finite type.

Conversely, assume that we have a finite index subgroup $\Lambda_0 < \Lambda$ and a group homomorphism $\delta : \Lambda_0 \to G$ such that $\Ker \delta$ is finite, $\delta(\Lambda_0) < G$ has finite index and $\om \, (\overline{\mu} \circ \delta)$ is of finite type on $\Lambda_0$. Since $G$ is icc and $\Lambda_0 < \Lambda$ has finite index, it follows that the virtual center $\Lambda\fc$ is finite and that $\Ker \delta = \Lambda_0 \cap \Lambda\fc$. Defining $\delta(g) = e$ for all $g \in \Lambda\fc$, we can uniquely extend $\delta$ to a group homomorphism $\Lambda_0 \Lambda\fc \to G$ with kernel equal to $\Lambda\fc$. Since the $2$-cocycle $\om \, (\overline{\mu} \circ \delta)$ is defined on $\Lambda_0 \Lambda\fc$, it remains of finite type on $\Lambda_0\Lambda\fc$ by Lemma \ref{lem.induce-finite-type}. We may thus assume from the start that $\Lambda\fc < \Lambda_0$, so that $\Lambda_0 = q^{-1}(\Lambda_1)$, where $\Lambda_1 < \Lambda/\Lambda\fc$ has finite index. Define the injective group homomorphism $\delta_1 : \Lambda_1 \to G$ such that $\delta = \delta_1 \circ q$.

Since $\om \, (\overline{\mu} \circ \delta)$ is of finite type on $\Lambda_0$, we can choose a finite dimensional projective representation $\theta$ of $\Lambda_0$ with $2$-cocycle $\om_\theta = \om \, (\overline{\mu} \circ \delta)$. Then the restriction of $\theta$ to $\Lambda\fc$ is a nonzero $\om$-representation. Choose a minimal central projection $z \in \cZ(L_\om(\Lambda\fc))$ such that $\theta(z) \neq 0$.

Using the notation of Proposition \ref{prop.decompose-virtual-center}, we replace $\Lambda_1$ by $(\Lambda_1 \cap \Lambda_z)/\Lambda\fc$ and we may assume that $\Lambda_1 \subset \Lambda_z / \Lambda\fc$. By Lemma \ref{lem.cohom}, the $2$-cocycle $\om_z \, (\overline{\mu} \circ \delta_1)$ on $\Lambda_1$ is of finite type.

So $L_{\om_z}(\Lambda_1)$ is virtually isomorphic with $L_{\mu \circ \delta_1}(\Lambda_1) \cong L_\mu(G_0)$, where $G_0 = \delta_1(\Lambda_1)$. Since $\Lambda_1 < \Lambda_z/\Lambda\fc$ and $G_0 < G$ have finite index, we also get that $L_{\om_z}(\Lambda_1)$ is virtually isomorphic with $L_{\om_z}(\Lambda_z/\Lambda\fc)$ and that $L_\mu(G_0)$ is virtually isomorphic with $L_\mu(G)$. So it follows that $L_{\om_z}(\Lambda_z/\Lambda\fc)$ and $L_\mu(G)$ are virtually isomorphic. By Proposition \ref{prop.decompose-virtual-center}, $L_{\om_z}(\Lambda_z/\Lambda\fc)$ is isomorphic with a corner of $L_\om(\Lambda)$, so that there exists a nonzero bifinite $L_\mu(G)$-$L_\om(\Lambda)$-bimodule.

(ii)\ First assume that there exists a \emph{faithful} bifinite $L_\mu(G)$-$L_\om(\Lambda)$-bimodule. We can apply Theorem \ref{thm.virtual-iso-superrigidity-precise} to every minimal projection $z \in \cZ(L_\om(\Lambda\fc))$. We find subgroups $G_z < \Aut G$ such that $G \cap G_z$ has finite index in both $G$ and $G_z$, $2$-cocycles $\mu_z \in Z^2(G_z,\T)$ and group isomorphisms $\delta_z : \Lambda_z/\Lambda\fc \to G_z$ such that the $2$-cocycle $\mu \, \overline{\mu_z}$ on $G \cap G_z$ is of finite type and $\mu_z \circ \delta_z \sim \om_z$.

Denote by $\{z_1,\ldots,z_n\}$ the minimal projections in $\cZ(L_\om(\Lambda\fc))$. Define
$$\Lambda_1 = \bigcap_{i=1}^n \delta_{z_i}^{-1}(G \cap G_{z_i})$$
and note that $\Lambda_1$ is a finite index subgroup of $\Lambda/\Lambda\fc$. Write $\delta_1 = \delta_{z_1}$ and $G_0 = \delta_1(\Lambda_1)$. Then $G_0 < G$ has finite index and $\delta_1 : \Lambda_1 \to G_0$ is an isomorphism.

For every $i \in \{1,\ldots,n\}$, $\delta_{z_i} \circ \delta_1^{-1}$ is an isomorphism between the finite index subgroups $G_0$ and $\delta_{z_i}(\Lambda_1)$ of $G$. By Proposition \ref{prop.commensurators-wreath}, there is a unique automorphism $\al_i \in \Aut G$ such that $\delta_{z_i}(g) = \al_i(\delta(g))$ for all $g \in \Lambda_1$.

As in the proof of (i), we get that the $2$-cocycle $\om_{z_i} \, (\overline{\mu} \circ \al_i \circ \delta_1)$ on $\Lambda_1$ is of finite type. Define $\mu_i \in \cZ$ by $\mu_i = \mu \circ \al_i$. Write $\Lambda_0 = q^{-1}(\Lambda_1)$ and $\delta = \delta_1 \circ q$. By Lemma \ref{lem.cohom}, we find finite dimensional projective representations $\theta_i$ of $\Lambda_0$ with $2$-cocycle $\om_{\theta_i} = \om \, (\overline{\mu_i} \circ \delta)$ on $\Lambda_0$ and such that $\theta_i(z_i) \neq 0$. So we have proven that every irreducible $\om$-representation of $\Lambda\fc$ appears in an appropriate finite dimensional projective representation of $\Lambda_0$.

To prove that converse, it suffices to note that the argument in (i) gives that $L_\om(\Lambda)$ is isomorphic to a direct sum of II$_1$ factors that are virtually isomorphic to $L_{\mu'}(G)$ with $\mu' \in \cZ$. Since $L_{\mu'}(G) \cong L_\mu(G)$ for every $\mu' \in \cZ$, it follows that $L_\om(\Lambda)$ is virtually isomorphic to $L_\mu(G)$.
\end{proof}

The nuance between the existence of nonzero bifinite bimodules and faithful bifinite bimodules is even more clear in the following corollary. In Section \ref{sec.examples}, we provide concrete examples to further illustrate these phenomena. We say that a $2$-cocycle $\mu$ is of \emph{finite order} if there exists an integer $n \geq 1$ such that $\mu^n \sim 1$.

\begin{corollary}\label{cor.untwisting-2-cocycles}
Let $\Gamma$ be a group in $\cC$ and write $G = (\Z/2\Z)^{(\Gamma)} \rtimes (\Gamma \times \Gamma)$. Let $\mu \in Z^2(G,\T)$. Then the following are equivalent.
\begin{enumlist}
\item There exists a countable group $\Lambda$ and a nonzero bifinite $L_\mu(G)$-$L(\Lambda)$-bimodule.
\item There exists a finite index subgroup $G_0 < G$ such that $\mu|_{G_0}$ is the product of a $2$-cocycle of finite type and a $2$-cocycle of finite order.
\end{enumlist}
Also the following are equivalent.
\begin{enumlist}
\item There exists a countable group $\Lambda$ and a faithful bifinite $L_\mu(G)$-$L(\Lambda)$-bimodule.
\item The $2$-cocycle $\mu$ is of finite type.
\item There exists a bifinite $L_\mu(G)$-$L(G)$-bimodule.
\end{enumlist}
\end{corollary}
\begin{proof}
First assume that $\Lambda$ is a countable group and that there exists a nonzero bifinite $L_\mu(G)$-$L(\Lambda)$-bimodule. By Corollary \ref{cor.nonzero-or-faithful-bimodules.i}, we find a finite index subgroup $G_1 < G$ and a finite extension $e \to \Sigma_1 \to \Gtil_1 \overset{q_1}{\to} G_1 \to e$ such that $\mu \circ q_1$ is of finite type on $\Gtil_1$.

Since $(\Ad g)_{g \in \Gtil_1}$ is an action on the finite group $\Sigma_1$, we find a finite index subgroup $\Gtil_0 < \Gtil_1$ such that every $g \in \Gtil_0$ commutes with $\Sigma_1$. Define $G_0 = q_1(\Gtil_0)$ and note that $G_0 < G_1$ has finite index. Write $\Sigma_0 = \Sigma_1 \cap \Gtil_0$. We have found the finite central extension $e \to \Sigma_0 \to \Gtil_0 \overset{q_0}{\to} G_0 \to e$.

Since $\mu \circ q_0$ is the restriction of $\mu \circ q_1$, we still have that $\mu \circ q_0$ is of finite type. Choose a finite dimensional projective representation $\theta : \Gtil_0 \to \cU(d)$ with $2$-cocycle $\om_{\theta} = \mu \circ q_0$. In particular $\om_{\theta}$ is equal to $1$ on $\Gtil_0 \times \Sigma_0$ and on $\Sigma_0 \times \Gtil_0$. So, $\theta|_{\Sigma_0}$ is an ordinary representation of the finite central subgroup $\Sigma_0$. Taking one of its irreducible components and reducing $\theta$, we may assume that $\theta(a) = \eta(a) 1$ for all $a \in \Sigma_0$ where $\eta : \Sigma_0 \to \T$ is a character on the finite abelian group $\Sigma_0$.

Choose a lift $\phi : G_0 \to \Gtil_0$. Then $\gamma = \theta \circ \phi$ is a finite dimensional projective representation of $G_0$. Since $\phi(g)\phi(h) \in \phi(gh) \Sigma_0$, we find that
$$\gamma(g) \gamma(h) \in \mu(g,h) \eta(\Sigma_0) \gamma(gh) \quad\text{for all $g,h \in G_0$.}$$
Denote by $\om_\gamma$ the $2$-cocycle of $\gamma$ and define $\mu_0 = \mu|_{G_0} \, \overline{\om_\gamma}$. Then $\mu_0(g,h) \in \eta(\Sigma_0)$ for all $g,h \in G_0$. Taking an integer $n \geq 1$ such that $\eta^n = 1$, we have written $\mu|_{G_0}$ as the product of the finite type $2$-cocycle $\om_\gamma$ and the $2$-cocycle $\mu_0$ that is of finite order because $\mu_0^n = 1$.

Conversely assume that $G_0 < G$ is a finite index subgroup such that $\mu|_{G_0} = \om \, \mu_0$, where $\om$ is of finite type and $\mu_0^n = 1$. Denote by $\Lambda$ the degree $n$ central extension defined by $\mu_0$, with quotient map $q : \Lambda \to G_0$. By construction, $\mu \circ q$ is of finite type, so that there exists a nonzero bifinite $L_\mu(G)$-$L(\Lambda)$-bimodule.

For the second part of the corollary, we first prove (i) $\Rightarrow$ (ii). Assume that $\Lambda$ is a countable group and that there exists a faithful bifinite $L_\mu(G)$-$L(\Lambda)$-bimodule. Applying the conclusion of Corollary \ref{cor.nonzero-or-faithful-bimodules.ii} to the trivial representation of $\Lambda\fc$, it follows that $\mu$ is of finite type. If $\mu$ is of finite type, associated with the finite-dimensional projective representation $\pi : G \to \cU(n)$, the finite index embedding $L_\mu(G) \to M_n(\C) \ot L(G) : u_g \mapsto \pi(g) \ot u_g$ defines a faithful bifinite $L_\mu(G)$-$L(G)$-bimodule, so that (iii) holds. Finally, (iii) $\Rightarrow$ (i) is trivial.
\end{proof}

\begin{theorem}\label{thm.iso-superrigidity-precise}
Let $\Gamma$ be a group in $\cC$ and write $G = (\Z/2\Z)^{(\Gamma)} \rtimes (\Gamma \times \Gamma)$. Let $\Lambda$ be any countable group and let $\mu \in Z^2(G,\T)$ and $\om \in Z^2(\Lambda,\T)$ be arbitrary scalar $2$-cocycles. Let $p \in M_n(\C) \ot L_\om(\Lambda)$ be any projection.

Then $L_\mu(G) \cong p(M_n(\C) \ot L_\om(\Lambda))p$ if and only if the following holds.
\begin{itemlist}
\item The virtual center $\Lambda\fc$ is finite.
\item $p \leq 1 \ot q$ where $q$ is a minimal projection in $\cZ(L_\om(\Lambda))$.
\item $(\Tr \ot \tau)(p) = \tau(q)/(d k)$ where $L_\om(\Lambda\fc) q \cong M_d(\C) \ot \C^k$.
\item When $z \in \cZ(L_\om(\Lambda\fc)) q$ is a minimal projection, there exists an isomorphism $\delta : \Lambda_z/\Lambda\fc \to G$ such that $\mu \circ \delta \sim \om_z$.
\end{itemlist}
In particular, $L_\mu(G) \cong L_\om(\Lambda)$ if and only if there exists an isomorphism $\delta : \Lambda \to G$ such that $\mu \circ \delta \sim \om$.
\end{theorem}
\begin{proof}
Assume that $L_\mu(G) \cong p(M_n(\C) \ot L_\om(\Lambda))p$. By Lemma \ref{lem.virtual-center-finite}, the virtual center $\Lambda\fc$ is finite. Since $L_\mu(G)$ is a factor, we must have $p \leq 1 \ot q$ where $q$ is a minimal projection in $\cZ(L_\om(\Lambda))$. So, $L_\mu(G) \cong (L_\om(\Lambda)q)^t$ with $t = (\Tr \ot \tau)(p) / \tau(q)$.

Write $L_\om(\Lambda\fc)q \cong M_d(\C) \ot \C^k$ and choose a minimal projection $z \in \cZ(L_\om(\Lambda\fc)) q$. By Proposition \ref{prop.decompose-virtual-center}, we have that $L_\om(\Lambda) q \cong L_{\om_z}(\Lambda_z / \Lambda\fc)^{dk}$. Writing $s = (t d k)^{-1}$, we find a $*$-isomorphism $\al : L_{\om_z}(\Lambda_z/\Lambda\fc) \to L_\mu(G)^s$.

As in the proof of Theorem \ref{thm.virtual-iso-superrigidity-precise}, we find that $h_G(\al(\Lambda_z/\Lambda\fc))>0$ and we can use Theorem \ref{thm.height-1} to conclude that $s=1$ and that there exists an isomorphism $\delta : \Lambda_z/\Lambda\fc \to G$ such that $\mu \circ \delta \sim \om_z$. Since the converse follows from Proposition \ref{prop.decompose-virtual-center}, this concludes the proof of the first part of the theorem.

In the special case where $L_\mu(G) \cong L_\om(\Lambda)$, we have $p=q=1$, which forces $1 = \tau(p) = \tau(q) /(dk) = 1 / (dk)$, so that $d=k=1$. This means that $\Lambda\fc = \{e\}$ and we find an isomorphism $\delta : \Lambda \to G$ such that $\mu \circ \delta \sim \om$. Here, the converse is trivial.
\end{proof}

\subsection{Examples}\label{sec.examples}

Although it is straightforward to compute the $2$-cohomology $H^2(G,\T)$ for an arbitrary wreath product $G$, we only state the following explicit version for our left-right wreath products $G = (\Z/2\Z)^{(\Gamma)} \rtimes (\Gamma \times \Gamma)$ with base of order $2$. We denote by $a_g \in G$ the generator of $(\Z/2\Z)^{(\Gamma)}$ sitting in coordinate $g$. We denote by $\Bchar(\Gamma_1,\Gamma_2)$ the group of bicharacters $\Gamma_1 \times \Gamma_2 \to \T$.

\begin{proposition}\label{prop.cohom-left-right-wreath}
Let $\Gamma$ be a group and $G = (\Z/2\Z)^{(\Gamma)} \rtimes (\Gamma \times \Gamma)$ the left-right wreath product. Denote by $S$ the group of conjugation invariant maps $s : \Gamma \to \{\pm 1\}$ satisfying $s(e)=1$. Then,
\begin{multline*}
\Theta : H^2(G,\T) \to H^2(\Gamma,\T) \times H^2(\Gamma,\T) \times \Bchar(\Gamma,\Gamma) \times \Hom(\Gamma,\{\pm 1\})  \times S :
\\ \mu \mapsto (\mu|_{\Gamma \times e},\mu|_{e \times \Gamma},\Om_\mu,\eta_\mu,s_\mu)
\end{multline*}
is an isomorphism of groups, where the components $\Om_\mu$, $\eta_\mu$ and $s_\mu$ are explicitly given by
\begin{align*}
& \Om_\mu(g,h) = \mu((g,e),(e,h)) \, \overline{\mu}((e,h),(g,e)) \;\; , \quad \eta_\mu(g) = \mu((g,g),a_e) \, \overline{\mu}(a_e,(g,g)) \quad\text{and}\\
& s_\mu(g) = \mu(a_e,a_g) \, \overline{\mu}(a_g,a_e) \quad\text{for all $g,h \in \Gamma$.}
\end{align*}
A $2$-cocycle $\mu \in H^2(G,\T)$ is of finite type if and only if $\mu|_{\Gamma \times e}$ and $\mu|_{e \times \Gamma}$ are both of finite type and there exists a finite index normal subgroup $\Gamma_0 \lhd \Gamma$ such that $\Om_\mu$ factors through $\Gamma/\Gamma_0$ and $s_\mu$ is left and right $\Gamma_0$-invariant, so that in particular $s_\mu(g) = 1$ for all $g \in \Gamma_0$.
\end{proposition}
\begin{proof}
When $G$ is any group, $\mu \in Z^2(G,\T)$ any $2$-cocycle and $\Lambda_1,\Lambda_2 < G$ commuting subgroups, the map
$$\si : \Lambda_1 \times \Lambda_2 \to \T : (a,b) \mapsto \mu(a,b) \, \overline{\mu}(b,a)$$
is a bicharacter that remains unchanged if we replace $\mu$ by a cohomologous $2$-cocycle. Moreover, if $g \in G$ normalizes both $\Lambda_1$ and $\Lambda_2$, we have that $\si(gag^{-1},gbg^{-1}) = \si(a,b)$ for all $a \in \Lambda_1$, $b \in \Lambda_2$.

Denote $\Lambda = (\Z/2\Z)^{(\Gamma)}$ and let $\Achar(\Lambda,\Lambda)$ be the group of bicharacters $\si : \Lambda \times \Lambda \to \{\pm 1\}$ satisfying $\si(a,a) = 1$ for all $a \in \Lambda$. In particular, $\si(a,b) = \si(b,a)$ for all $a,b \in \Lambda$. We consider the action of $r \in \Gamma \times \Gamma$ on $\si \in \Achar(\Lambda,\Lambda)$ by $(r \cdot \si)(a,b) = \si(r^{-1} \cdot a,r^{-1} \cdot b)$. By the previous paragraph, replacing in the definition of $\Theta$, the group $S$ at the right hand side by the group of $(\Gamma \times \Gamma)$-invariant elements $\Achar(\Lambda,\Lambda)^{\Gamma \times \Gamma}$ and the corresponding component of $\Theta$ by
$$\si_\mu : \Lambda \times \Lambda \to \{\pm 1\} : \si_\mu(a,b) = \mu(a,b) \, \overline{\mu}(b,a) \; ,$$
the modified $\Theta$ is a well-defined group homomorphism. One checks easily that
$$\Achar(\Lambda,\Lambda)^{\Gamma \times \Gamma} \to S : \si \mapsto s_\si : s_\si(g) = \si(a_e,a_g)$$
is an isomorphism of groups, so that it suffices to prove the proposition for our modified $\Theta$.

Assume that $\mu \in \Ker \Theta$. Denote by $G_\mu$ the central extension of $G$ by $\T$ given by $\mu$. Since the restriction of $\mu$ to $\Gamma \times e$ and $e \times \Gamma$ is trivial, we can choose homomorphic lifts $\phi_1 : \Gamma \times e \to G_\mu$ and $\phi_2 : e \times \Gamma \to G_\mu$. Since $\Om_\mu = 1$, we find that $\phi_1(g)$ and $\phi_2(h)$ commute for all $g,h \in \Gamma$. They thus combine into a homomorphic lift $\phi : \Gamma \times \Gamma \to G_\mu$ given by $\phi(g,h) = \phi_1(g) \phi_2(h)$. Since we can take square roots in $\T$, we can choose a lift $\phi(a_e)$ of order $2$ for the element $a_e \in G$. Since $\eta_\mu = 1$, we find that $\phi(a_e)$ commutes with $\phi(g,g)$ for all $g \in \Gamma$. We can thus unambiguously define the order $2$ elements $\phi(a_k) \in G_\mu$ satisfying
\begin{equation}\label{eq.equivariance}
\phi(a_{gkh^{-1}}) = \phi(g,h) \phi(a_k) \phi(g,h)^{-1} \quad \text{for all $g,h,k \in \Gamma$.}
\end{equation}
By construction, $\phi(a_k)$ is a lift of $a_k$ for all $k \in \Gamma$. Since $\si_\mu = 1$, the elements $\phi(a_k)$ all commute and thus together define a homomorphic lift $\phi : \Lambda \to G_\mu$. By \eqref{eq.equivariance}, the homomorphisms $\phi|_\Lambda$ and $\phi|_{\Gamma \times \Gamma}$ together define a homomorphic lift $G \to G_\mu$. It follows that $\mu \sim 1$.

Conversely, assume that we are given $\mu_1,\mu_2 \in Z^2(\Gamma,\T)$, $\Om \in \Bchar(\Gamma,\Gamma)$, $\eta \in \Hom(\Gamma,\{\pm 1\})$ and $\si \in \Achar(\Lambda,\Lambda)^{\Gamma \times \Gamma}$. First, the formula
\begin{equation}\label{eq.this-is-mu-0}
\mu_0\bigl((g,h),(g',h')\bigr) = \Om(g,h') \, \mu_1(g,g') \, \mu_2(h,h')
\end{equation}
defines a $2$-cocycle $\mu_0 \in Z^2(\Gamma \times \Gamma,\T)$ and its composition with the quotient homomorphism $q : G \to \Gamma \times \Gamma$ gives $\mu_0 \circ q \in Z^2(G,\T)$ with $\Theta(\mu_0 \circ q) = (\mu_1,\mu_2,\Om,1,1)$. It thus suffices to construct $\mu \in H^2(G,\T)$ with $\Theta(\mu) = (1,1,1,\eta,\si)$.

Define $\Lambda_\si$ as the group generated by $\T$ and elements $(b_g)_{g \in \Gamma}$ satisfying the relations
$$\T \;\;\text{is central,}\quad b_g^2=e \quad\text{and}\quad b_g b_h = \si(a_g,a_h) b_h b_g \quad\text{for all $g,h \in \Gamma$.}$$
Since $\si$ is $(\Gamma \times \Gamma)$-invariant, we can define the action $\be$ of $\Gamma \times \Gamma$ on $\Lambda_\si$ by
$$\be_{(g,h)}(z) = z \quad\text{and}\quad \be_{(g,h)}(b_k) = \eta(h) \, b_{gkh^{-1}} \quad\text{for all $z \in \T$, $g,h,k \in \Gamma$.}$$
Consider the semidirect product $\cG = \Lambda_\si \rtimes_\be (\Gamma \times \Gamma)$ and the homomorphism
\begin{equation}\label{eq.central-extension-with-theta}
\theta : \cG \to G : \theta(z) = e \;\; , \;\; \theta(b_k) = a_k \;\; , \;\; \theta(g,h) = (g,h) \quad\text{for all $z \in \T$, $g,h,k \in \Gamma$.}
\end{equation}
Then $e \to \T \to \cG \to G \to e$ is a central extension. By construction, the associated $2$-cocycle $\mu \in H^2(G,\T)$ satisfies $\Theta(\mu) = (1,1,1,\eta,\si)$. So, $\Theta$ is surjective.

Assume that $\mu \in H^2(G,\T)$ is a $2$-cocycle of finite type. Then the restrictions of $\mu$ to $\Gamma \times e$ and $e \times \Gamma$ are of finite type and, by Lemma \ref{lem.finite-type-cocycles}, the bicharacter $\Om_\mu$ factors through a finite quotient $\Gamma/\Gamma_0$ of $\Gamma$.

By Lemma \ref{lem.finite-type-cocycles}, also the bicharacter $\si_\mu \in \Achar(\Lambda,\Lambda)$ factors through a finite quotient of $\Lambda$. This means that $\Lambda_0 = \{a \in \Lambda \mid \forall b \in \Lambda : \si_\mu(a,b) = 1\}$ is a finite index subgroup of $\Lambda$. Since $(g,h) \cdot \Lambda_0 = \Lambda_0$, we find an action of $\Gamma \times \Gamma$ on the finite group $\Lambda/\Lambda_0$. Making the finite index normal subgroup $\Gamma_0 \lhd \Gamma$ smaller, we may assume that the action of $(g,h)$ on $\Lambda/\Lambda_0$ is trivial for all $g,h \in \Gamma_0$. So whenever $g,h \in \Gamma_0$ and $b \in \Lambda$, we can take $b_0 \in \Lambda_0$ such that $(g,h) \cdot b = b b_0$, implying that
$$\si_\mu(a,b) = \si_\mu(a,bb_0) = \si_\mu(a,(g,h) \cdot b) \; .$$
It follows that the map $s_\mu \in S$ given by $s_\mu(g) = \si_\mu(a_e,a_g)$ is left and right $\Gamma_0$-invariant.

Conversely, assume that $\mu \in H^2(G,\T)$ is a $2$-cocycle such that the restrictions $\mu_1 = \mu|_{\Gamma \times e}$ and $\mu_2 = \mu|_{e \times \Gamma}$ are of finite type and such that for some finite index normal subgroup $\Gamma_0 \lhd \Gamma$, we have that $\Om_\mu$ factors through $\Gamma/\Gamma_0$ and $s_\mu$ is left and right $\Gamma_0$-invariant. We have to prove that $\mu$ is of finite type. Since the kernel of $\eta_\mu$ has index at most $2$, we may assume that also $\eta_\mu(g) = 1$ for all $g \in \Gamma_0$.

Since the restriction of $\Om_\mu$ to $\Gamma_0 \times \Gamma_0$ is equal to $1$, the restriction to $\Gamma_0 \times \Gamma_0$ of the $2$-cocycle $\mu_0 \in Z^2(\Gamma \times \Gamma,\T)$ defined by \eqref{eq.this-is-mu-0} is of finite type. By Lemma \ref{lem.induce-finite-type}, $\mu_0$ is of finite type and hence, also $\mu_0 \circ q$ is of finite type, where $q : G \to \Gamma \times \Gamma$ is the quotient homomorphism defined above.

It thus suffices to prove that when $\si \in \Achar(\Lambda,\Lambda)^{\Gamma \times \Gamma}$ is $(\Gamma_0 \times \Gamma_0)$-invariant in both variables and $\eta_\mu(g) = 1$ for all $g \in \Gamma_0$, then the $2$-cocycle $\mu \in H^2(G,\T)$ associated with the central extension $e \to \T \to \cG \to G \to e$ given by \eqref{eq.central-extension-with-theta} is of finite type.

Since the map $(g,h) \mapsto \si(a_g,a_h)$ is $\Gamma_0$-invariant in both variables, we can define a group $\Lambda_{0,\si}$ generated by $\T$ and elements $(d_g)_{g \in \Gamma/\Gamma_0}$ satisfying the relations
$$\T \;\;\text{is central,}\quad d_g^2=e \quad\text{and}\quad d_g d_h = \si(a_g,a_h) d_h d_g \quad\text{for all $g,h \in \Gamma/\Gamma_0$.}$$
Define $\Lambda_0 = (\Z/2\Z)^{\Gamma/\Gamma_0}$ with canonical generators $(c_g)_{g \in \Gamma/\Gamma_0}$. Define $\theta_0 : \Lambda_{0,\si} \to \Lambda_0$ by $\theta_0(d_g) = c_g$ for all $g \in \Gamma/\Gamma_0$. Then $e \to \T \to \Lambda_{0,\si} \to \Lambda_0 \to e$ is a central extension. We get a commutative diagram
$$
\begin{array}{ccc}
\Lambda_\si \rtimes_\be (\Gamma_0 \times \Gamma_0) & \to & \Lambda \rtimes (\Gamma_0 \times \Gamma_0)  \\ \downarrow & & \downarrow \\ \Lambda_{0,\si} & \to & \Lambda_0
\end{array}
$$
and it follows that the restriction of $\mu$ to the finite index subgroup $G_0 = \Lambda \rtimes (\Gamma_0 \times \Gamma_0)$ is a $2$-cocycle that factors through the finite group $\Lambda_0$. We conclude that $\mu|_{G_0}$ is of finite type. By Lemma \ref{lem.induce-finite-type}, also $\mu$ is of finite type.
\end{proof}

Using Proposition \ref{prop.cohom-left-right-wreath}, we immediately get the following concrete examples.

\begin{example}\label{ex.generic-degree-2-extension}
Let $\Gamma$ be any group in $\cC$ and write $G = (\Z/2\Z)^{(\Gamma)} \rtimes (\Gamma \times \Gamma)$. There is a central extension
$$e \to \Z/2\Z \to \Gtil \to G \to e$$
and $\mu \in Z^2(G,\T)$ such that
\begin{itemlist}
\item $L(\Gtil) \cong L(G) \oplus L_\mu(G)$ and in particular, there is a nonzero bifinite $L_\mu(G)$-$L(\Gtil)$-bimodule~;
\item there is no \emph{faithful} bifinite $L_\mu(G)$-$L(\Lambda)$-bimodule for any countable group $\Lambda$.
\end{itemlist}
\end{example}
\begin{proof}
Define $\mu \in Z^2(G,\{\pm 1\})$ such that, in the notation of Proposition \ref{prop.cohom-left-right-wreath}, we have $\Theta(\mu) =(1,1,1,1,s)$ where $s(g) = -1$ for all $g \neq e$ and $s(e) = 1$. Denote by $\Gtil$ the corresponding central extension of degree $2$. By construction, $L(\Gtil) \cong L(G) \oplus L_\mu(G)$. By Proposition \ref{prop.cohom-left-right-wreath}, the $2$-cocycle $\mu$ is not of finite type. It then follows from Corollary \ref{cor.untwisting-2-cocycles} that there is no faithful bifinite $L_\mu(G)$-$L(\Lambda)$-bimodule for any countable group $\Lambda$.
\end{proof}

\begin{example}\label{ex.specific-F2-examples}
Denote $\Gamma = \F_2$ and $G = (\Z/2\Z)^{(\Gamma)} \rtimes (\Gamma \times \Gamma)$.
\begin{enumlist}
\item For every countable group $\cG$ and $\mu \in Z^2(\cG,\T)$ such that $\overline{\mu} \sim \mu$, we have $L_\mu(\cG) \cong L_\mu(\cG)\op$ and we have a canonical nonzero bifinite $L_\mu(\cG)$-$L(\cGtil)$-bimodule, where the countable group $\cGtil$ is a degree $2$ central extension of $\cG$.

    Nevertheless, there exists $\mu \in Z^2(G,\T)$ such that $L_\mu(G) \cong L_\mu(G)\op$ and such that there is \emph{no} nonzero bifinite $L_\mu(G)$-$L(\Lambda)$-bimodule with \emph{any} countable group $\Lambda$. The difference with the previous point lies in the fact that in these examples $\overline{\mu} \not\sim \mu$, even though we will have $\overline{\mu} \sim \mu \circ \delta$ for some $\delta \in \Aut G$.

    And nevertheless, by Example \ref{ex.generic-degree-2-extension}, there also exists $\mu \in Z^2(G,\T)$ such that $\overline{\mu} = \mu$ and such that there is \emph{no faithful} bifinite $L_\mu(G)$-$L(\Lambda)$-bimodule with \emph{any} countable group $\Lambda$.

\item There also exist $\mu \in Z^2(G,\T)$ such that there is no nonzero bifinite $L_\mu(G)$-$L_\mu(G)\op$-bimodule and such that there is no nonzero bifinite $L_\mu(G)$-$L(\Lambda)$-bimodule with any countable group $\Lambda$. Also, there is no nonzero bifinite $L_\mu(G)$-$N$-bimodule with any II$_1$ factor $N$ that is virtually isomorphic to its opposite $N\op$.
\end{enumlist}
\end{example}

\begin{proof}
Let $\cG$ be any countable group and $\mu \in Z^2(\cG,\T)$ such that $\overline{\mu} \sim \mu$. Take $\vphi_1 : \cG \to \T$ such that $\mu^2(g,h) = \vphi_1(gh) \overline{\vphi_1}(g) \overline{\vphi_1}(h)$ for all $g,h \in \cG$. Choose a function $\vphi : \cG \to \T$ such that $\vphi^2 = \vphi_1$. Replacing $\mu$ by the cohomologous $2$-cocycle $(g,h) \mapsto \mu(g,h) \overline{\vphi}(gh) \vphi(g) \vphi(h)$, we may assume that $\mu^2 = 1$, i.e.\ $\mu : G \times G \to \{\pm 1\}$. The degree $2$ central extension $\cGtil$ defined by $\mu$ now satisfies $L(\cGtil) \cong L(\cG) \oplus L_\mu(\cG)$.

To construct the examples, denote by $q_0 : \F_2 \to \Z^2$ the quotient homomorphism that corresponds to taking the abelianization and define $q : G \to \Z^2 \times \Z^2$ by composing $G \to \F_2 \times \F_2$ with $q_0 \times q_0$. Note that by Proposition \ref{prop.commensurators-wreath}, for every automorphism $\delta \in \Aut G$, there is an automorphism $\delta_1 \in \Aut(\Z^2 \times \Z^2)$ such that $\delta_1 \circ q = q \circ \delta$ and such that $\delta_1$ is symmetric, which means that either $\delta_1(a,b) = (\delta_0(a),\delta_0(b))$ or $\delta_1(a,b) = (\delta_0(b),\delta_0(a))$ for all $a,b \in \Z^2$ and $\delta_0 \in \Aut \Z^2 = \GL(2,\Z)$.

For every bicharacter $\Om : \Z^2 \times \Z^2 \to \T$, we define $\mu_0 \in Z^2(\Z^2 \times \Z^2,\T)$ by $\mu_0((a,b),(a',b')) = \Om(a,b')$ and then consider $\mu \in Z^2(G,\T)$ given by $\mu = \mu_0 \circ q$. In the notations of Proposition \ref{prop.cohom-left-right-wreath}, we have that $\Theta(\mu) = (1,1,\om,1,1)$ with $\om = \Om \circ (q_0 \times q_0)$. Note that this $\om$ factors through a finite quotient of $\F_2$ if and only if $\Om^n = 1$ for some integer $n \geq 1$. So by Proposition \ref{prop.cohom-left-right-wreath}, we have that $\mu$ is of finite type if and only if $\Om^n = 1$ for some integer $n \geq 1$.

We can now easily prove the following claims.
\begin{enumlist}
\item If for all integers $n \geq 1$, $\Om^n \neq 1$, there is no nonzero bifinite $L_\mu(G)$-$L(\Lambda)$-bimodule with any countable group $\Lambda$.
\item If $\Om(a,b) = \Om(b,a)$ for all $a,b \in \Z^2$ and if $\zeta \in \Aut G$ is the automorphism satisfying $\zeta(g,h) = (h,g)$ for all $(g,h) \in \Gamma \times \Gamma < G$, then $\overline{\mu} \circ \zeta \sim \mu$, so that $L_\mu(G) \cong L_\mu(G)\op$.
\item Given a bicharacter $\Om : \Z^2 \times \Z^2 \to \T$, write $\Oms(a,b) = \Om(b,a)$. If for every integer $n \geq 1$ and every $\delta_0 \in \Aut \Z^2$, we have
\begin{equation}\label{eq.this-is-what-we-need}
\Om^n \neq {\overline{\Om}}^n \circ (\delta_0 \times \delta_0) \quad\text{and}\quad \Om^n \neq \Oms^n \circ (\delta_0 \times \delta_0) \; ,
\end{equation}
then there is no nonzero bifinite $L_\mu(G)$-$L_\mu(G)\op$-bimodule and there is no nonzero bifinite $L_\mu(G)$-$L(\Lambda)$-bimodule with any countable group $\Lambda$. In particular, there is no nonzero bifinite $L_\mu(G)$-$N$-bimodule with any II$_1$ factor $N$ that is virtually isomorphic to its opposite $N\op$.
\end{enumlist}

To prove (i), assume that $\Om^n \neq 1$ for all integers $n \geq 1$. By the discussion above, no power $\mu^k$ is of finite type. The conclusion then follows from the first part of Corollary \ref{cor.untwisting-2-cocycles}.

To prove (ii) and (iii), denote for every bicharacter $\Om : \Z^2 \times \Z^2 \to \T$ the associated $2$-cocycle on $G$ as $\mu_\Om$. Taking the coboundary of the map $\vphi : G \to \T : \vphi(g) = \Om(q(g))$, we get that
$$\mu_\Om \circ \zeta \sim \mu_{\Omtil} \quad\text{where}\quad \Omtil(a,b) = \overline{\Om}(b,a) \; .$$
From this, point (ii) follows.

Given a bicharacter satisfying the condition in point~(iii), there is no automorphism $\delta \in \Aut G$ such that $\mu \, (\mu \circ \delta)$ is of finite type. By Corollary \ref{cor.nonzero-or-faithful-bimodules.i} and Proposition \ref{prop.commensurators-wreath}, there is no nonzero (equivalently, faithful) bifinite $L_\mu(G)$-$L_\mu(G)\op$-bimodule. Since also $\Om^n \neq 1$, it follows from point~(i) that there is no nonzero bifinite $L_\mu(G)$-$L(\Lambda)$-bimodule with any countable group $\Lambda$. Also if $N$ is a II$_1$ factor, any nonzero bifinite $L_\mu(G)$-$N$-bimodule gives a virtual isomorphism between $L_\mu(G)$ and $N$, and also between $N\op$ and $L_\mu(G)\op$. So if $N$ and $N\op$ are virtually isomorphic, it follows that $L_\mu(G)$ is virtually isomorphic to $L_\mu(G)\op$, which is not the case.

To conclude the proof Example \ref{ex.specific-F2-examples}, it thus suffices to produce bicharacters on $\Z^2 \times \Z^2$ satisfying the appropriate conditions from (i), (ii) and (iii).

Whenever $\al \in \R \setminus \Q$, we can take $\Om(a,b) = \exp(2\pi i \al a_1 b_1)$, we have that $\Om^n \neq 1$ for all integers $n \geq 1$ and $\Om(a,b) = \Om(b,a)$ for all $a,b \in \Z^2$, providing the first example.

Viewing elements $a \in \Z^2$ as $1 \times 2$ matrices with entries in $\Z$, we can define for every $X \in \R^{2 \times 2}$ the bicharacter $\Om$ by $\Om(a,b) = \exp(2 \pi i a X b^\top)$. Then \eqref{eq.this-is-what-we-need} is equivalent to the requirement that for every $P \in \GL(2,\Z)$, we have
$$P^{-1} X + X P^\top \not\in \Q^{2 \times 2} \quad\text{and}\quad P^{-1} X - X^\top P^\top \not\in \Q^{2 \times 2} \; .$$
Choosing $\al,\be \in \R$ such that $\{\al,\be,1\}$ are linearly independent over $\Q$, the matrix $X = \bigl(\begin{smallmatrix} \al & \be \\ 0 & 0\end{smallmatrix}\bigr)$ satisfies these conditions. So the associated $\Om$ satisfies \eqref{eq.this-is-what-we-need}, providing the second example.
\end{proof}

\section{Proof of Theorem \ref{thm.B}}

We deduce from \cite{IPP05,Ioa12} the following ad hoc lemma.

\begin{lemma}\label{lem.conjugacy-to-M}
Let $\Gamma$ be a nonamenable group and $(A_0,\tau_0)$ a tracial von Neumann algebra. Denote by $M = (A_0,\tau_0)^\Gamma \rtimes (\Gamma \times \Gamma)$ the left-right Bernoulli crossed product.

Let $P$ and $Q$ be II$_1$ factors and $\Lambda$ a weakly amenable, biexact group. Write $\cN = P \ovt (Q \ast L(\Lambda))$.

If $\psi : M \to \cN$ is an embedding such that the bimodule $\bim{\psi(M)}{L^2(\cN)}{P \ot 1}$ is coarse, then $\psi(M)$ can be unitarily conjugated into $P \ovt Q$.
\end{lemma}
\begin{proof}
Write $S_1 = \psi(L(\Gamma \times e))$ and $S_2 = \psi(L(e \times \Gamma))$. Also write $S = S_1 \vee S_2 = \psi(L(\Gamma \times \Gamma))$. By the coarseness assumption and Lemma \ref{lem.coarse-gives-nonintertwining.ii}, we have that $S_1$ and $S_2$ are strongly nonamenable relative to $P \ot 1$. By Corollary \ref{cor.relative-omega-solidity}, $S \not\prec_\cN P \ovt L(\Lambda)$. It then follows from \cite[Theorem 6.4]{Ioa12} that $S \prec_\cN^f P \ovt Q$.

By \cite[Theorem 1.1]{IPP05}, if $S_0 \subset p (M_n(\C) \ovt P \ovt Q)p$ is a von Neumann subalgebra such that $S_0 \not\prec_{P \ovt Q} P \ot 1$, then
$$S_0' \cap p (M_n(\C) \ot \cN) p \subset p (M_n(\C) \ovt P \ovt Q)p \; .$$
In combination with Lemma \ref{lem.intertwine-transitivity.i} and the facts that $P \ovt Q$ is a factor and $S \prec_\cN^f P \ovt Q$ and $S \not\prec_\cN P \ot 1$, it follows that $S$ can be unitarily conjugated into $P \ovt Q$. So we may assume that $S \subset P \ovt Q$.

Denote by $S_0 \subset S$ the von Neumann algebra generated by $\psi(u_{(g,g)})$, $g \in \Gamma$. Again by the coarseness assumption on $\psi$ and Lemma \ref{lem.coarse-gives-nonintertwining.i}, we have that $S_0 \not\prec_\cN P \ot 1$. Considering the relative commutant $S_0' \cap \cN$ and again using \cite[Theorem 1.1]{IPP05} in the way explained above, it follows that $\psi(\pi_e(A_0)) \subset P \ovt Q$. Since $\psi(\pi_e(A_0))$ and $S$ generate $\psi(M)$, we have proven that $\psi(M) \subset P \ovt Q$.
\end{proof}

\begin{theorem}\label{thm.no-coarse-embedding}
Let $N$ be defined as in Theorem \ref{thm.B} and assume that $\tau_0$ is not uniform. There is no coarse embedding $N \to (N \ovt N\op \ovt N)^t$ for any $t > 0$.
\end{theorem}
\begin{proof}
Define $M = (A_0,\tau_0)^\Gamma \rtimes (\Gamma \times \Gamma)$ as in Theorem \ref{thm.B}, so that $N = M \ast L(\Lambda)$. Assume that $\Psi : N \to (N \ovt N\op \ovt N)^t$ is a coarse embedding. Since $A_0$ is abelian, we canonically identify $N\op = N$.

Applying three times Lemma \ref{lem.conjugacy-to-M}, after a unitary conjugacy, we have $\Psi(M) \subset (M \ovt M \ovt M)^t$. We denote by $\psi$ the restriction to $M$ of this unitary conjugacy of $\Psi$, so that $\psi$ is a coarse embedding of $M$ into $(M \ovt M \ovt M)^t$. We now apply Theorem \ref{thm.coarse-embedding-generalized-Bernoulli} to $\psi$ and find a direct summand $\psi_0 : M \to M_n(\C) \ovt M \ovt M \ovt M$ of the special form described there.

In particular, we find an abelian $*$-algebra $D \subset M_n(\C)$ and a unital $*$-homomorphism $\psi_1 : A_0 \to D \ot A_0 \ot A_0 \ot A_0$ such that
$$\psi_0(\pi_e(b)) = (\id \ot \pi_e \ot \pi_e \ot \pi_e)\psi_1(b) \quad\text{and}\quad (\id \ot \tau_0 \ot \id \ot \id)\psi_1(b) = \tau_0(b) 1$$
for all $b \in A_0$.

Denote by $(p_i)_{i \in I}$ the minimal projections of $A_0$. Since $\tau_0$ is not uniform, we can take $0 < s < 1$ such that the set $J = \{i \in I \mid \tau_0(p_i)=s\}$ satisfies $\emptyset \neq J \neq I$. Define $p_0 = \sum_{i \in J} p_i$ and note that $0 < p_0 < 1$.

Fix arbitrary minimal projections $p \in D$ and $q,q' \in A_0$. Since $D$ and $A_0$ are abelian, we find the unital $*$-homomorphism $\psi_2 : A_0 \to A_0$ such that
\begin{equation}\label{eq.psi-in-every-component}
\psi_1(a)(p \ot 1 \ot q \ot q') = p \ot \psi_2(a) \ot q \ot q' \quad\text{and}\quad \tau_0(\psi_2(a)) = \tau_0(a)
\end{equation}
for all $a \in A_0$. In particular, $\psi_2$ is faithful and the projections $(\psi_2(p_i))_{i \in J}$ must be a permutation of the projections $(p_i)_{i \in J}$. It follows that $\psi_2(p_0) = p_0$. Since this holds for all $p,q,q'$, we have proven that $\psi_1(p_0) = 1 \ot p_0 \ot 1 \ot 1$.

Define the diffuse von Neumann subalgebra $\cA \subset (A_0,\tau_0)^\Gamma$ generated by $\pi_g(p_0)$, $g \in \Gamma$. Looking at the special form of $\psi_0$ in Theorem \ref{thm.coarse-embedding-generalized-Bernoulli}, it follows that $\psi_0(a) = 1 \ot a \ot 1 \ot 1$ for all $a \in \cA$. By Lemma \ref{lem.coarse-gives-nonintertwining.i}, this contradicts the coarseness of $\psi_0$.
\end{proof}

We are now ready to prove Theorem \ref{thm.B}.

\begin{proof}[{Proof of Theorem \ref{thm.B}}]
First assume that $\tau_0$ is not uniform and define $N = M \ast L(\Lambda)$ as in the formulation of the theorem. Let $\cG$ be a discrete pmp groupoid with space of units $(X,\mu)$. Let $\om \in Z^2(\cG,\T)$ and write $P = L_\om(\cG)$. Let $\bim{N}{K}{P}$ be a nonzero bifinite bimodule. Denote by $z \in \cZ(P)$ the support projection of the right $P$-action on $K$.

Write $B = L^\infty(X,\mu)$ and denote by $z_1 \in B$ the projection that corresponds to the atomic part of $(X,\mu)$. Note that $z_1 \in \cZ(P)$. We claim that $z \leq z_1$. Assume the contrary and write $p = z (1-z_1)$. Then $p$ is a nonzero projection in $\cZ(P)$, $K p \neq \{0\}$ and $B p$ is diffuse abelian.

As explained in Section \ref{sec.twisted-groupoid-vNalg}, $B \subset P$ is regular. Since $p \in \cZ(P)$, also $Bp \subset Pp$ is regular. The nonzero bifinite bimodule $\bim{N}{(Kp)}{Pp}$ defines a finite index embedding $\vphi : P p \to N^t$ for some $t > 0$. It follows that $\vphi(Bp) \subset (M \ast L(\Lambda))^t$ is a diffuse abelian von Neumann subalgebra whose normalizer contains $\vphi(Pp)$ and thus has finite index. This contradicts \cite[Corollary 9.1]{Ioa12}. So the claim is proven.

By the claim, $K z_1 \neq \{0\}$. Since $B z_1$ is atomic, we can choose a minimal projection $q \in B$ such that $K q \neq \{0\}$. As explained in Section \ref{sec.twisted-groupoid-vNalg}, $q P q \cong L_\om(\cG_q)$, where $\cG_q$ is the isotropy group of the atom in the unit space that corresponds to $q$. Since $Kq \neq \{0\}$, we have found a nonzero bifinite $N$-$L_\om(\cG_q)$-bimodule.

Repeating the first paragraphs of the proof of Lemma \ref{lem.virtual-center-finite}, we find a coarse embedding $N \to (N \ovt N\op \ovt N)^t$ for some $t > 0$. This contradicts Theorem \ref{thm.no-coarse-embedding}.

If $\tau_0$ is uniform, then $A_0$ is finite dimensional and with $n = \dim A_0$, we find a trace preserving isomorphism $(A_0,\tau_0) \cong (L(\Z/n\Z),\tau)$. It then follows that $N \cong L(\cG)$ with $\cG$ the countable group as stated in the theorem.
\end{proof}

\section{Proof of Theorem \ref{thm.C}}

We prove the following theorem, which is more general than Theorem \ref{thm.C}. We again use the notation of Proposition \ref{prop.decompose-virtual-center}. Recall from \eqref{eq.diagonal-2-cocycle} the definition of the diagonal $2$-cocycle $\om_0^\Gamma$ on a wreath product group.

\begin{theorem}
Let $\Gamma$ be a group in $\cC$ and $(A_0,\tau_0)$ any nontrivial amenable tracial von Neumann algebra. Denote by $M = (A_0,\tau_0)^\Gamma \rtimes (\Gamma \times \Gamma)$ the left-right Bernoulli crossed product. Let $\Lambda$ be any countable group and $\om \in Z^2(\Lambda,\T)$ any $2$-cocycle. Let $p \in M_n(\C) \ot L_\om(\Lambda)$ be a projection.

Then $M \cong p(M_n(\C) \ot L_\om(\Lambda))p$ if and only if the following holds.
\begin{itemlist}
\item The virtual center $\Lambda\fc$ is finite, $p \leq 1 \ot q$ where $q$ is a minimal projection in $\cZ(L_\om(\Lambda))$ and $(\Tr \ot \tau)(p) = \tau(q)/(d k)$ where $L_\om(\Lambda\fc) q \cong M_d(\C) \ot \C^k$.
\item Fix a minimal projection $z \in \cZ(L_\om(\Lambda\fc)) q$. Then there exists a countable group $\Lambda_0$, an isomorphism $\delta : \Lambda_z/\Lambda\fc \to \Lambda_0^{(\Gamma)} \rtimes (\Gamma \times \Gamma)$ and a $2$-cocycle $\om_0 \in Z^2(\Lambda_0,\T)$ such that $\om_0^\Gamma \circ \delta \sim \om_z$ and $(A_0,\tau_0) \cong (L_{\om_0}(\Lambda_0),\tau_0)$.
\end{itemlist}
\end{theorem}

\begin{proof}
The easy constructive implication of the theorem follows from Proposition \ref{prop.decompose-virtual-center}. So assume that $\Lambda$ is a countable group, $\om \in Z^2(\Lambda,\T)$ a $2$-cocycle, $p \in M_n(\C) \ot L_\om(\Lambda)$ a projection and $\al : M \to p(M_n(\C) \ot L_\om(\Lambda))p$ a $*$-isomorphism.

We write $N = L_\om(\Lambda)$. Since $M$ is a factor, there is a unique minimal projection $q \in \cZ(N)$ such that $p \leq 1 \ot q$. We identify $p(M_n(\C) \ot L_\om(\Lambda))p = (Nq)^t$ where $t = (\Tr \ot \tau)(p)/\tau(q)$. Denote by $\Delta_3 : N \to N \ovt N\op \ovt N$ the triple comultiplication given by Proposition \ref{prop.triple-comult}. Write $p_1 = (q \ot q\op \ot q)\Delta_3(q)$. By Proposition \ref{prop.triple-comult}, $p_1 \neq 0$ and
$$\Delta_q : Nq \to p_1(Nq \ovt (Nq)\op \ovt Nq)p_1 : \Delta_q(a) = \Delta_3(a)(q \ot q\op \ot q)$$
is a coarse embedding. Viewing $p_1(Nq \ovt (Nq)\op \ovt Nq)p_1 = (Nq \ovt (Nq)\op \ovt Nq)^s$ for the appropriate value of $s > 0$, we amplify $\Delta_q$ to a coarse embedding
$$\Phi : (Nq)^t \to (Nq \ovt (Nq)\op \ovt Nq)^{ts} \; ,$$
so that
$$\Psi : M \to (M \ovt M\op \ovt M)^{st^{-2}} : \Psi = (\al \ot \al\op \ot \al)^{-1} \circ \Phi \circ \al$$
is a well-defined coarse embedding. By Theorem \ref{thm.coarse-embedding-generalized-Bernoulli}, the relative commutant of $\Psi(M)$ is finite dimensional. So also the relative commutant of $\Delta_q(Nq)$ is finite dimensional. By Proposition \ref{prop.triple-comult}, the virtual center $\Lambda\fc$ is finite.

Choosing a minimal projection $z \in \cZ(L_\om(\Lambda\fc)) q$, we know from Proposition \ref{prop.decompose-virtual-center} that $N q \cong M_{dk}(\C) \ot L_{\om_z}(\Lambda_z/\Lambda\fc)$, where the integers $d$ and $k$ are such that $L_\om(\Lambda\fc) q \cong M_d(\C) \ot \C^k$.

Since $\Lambda_z/\Lambda\fc$ is an icc group, we may thus assume from the start that $\Lambda$ is icc and that $\al : M \to L_\om(\Lambda)^t$ is a $*$-isomorphism. We have to prove that $t=1$ and that there exists a countable group $\Lambda_0$, an isomorphism $\delta : \Lambda \to \Lambda_0^{(\Gamma)} \rtimes (\Gamma \times \Gamma)$ and a $2$-cocycle $\om_0 \in Z^2(\Lambda_0,\T)$ such that $\om_0^\Gamma \circ \delta \sim \om$ and $(A_0,\tau_0) \cong (L_{\om_0}(\Lambda_0),\tau_0)$.

We still write $N = L_\om(\Lambda)$ and denote by $\Phi : N^t \to (N \ovt N\op \ovt N)^t$ the amplification of the triple comultiplication $\Delta_3$. We consider the coarse embedding
$$\Psi : M \to (M \ovt M\op \ovt M)^{t^{-2}} : \Psi = (\al \ot \al\op \ot \al)^{-1} \circ \Phi \circ \al \; .$$
Theorem \ref{thm.coarse-embedding-generalized-Bernoulli} provides a very precise description of how $\Psi$ looks like. By Proposition \ref{prop.properties-triple-comult.i}, $\Psi(M)$ has trivial relative commutant. So in the description of Theorem \ref{thm.coarse-embedding-generalized-Bernoulli}, there is only one direct summand, $t^{-2}$ is an integer $n \in \N$ and $\Psi$ can be unitarily conjugated to a coarse embedding $\psi : M \to M_n(\C) \ovt M \ovt M\op \ovt M$ of the form given in Theorem \ref{thm.coarse-embedding-generalized-Bernoulli}.

Write $A = (A_0,\tau_0)^\Gamma$. We can choose a finite index subgroup $\Gamma_0 < \Gamma$ so that $\pi(g,h)$ commutes with $D$ for all $(g,h) \in \Gamma_0 \times \Gamma_0$. It follows that $D \ot 1 \ot 1 \ot 1$ commutes with $\psi(A \rtimes (\Gamma_0 \times \Gamma_0))$. Since $A \rtimes (\Gamma_0 \times \Gamma_0)$ is an irreducible finite index subfactor of $M$, it follows from Proposition \ref{prop.properties-triple-comult.i} that the relative commutant of $\psi(A \rtimes (\Gamma_0 \times \Gamma_0))$ is trivial. So, $D = \C 1$.

Assume that $n \geq 1$ and denote by $W \subset M_n(\C)$ the matrices of trace zero. It follows that $(W \ot 1 \ot 1 \ot 1)\psi(L^2(M))$ defines a $\psi(M)$-$\psi(M)$-bimodule that is finitely generated as a right Hilbert $\psi(M)$-module and that is orthogonal to $\psi(L^2(M))$. This contradicts Proposition \ref{prop.properties-triple-comult.i}. So $n=1$, which means that $t=1$.

We now have that $\psi : M \to M \ovt M\op \ovt M$ is unitarily conjugate to $(\al \ot \al\op \ot \al)^{-1} \circ \Delta_3 \circ \al$, where $\Delta_3 : N \to N \ovt N\op \ovt N$ is the triple comultiplication. We also have that $\psi(u_r) = u_{\delta_1(r)} \ot \overline{u_{\delta_2(r)}} \ot u_{\delta_3(r)}$ for all $r \in \Gamma \times \Gamma$, where $\delta_1,\delta_2,\delta_3$ are symmetric automorphisms of $\Gamma \times \Gamma$.

In particular, we have a unitary $W \in N \ovt N\op \ovt N$ such that
$$(\al(u_{\delta_1(r)}) \ot \overline{\al(u_{\delta_2(r)})} \ot \al(u_{\delta_3(r)})) W = W \Delta_3(\al(u_r)) \quad\text{for all $r \in \Gamma \times \Gamma$.}$$
The same computation as the one to prove \eqref{eq.claim-to-control-height} now gives that $h_\Lambda(\al(\Gamma \times \Gamma))>0$. Since $(\Ad u_r)_{r \in \Gamma \times \Gamma}$ is weakly mixing on $L^2(M \ominus \C 1)$, also $(\Ad \al(u_r))_{r \in \Gamma \times \Gamma}$ is weakly mixing on $L^2(N \ominus \C 1)$.

{\bf Claim.} For every $k \in \Lambda \setminus \{e\}$, we have that $\al(L(\Gamma \times \Gamma)) \not\prec_N L_\om(C_\Lambda(k))$.

Once this claim is proven, all hypotheses of Theorem \ref{thm.height-1} are satisfied. So after replacing $\al$ by $(\Ad w) \circ \al$ for some unitary $w \in \cU(N)$, we then have that $\al(u_r) = \eta(r) v_{\delta(r)}$ for all $r \in \Gamma \times \Gamma$, where $\delta : \Gamma \times \Gamma \to \Lambda$ is an injective group homomorphism and $\eta : \Gamma \times \Gamma \to \T$. Here and later in the proof, we denote by $(v_h)_{h \in \Lambda}$ the canonical unitaries in $N = L_\om(\Lambda)$.

To prove the claim, assume the contrary. Writing $w_r = \al(u_r)$ for all $r \in \Gamma \times \Gamma$, we find a projection $p \in M_n(\C) \ot L_\om(C_\Lambda(k))$, a nonzero $X \in p(\C^n \ot N)$ and a group homomorphism $\theta : \Gamma \times \Gamma \to \cU(p(M_n(\C) \ot L_\om(C_\Lambda(k)))p)$ satisfying $X w_r = \theta(r)X$ for all $r \in \Gamma \times \Gamma$.

Using the same notation $(v_h)_{h \in \Lambda}$ for the canonical unitaries in $L(\Lambda)$, we define $\zeta : N \to N \ovt L(\Lambda) : \zeta(v_h) = v_h \ot v_h$ for all $h \in \Lambda$. Note that together with $\Delta : L(\Lambda) \to N\op \ovt N : \Delta(v_h) = \overline{v_h} \ot v_h$, we get that $\Delta_3 = (\id \ot \Delta) \circ \zeta$. We write
$$Y := (\id \ot \zeta)(X) \in \C^n \ovt N \ovt L(\Lambda) \quad\text{and}\quad Z := Y^* (1 \ot 1 \ot v_k) Y \in N \ovt L(\Lambda) \; .$$
Since $(\id \ot \zeta)\theta(r) \in M_n(\C) \ovt N \ovt L(C_\Lambda(k))$ for all $r \in \Gamma \times \Gamma$, we see that $Z$ commutes with $\zeta(w_r)$ for all $r \in \Gamma \times \Gamma$.

Since $\Delta_3(w_r)$ is unitarily conjugate to $\al(u_{\delta_1(r)}) \ot \overline{\al(u_{\delta_2(r)})} \ot \al(u_{\delta_3(r)})$ and since
$$(u_{\delta_1(r)} \ot \overline{u_{\delta_2(r)}} \ot u_{\delta_3(r)})_{r \in \Lambda \times \Lambda}$$
has trivial relative commutant in $M \ovt M\op \ovt M$, also the relative commutant of $(\Delta_3(w_r))_{r \in \Gamma \times \Gamma}$ is trivial. Since $\Delta_3(w_r) = (\id \ot \Delta)\zeta(w_r)$, a fortiori, the relative commutant of $(\zeta(w_r))_{r \in \Gamma \times \Gamma}$ inside $N \ovt L(\Lambda)$ is trivial. This means that $Z \in \C 1$.

Writing $X = \sum_{i=1}^n \sum_{h \in \Lambda} (X)_{i,h} (e_i \ot v_h)$, we have that
$$Z = \sum_{i=1}^n \sum_{g,h \in \Lambda} \overline{(X)_{i,g}} \, (X)_{i,h} \, (v_g^* v_h \ot v_{g^{-1} k h}) \; .$$
It follows that
\begin{equation}\label{eq.Z-with-positive-terms}
(\tau \ot \id)(Z) = \sum_{i=1}^n \sum_{h \in \Lambda} |(X)_{i,h}|^2 \, v_{h^{-1}kh} \; .
\end{equation}
Since $k \neq e$, we get that $(\tau \ot \tau)(Z) = 0$. Since $Z \in \C 1$, it then follows that $(\tau \ot \id)(Z) = 0$. But now \eqref{eq.Z-with-positive-terms} implies that $(X)_{i,h} = 0$ for all $i$ and $h$. We arrive at the contradiction that $X = 0$. So the claim is proven.

As explained above, we may now assume that $\al(u_r) = \eta(r) v_{\delta(r)}$ for all $r \in \Gamma \times \Gamma$, where $\delta : \Gamma \times \Gamma \to \Lambda$ is an injective group homomorphism and $\eta : \Gamma \times \Gamma \to \T$. Define the subgroup $\Lambda_0 < \Lambda$ consisting of those elements $h \in \Lambda$ for which $\{\delta(g,g) h \delta(g,g)^{-1} \mid g \in \Gamma\}$ is finite.

Because $(\Ad u_{(g,g)})_{g \in \Gamma}$ is weakly mixing on $L^2(M \ominus \pi_e(A_0))$, we have $\al^{-1}(v_h) \in \pi_e(A_0)$ for all $h \in \Lambda_0$. Since by definition, $(\Ad v_{\delta(g,g)})_{g \in \Gamma}$ is weakly mixing on $L^2(L_\om(\Lambda) \ominus L_\om(\Lambda_0))$ and $\al(\pi_e(A_0))$ commutes with $v_{\delta(g,g)}$ for all $g \in \Gamma$, also $\al(\pi_e(A_0)) \subset L_\om(\Lambda_0)$. Denoting by $\om_0$ the restriction of $\om$ to $\Lambda_0$, we have thus found a trace preserving isomorphism $\al_0 : (A_0,\tau_0) \to (L_{\om_0}(\Lambda_0),\tau)$ such that $\al \circ \pi_e = \al_0$.

Denote $w_s = \al_0^{-1}(v_s)$ for all $s \in \Lambda_0$. Since the subalgebras $\al(\pi_g(A_0))$ commute, we can unambiguously define the projective representation
$$\theta : \Lambda_0^{(\Gamma)} \to \cU(N) : \theta(s) = \prod_{g \in \Gamma} \al(\pi_g(w_{s_g}))$$
whose $2$-cocycle $\om_\theta$ equals $\om_0^\Gamma$. We extend $\theta$ to a projective representation $\Lambda_0^{(\Gamma)} \rtimes (\Gamma \times \Gamma)$ by putting $\theta(r) = \al(u_r)$ for all $r \in \Gamma \times \Gamma$. The associated $2$-cocycle is still given by $\om_0^\Gamma$.

By construction, $\theta(r) \in \T \cdot \Lambda \subset \cU(N)$ for all $r \in \Lambda_0^{(\Gamma)} \rtimes (\Gamma \times \Gamma)$ and $\tau(\theta(r)) = 0$ when $r \neq e$. Also by construction, the image of $\theta$ generates $N$. So, $\theta(r) = \eta(r) v_{\delta(r)}$ for all $r \in \Lambda_0^{(\Gamma)} \rtimes (\Gamma \times \Gamma)$, where $\delta : \Lambda_0^{(\Gamma)} \rtimes (\Gamma \times \Gamma) \to \Lambda$ is an isomorphism of groups and $\eta$ is the map that realizes $\om \circ \delta \sim \om_0^\Gamma$.
\end{proof}

To state our final result, we first give the following definition, as stated in the introduction.

\begin{definition}\label{def.cocycle-Wstar-superrigid}
We say that a pair $(G,\mu)$ of a countable group $G$ and a $2$-cocycle $\mu \in Z^2(G,\T)$ is \emph{cocycle W$^*$-superrigid} if $L_\mu(G) \cong L_\om(\Lambda)$ with $\Lambda$ an arbitrary countable group and $\om \in Z^2(\Lambda,\T)$ an arbitrary $2$-cocycle implies that $(G,\mu) \cong (\Lambda,\om)$, i.e.\ there exists a group isomorphism $\delta : \Lambda \to G$ with $\mu \circ \delta \sim \om$.
\end{definition}

\begin{corollary}\label{cor.further-superrigidity}
If $C$ is a nontrivial finite group and $\mu_0 \in Z^2(C,\T)$ such that $(C,\mu_0)$ is cocycle W$^*$-superrigid, then for every $\Gamma$ in $\cC$ the group $G = C^{(\Gamma)} \rtimes (\Gamma \times \Gamma)$ together with the $2$-cocycle $\mu = \mu_0^\Gamma \in Z^2(G,\T)$ are cocycle W$^*$-superrigid.

This applies in the following cases:
\begin{enumlist}
\item $|C| = p$ with $p$ prime,
\item $|C| = pq$ with $p,q$ distinct primes,
\item $|C| = p^2$ with $p$ prime and $\mu_0\not\sim 1$.
\end{enumlist}
\end{corollary}

The proof will be obviously ad hoc and concrete, because we know all groups of the orders described in the statement and can easily describe their $H^2(C,\T)$, which will always be trivial in the first two cases and nontrivial in the third case iff $C = \Z/p\Z \times \Z/p\Z$. These basic results can for instance be found in \cite[Cor.\ 4.5, Thm.\ 4.20 and Cor.\ 7.70]{Rot95}.

\begin{proof}
The generic statement in the corollary is an immediate consequence of Theorem \ref{thm.C}.

Assume that $C$ is a finite group, $\mu_0 \in Z^2(C,\T)$ and that $L_{\mu_0}(C) \cong L_{\om_0}(\Lambda_0)$ for some countable group $\Lambda_0$ and $\om_0 \in Z^2(\Lambda_0,\T)$. By looking at the dimensions of these algebras, we get that $|\Lambda_0| = |C|$. We now distinguish the three cases.

(i)\ Since $|C|=p$ with $p$ prime, we first conclude that $C \cong \Z/p\Z \cong \Lambda_0$ and then note that this forces $\mu_0 \sim 1$ and $\om_0 \sim 1$ because $H^2(\Z/p\Z,\T)$ is trivial.

(ii)\ By symmetry, we may assume that $p<q$. If $p$ does not divide $q-1$, there is only one group of order $pq$ and that is the cyclic group of order $pq$. Since $H^2(\Z/n\Z,\T)$ is trivial for all integers $n \geq 1$, the same reasoning as in the previous case gives the conclusion.

Next assume that $p$ divides $q-1$. There are exactly two nonisomorphic groups of order $pq$: the cyclic group $\Z/(pq)\Z$ and the nonabelian group $C = \Z/q\Z \rtimes_\al \Z/p\Z$ where $\al$ is a period $p$ automorphisms of $\Z/q\Z$. In both cases, $H^2(C,\T)$ is trivial. The isomorphism $L_{\mu_0}(C) \cong L_{\om_0}(\Lambda_0)$ forces $\Lambda_0$ to be abelian, resp.\ nonabelian. So, $\Lambda_0 \cong C$ and also $\om_0 \sim 1$.

(iii)\ There are two nonisomorphic groups of order $p^2$: the cyclic one and $\Z/p\Z \times \Z/p\Z$. Since we assume that $\mu_0 \not\sim 1$, we get that $C$ is not cyclic and that $L_{\mu_0}(C)$ is not abelian. Since $|\Lambda_0| = p^2$, also $\Lambda_0$ is either cyclic or isomorphic to $C$. If $\Lambda_0$ would be cyclic, we have that $\om_0 \sim 1$ and $L_{\om_0}(\Lambda_0)$ abelian, which is impossible because $L_{\mu_0}(C)$ is not abelian. So $\Lambda_0 \cong C$. Because $L_{\om_0}(\Lambda_0)$ must be nonabelian, we have that $\om_0 \not\sim 1$.

To conclude the proof, note that with $C = \Z/p\Z \times \Z/p\Z$, we have $H^2(C,\T) \cong \Z/p\Z$ and the action of $\Aut(C)$ on the nontrivial elements of $H^2(C,\T)$ is transitive. It thus follows that $(\Lambda_0,\om_0) \cong (C,\mu_0)$.
\end{proof}

\end{document}